\documentclass[11pt,a4paper,reqno]{article}
\pdfoutput=1
\usepackage{amsfonts}
\usepackage{amsthm}
\usepackage{amsmath}
\usepackage{amscd}
\usepackage[latin2]{inputenc}
\usepackage{t1enc}
\usepackage[mathscr]{eucal}
\usepackage{indentfirst}
\usepackage{graphicx}
\usepackage{graphics}
\usepackage{pict2e}
\usepackage{epic}
\usepackage{verbatim}
\usepackage{stmaryrd}
\usepackage{systeme}
\usepackage{amssymb}
\usepackage{mathtools}
\usepackage{subcaption}
\usepackage{textcomp}
\usepackage{bbm}
\usepackage{bm}
\usepackage{sectsty}
\usepackage[title]{appendix}
\usepackage{enumitem}

\numberwithin{equation}{section}
\usepackage[margin=2.0cm]{geometry}
\usepackage{epstopdf} 
\usepackage[flushleft]{threeparttable}

\newcommand{\Err}{\text{Err}}

\theoremstyle{plain}
\newtheorem{Th}{Theorem}[section]
\newtheorem{Lemma}[Th]{Lemma}
\theoremstyle{definition}

\author{Yat Tin Chow \thanks{Department of Mathematics, University of California, Riverside. Research supported by Omnibus Research and Travel Award 2019, University of California, Riverside. ({ytchow@ucr.edu}).} 
\and Fuqun Han \thanks{Department of Mathematics, The Chinese University of Hong Kong, Shatin, N.T., Hong Kong.  ({fqhan@math.cuhk.edu.hk}).}
\and Jun Zou \thanks{Department of Mathematics, The Chinese University of Hong Kong, Shatin, N.T., Hong Kong. The work of this author was substantially supported by Hong Kong RGC General Research Fund 
(projects 14322516 and 14306718). ({zou@math.cuhk.edu.hk}).}}
\begin{document}
\title{A direct sampling method for the inversion of the Radon transform}
\date{}
\maketitle

\abstract{
We propose a novel direct sampling method (DSM) for the effective and stable inversion of the Radon transform.  
The DSM is based on a generalization of the important almost orthogonality property in classical DSMs to fractional order Sobolev duality products and to a new family of probing functions. The fractional order duality product proves to be able to greatly enhance the robustness of the reconstructions in some practically important but severely ill-posed inverse problems associated with the Radon transform. We present a detailed analysis to better understand the performance of the new probing and index functions, which are crucial to stable and effective numerical reconstructions. The DSM can be computed in a very fast and highly parallel manner. Numerical experiments are carried out to compare the DSM with a popular existing method, and to illustrate the efficiency, stability, and accuracy of the DSM.}



\section{Introduction}
\label{sec_intro}
In this work, we consider the inverse problem of recovering a function from its Radon transform. This problem arises when we aim at recovering an object from its projections in the computed tomography (CT). Accurate, stable, and fast numerical reconstruction methods are of great importance in practice in view of the broad and increasing applications of CT scan in, e.g., medical imaging, flaw detection, and baggage security scanning. 

To recover a function from its Radon transform, analytical inversion formulas are available. And some popular approaches nowadays are based on these formulas along with various low pass filters, known as filtered back projections (FBP). Two major reasons for the popularity of the FBP method are its easy implementation and its relatively low computational complexity \cite{why_FBP}. The method performs very well when the measurement data is very accurate and available from all directions. Nonetheless, the measurement data may be highly noisy and is only available in a limited range or only a number of angles, in many applications. For instance, to minimize adverse effects brought by radiation exposure upon a patient's body during the scanning process, low dose CT is widely employed for lung cancer detection \cite{NEJMoa1102873}. However, this may lead to severely polluted measurement data \cite{lu2001noise}, and in this case, it is difficult for traditional methods to work stably. Another instance is when we apply the CT scan in luggage security checks, we may only be able to collect measurement data from a small number of or/and a limited range of angles. Those scenarios are usually named as sparse tomography and limited angle tomography \cite{sparse_tomo, limited_angle_pde_based}.
It is challenging to develop efficient and effective reconstruction methods in these scenarios as the corresponding inverse problems are severely ill-posed \cite{math_of_CT}, and the data may not be adequate to sustain the robustness of traditional methods. 

Comprehensive reviews of traditional reconstruction methods, including the FBP method, the Fourier method, and the algebraic reconstruction method, can be found in two popular monographs \cite{principle_CT, math_of_CT}. 
In order to tackle those aforementioned ill-posed scenarios, many alternative numerical reconstruction methods have been developed in recent years, based on various strategies, including wavelet or shearlet approximations \cite{shearlet_TV, wavelet_limited_angle}, Bayesian method with appropriate priors \cite{chen2008nonlocal}, the PDE approach making use of the propagation of X-rays \cite{limited_angle_pde_based}, and the deep learning technique to complete missing data \cite{sinogram_completion_NN}.
Many of these methods work well in the sparse tomography \cite{sparse_tomo}, the limited angle tomography \cite{sinogram_completion_NN,shearlet_TV, limited_angle_pde_based}, and the low dose CT \cite{chen2008nonlocal}. These methods have demonstrated good potential in improving existing methods and were already applied in many real life scenarios. 
We would like to remark that since these algorithms usually involve more advanced mathematical or statistical tools, they may either employ certain optimization functional or leverage on the availability of a huge training dataset, which may lead to higher computational and storage complexities than the standard FBP methods.  Instead, we will propose a method that avoids high computational and storage complexities, and at the same time obtain a reasonable reconstruction in these difficult cases and challenging scenarios. 

In this work, we design a novel direct sampling method (DSM) for the inversion of the Radon transform. 
This type of methods was originally motivated by the almost orthogonality property between the fundamental solution of the forward problem and a properly chosen family of probing functions under certain Sobolev duality products. The DSMs have been constructed and developed for various highly nonlinear and severely ill-posed inverse problems; see, e.g., \cite{DOT, Chow_2014, chow2018a, Ito_2012, li2013a,potthast2010study}, including the wave and non-wave type inverse problems. These developments have demonstrated that the DSMs are robust against noise and could generate reasonable reconstruction results even with highly limited measurement data. These attractive features motivate us naturally to design a DSM for inverting the Radon transform, 
and this is the main focus of the current work. A key observation in our development is that if the measurement data is directly back-projected by the dual of the Radon transform, the result can be represented by an integral equation with the Green's function associated with (a fractional) Laplacian as its kernel. This suggests us to make full use of the important almost orthogonality property between the Green's function and a special family of probing functions under a fractional order Sobolev duality product. 
The choice of the fractional order operator arises naturally considering the ill-posedness of the inverse problem under noisy and incomplete measurement data, and turns out to be able to greatly enhance the robustness of the new DSM. In the meantime, in order to generate more satisfactory reconstruction results, we introduce the probing functions that depend on the sampling interval, which can further render a point-wise convergence of the index function in certain scenarios. 
From the perspective of the numerical computations, the DSM can be computed with low computational efforts and simultaneously with the measurement process. With these features, the new DSM is expected to find applications in tackling some inverse problems associated with the Radon transform, such as those arising from security scanning, cancer detection, and portable CT scanner.  These will be further verified numerically in section \ref{sec_num_eg}.

The rest of the paper runs as follows. Section \ref{sec_principle} introduces basic motivations and principles behind direct sampling type methods for the inversion of the Radon transform, including our detailed choices of probing and index functions. Section \ref{sec_verification} provides mathematical justifications for the novel DSM and investigates how the choice of some critical parameters in the sampling algorithm affects the reconstruction. Section \ref{sec_other} extends the newly proposed DSM to the limited angle tomography and the exponential Radon transform. Section \ref{sec_implementation} presents some strategies for the numerical implementation to enhance the robustness and reduce the computational complexity of the new DSM. Section \ref{sec_num_eg} demonstrates a series of numerical experiments by the new sampling method for some highly ill-posed scenarios, along with a comparison with the popular FBP method. 

\section{Principles of DSMs in inverting the Radon transform}
\label{sec_principle}
 In this section, we explain the basic principles of direct sampling type methods for the inversion of the Radon transform. The spirit of direct sampling type methods is to leverage upon an almost orthogonality property between the family of fundamental solutions of the forward problem and a set of probing functions under an appropriately chosen duality product \cite{DOT, Chow_2014, chow2018a}. With this in mind, we first represent the measurement data with the Green's function of (a fractional) Laplacian and then introduce a fractional order Sobolev duality product for the coupling of the measurement data and the probing function. At the same time, a family of probing functions will be constructed. Finally, an index function is defined to generate a direct sampling method for the inversion of the Radon transform.

Let us consider the target function to be recovered as $f$, which is contained in $L^{2}(\Omega)$, where $\Omega$ is a compact set in $\mathbb{R}^n$ ($n = 2$, $3$). Moreover, we assume $B(0,r_1) \subseteq \Omega \subseteq B(0,r_2)$, with $0<r_1\leq r_2$, where $B(x,r)$ is the ball centered at $x$ with radius $r$. The Radon transform of a function $f$ and its dual acting on a function $g\in L^{\infty}(S^{n-1}\times \mathbb{R}^n)$ are defined respectively by
\begin{equation}
\label{radon_para}
    Rf(\theta,t) := \int_{x\cdot \theta = t}f(x)dx_L = \int_{\mathbb{R}^n}f(x)\delta(t-x\cdot \theta)dx\,,\quad R^*g(x) := \int_{S^{n-1}}g(\theta,x\cdot \theta)d\theta \,,
\end{equation}
where $\theta \in S^{n-1}$, $t\in \mathbb{R}$, $x\in \Omega$, and $t = x\cdot \theta$ represents a hyperplane with normal direction $\theta$ and distance $t$ to the origin. We shall first focus on the case that $Rf(\theta,t)$ is avaliable for all $\theta \in S^{n-1}$ and for all $t \in I_{\theta}$, where $I_{\theta}$ is defined such that 
\begin{equation}
\label{def_Itheta}
  \Omega \subset \bigcup_{t \in I_{\theta}} \{x : t = x\cdot \theta\}\,.
\end{equation}
 In other words, we have measurements for all hyperplanes that intersect with the convex hull of $\Omega$. In section \ref{section_Limited_Angle}, we shall further consider the application of DSM for reconstruction with limited angle measurement, i.e., $t$ is only available for a subset of $I_{\theta}$.

A crucial motivation in our subsequent design of a DSM is the following inherent mathematical connection between the Radon transform and  (a fractional) Laplacian \cite{helgason1980the}:
\begin{equation}
\label{dual_G_relation}
      R^*Rf(x) = \frac{c_{n}}{d_n}\int_{\Omega}f(y)G_x(y) dy \quad \mbox{with~~} d_n =  \frac{\pi^{1/2}}{(4\pi)^{\frac{n-1}{2}}\Gamma(\frac{n}{2})} \quad \mbox{and~~}  c_{n} = \frac{\Gamma(\frac{n-1}{2})}{2\pi^{\frac{n+1}{2}}}\,,
\end{equation}
where $G_x(y) = |x-y|^{-1}$ is the Green's function for the (fractional) Laplacian operator $(-\Delta)^{(n-1)/{2}}$. The fractional Laplacian can be defined through various approaches which are equivalent to one another under appropriate assumptions on the regularity of $f$ \cite{ten}. We shall consider two definitions via a singular integral representation and a Fourier multiplier, respectively.

The following equivalent inversion formula will be frequently used in our subsequent analysis:
\begin{equation}
\label{inverse_formula}
    f(x) = (-\Delta)^{\frac{(n-1)}{2}}u_s(x) \quad \mbox{with~~} u_s(x) := d_nR^*Rf(x)\,.
\end{equation}
We shall call $u_s$ as the measurement data since the dual transform or the back projection $R^*$ of the Radon transform is standard and explicitly available after the Radon transform $Rf(\theta,t)$. 

The relation \eqref{inverse_formula} can be considered as the most important motivation for many existing reconstruction methods, e.g., the FBP and Fourier methods. These reconstruction methods involve usually the application of a pseudo-differential operator on the noisy measurement data which is not preferable for those ill-posed scenarios that were mentioned in the Introduction.

We remark that in order to allow \eqref{inverse_formula} to be held in $\mathbb{R}^2$, we shall assume that $f$ lies in the Schwarz space which is the space of functions whose derivatives are all rapidly decreasing. This assumption will not affect the feasibility of reconstructing the target function $f \in L^2(\Omega)$. Using the density of smooth functions in $L^{2}(\Omega)$, all our upcoming analyses involving \eqref{inverse_formula} (section \ref{sec_verification}) can be first carried out for smooth functions, and then extended to a more general class of target functions by a standard density argument.

To define an index function in a direct sampling method, we first introduce a duality product of order $\gamma > 0$ for the coupling of the measurement data $u_s$ with some appropriately selected probing functions (to be defined):
\begin{equation}
\label{def_inner_pro}
    \langle v, w\rangle _{H^{\gamma}(\mathbb{R}^n)} := \int_{\mathbb{R}^n} v \,(-\Delta)^{\gamma}w dx\,, \quad \forall\,v\in L^2(\mathbb{R}^n)\,, ~w \in H^{2\gamma}(\mathbb{R}^n)\,.
 \end{equation}
We remark that $v$ will be often the noisy measurement data in our proposed DSM, and the parameter $\gamma$ is called the Sobolev scale of the duality product.

The new DSM will reply on a critical index function, which involves an appropriately selected family of probing functions. Before going on with more details, we first present one of the primary motivations for employing the duality product in \eqref{def_inner_pro} and the construction of probing functions for the inversion of the Radon transform. 
Letting us consider $n=2$, then we choose $w = u_s$ from \eqref{inverse_formula}, $v = G_z$ and Sobolev scale $\gamma = 1$ in \eqref{def_inner_pro}. Then we can easily derive by the definition of the Green's function and the inversion formula in \eqref{inverse_formula} that
\begin{equation}
\label{motivation_duality_inverse}
   ((-\Delta)u_s\ast G_0)(z) = \mathcal{F}^{-1}\{|\omega|\mathcal{F}(f)\mathcal{F}(G_0)\}(z) = f(z)\,,
\end{equation}
where $\mathcal{F}$ and $\mathcal{F}^{-1}$ denote the Fourier transform and the inverse Fourier transform, and $\omega$ is the variable in the frequency domain. Therefore, this duality product can be linked with an exact reconstruction formula. However, we may directly observe that taking the Laplacian on $u_s$ will cause numerical instability due to the noise in the data, especially at those scenarios we mentioned in the Introduction. Furthermore, essentially different from the previous DSMs \cite{DOT, Chow_2014, chow2018a} for which the measurement data is collected on a partial boundary of the sampling domain, we now have the data $u_s$ inside $\Omega$. Hence, it is not desirable in practice for our numerical methods to involve the singularity of the Green's function (as in \eqref{motivation_duality_inverse}) in computations. For this reason, we shall introduce and justify the following strategies (in sections \ref{subsec_definition} and \ref{sec_verification}.):
\begin{itemize}
    \item To enhance the robustness against noise, a smaller Sobolev scale $\gamma$ in the duality product will be preferable when the measurement data is highly noisy.  Moreover, we will illustrate in section \ref{sec_verification} the relationship between $\gamma$ and the variance of the index function under a simplified noise model.    
    \item We will introduce a special family of probing functions to avoid any singularities at the sampling point $z$ but still preserve the sharpness of the inversion formula.
\end{itemize}

\subsection{Probing and index functions}
\label{subsec_definition}
We are now going to propose an appropriate family of probing functions based on the primary motivation and principles of direct sampling type methods that we addressed earlier. For the purpose, we first define two sets of auxiliary functions $\zeta^{h}_{\alpha}$ and $\widetilde{\zeta}^{h}_{\alpha}$ for any $0<h<1$ and $\alpha \in \mathbb{R}$:
\begin{equation}
\label{def_zeta}
  \zeta^{h}_{\alpha}(x) := \begin{cases}
        |x|^{-\alpha}\,,& \quad \text{when } |x|\geq h \,,\\
        \psi_{\alpha}(|x|)\,,& \quad  \text{when } |x|<h\,;\\ 
    \end{cases}\qquad  
     \widetilde{\zeta}^{h}_{\alpha}(x) := \begin{cases}
        |x|^{-\alpha}\,,& \quad \text{when } |x|\geq h \,,\\
        h^{-\alpha}\,,& \quad  \text{when } |x|<h\,;\\
    \end{cases}
\end{equation}
where $\psi_{\alpha}(x)$ is a smooth extension function such that $\zeta^{h}_{\alpha}(x)\in C^{2,1}(\mathbb{R}^n)$ and $||\zeta^{h}_{\alpha} - \widetilde{\zeta}^h_{\alpha}||_{L^1(\mathbb{R}^n)}<h$. 
By the density of smooth functions in $L^2(\mathbb{R}^n)$, we will present an explicit choice of the smooth extension function $\psi_{\alpha}$ that we use in our numerical computations with verification of its desired property in Appendix \ref{sec_appencix}.

In the sequel, $\zeta^h_{\alpha}$ is used to construct a crucial family of probing functions, and $\widetilde{\zeta}^h_{\alpha}$ will be repeatedly employed in the theoretical justification of the DSM in section \ref{sec_verification}. These auxiliary functions can be regarded as some delicate modifications of the Green's function associated with the (fractional) Laplacian $(-\Delta)^{(n-1)/2}$. The modifications are necessary for two reasons. The first is that the original Green's function is singular at the origin, therefore we need to remove the singularity but still preserve certain smoothness property. Secondly, a key parameter $\alpha$ is introduced to realize a more satisfactory reconstruction result. 
Indeed, we will justify in section \ref{sec_alternative_chara} that a reasonable and reliable choice is $\alpha = n+1$. 

We are now ready to define a crucial family of probing functions $\eta_z^h$ at any sampling point $z\in\Omega$:
\begin{equation}
\label{def_probing}
  \eta_z^{h}(x) :=  \zeta_{n+1}^h(x-z)\,.
\end{equation}

For the notational sake, we also denote 
\begin{equation}
\label{def_auxiliary_probing}
\widetilde{\eta}_z^h(x) : = \widetilde{\zeta}_{n+1}^h(x-z)\,.
\end{equation}

Before we move on to introduce the important index function for defining the direct sampling method, we first provide some estimates of probing functions, 
which will be used repeatedly in the verification of the new DSM in section \ref{sec_verification}.

\begin{Lemma}
	\label{lemma_L1eta}
The following estimates hold for the probing and auxiliary functions $\eta_z^{h}$ and $\widetilde{\eta}_z^h$: 
\begin{enumerate}[label=(\alph*)] 
	\item  $(-\Delta)^{\gamma}\eta_z^h(x)$ belongs to $L^{\infty}(\mathbb{R}^n)$ for $0 < \gamma < \frac{n}{2}$;
	\item $(-\Delta)^{\gamma}\eta_z^{h}(x)$ belongs to $L^2(\mathbb{R}^n)$ for $\, 0<\gamma<1$;
	\item  $(-\Delta)^{\gamma}\widetilde{\eta}_z^{h}(x)$ belongs to $L^2(\mathbb{R}^n)$ for $ 0<\gamma<1$.
\end{enumerate}
\end{Lemma}
\begin{proof}
Without loss of generality, we assume that $z$ is the origin. 

\smallskip

To show part (a), we first consider the case $\gamma \in (0,1)$. By definition, the fractional Laplacian of $\eta_0^h$ for an arbitrary point $x\in \mathbb{R}^n$ can be written as
\begin{equation}
\label{eta_integration}
\begin{split}
	(-\Delta)^{\gamma}\eta_0^h(x) = 
	& -\frac{c_{n,\gamma}}{2} \lim_{\delta\rightarrow 0}\int_{|y|>\delta }\frac{\eta_0^h(x+y) + \eta_0^h(x-y)- 2\eta_0^h(x)}{|y|^{n+2\gamma}} dy \\
 =&-\frac{c_{n,\gamma}}{2} \big(I_1 + I_2\big) \,, \quad \text{with~~} c_{n,\gamma} = \frac{4^{\gamma}\Gamma(\frac{n}{2}+\gamma)}{\pi^{n/2}|\Gamma(-\gamma)|}\,,
\end{split}
\end{equation} 
where $I_1$ and $I_2$ are
\begin{equation*}
 I_1 = \int_{|y|>\frac{|x|}{2}}\frac{\eta_0^h(x+y) + \eta_0^h(x-y)- 2\eta_0^h(x)}{|y|^{n+2\gamma}} dy\,,\,\,
 I_2 = \lim_{\delta\rightarrow 0}\int_{\delta<|y|<\frac{|x|}{2}}\frac{\eta_0^h(x+y) + \eta_0^h(x-y)- 2\eta_0^h(x)}{|y|^{n+2\gamma}} dy\,.
\end{equation*}

$I_1$ can be bounded directly by
\begin{equation}
\label{I_1_est}
	|I_1| \leq \int_{|y|>\frac{|x|}{2}}\frac{|\eta_0^h(x+y) + \eta_0^h(x-y)- 2\eta_0^h(x)|}{|y|^{n+2\gamma}} dy \leq  4||\eta_0^h||_{L^1(\mathbb{R}^n)}|x/2|^{-n-2\gamma}\,.
\end{equation} 
while $I_2$ can be bounded by
\begin{equation}
\label{I_2_est}
	|I_2|\leq \lim_{\delta\rightarrow 0}\int_{\delta<|y|<\frac{|x|}{2}}\frac{||D^2\eta_0^h||_{L^{\infty}(B(x,|x|/2))}}{|y|^{n+2\gamma-2}}dy \leq c_2(n,\gamma)||D^2\eta_0^h||_{L^{\infty}(B(x,|x|/2))}|x/2|^{2-2\gamma}\,,
\end{equation}
by using the estimate
\begin{equation}
\frac{|\eta_0^h(x+y) + \eta_0^h(x-y)- 2\eta_0^h(x)|}{|y|^{n+2\gamma}}  \leq  \frac{||D^2 \eta_0^h||_{L^{\infty}(B(x,|x|/2))}}{|y|^{n+2\gamma -2}}\,, \quad \text{where } |y|<\frac{|x|}{2}\,.
\end{equation}
from the second order Taylor's theorem.
Here the constant $c_2(n,\gamma)$ only depends on $n$ and $\gamma$. Combining \eqref{I_1_est} and \eqref{I_2_est}, we conclude that $(-\Delta)^{\gamma}\eta_0^h \in L^{\infty}(\mathbb{R}^n)$ for $\gamma \in (0,1)$.

Now we show  $(-\Delta)^{\gamma}\eta_z^h(x) \in L^{\infty}(\mathbb{R}^n)$ for $\gamma \in [1,\frac{n}{2})$. We first establish a result that will be used twice in the following proof. For arbitrary $g\in C^{0,1}(\mathbb{R}^n) \cap L^{\infty}(\mathbb{R}^n)$ and $\beta \in (0,1/2)$, we have
\begin{align}
\label{C01_estimate}
	|(-\Delta)^{\beta}g(x)| 	\leq &\,\,c_{n,\beta}\bigg[\int_{B(0,1)} \frac{||g||_{C^{0,1}(\mathbb{R}^n)}}{|y|^{n+2\beta-1}}dy  +\int_{\mathbb{R}^n\setminus B(0,1)}\frac{2||g||_{L^{\infty}(\mathbb{R}^n)}}{|y|^{n+2\beta}}dy\bigg]\\
	\leq& \,\, c_3(n,\beta) \big(||g||_{C^{0,1}(\mathbb{R}^n)} + ||g||_{L^{\infty}(\mathbb{R}^n)}\big)\notag\,
\end{align}
for $c_{n,\beta}$ defined in \eqref{eta_integration} and some constant $c_3(n,\beta)$ that only depends on $n$ and $\beta$. The above estimate implies $(-\Delta)^{\beta}g\in L^{\infty}(\mathbb{R}^n)$ for any $g\in C^{0,1}(\mathbb{R}^n) \cap L^{\infty}(\mathbb{R}^n)$ and $\beta \in (0,1/2)$.

Next, by the construction of the probing function in \eqref{def_probing}, we have $(-\Delta)\eta_0^h \in C^{0,1}(\mathbb{R}^n) \cap L^{\infty}(\mathbb{R}^n)$ which shows the case $\gamma = 1$. For $\gamma >1$, to make use of the  estimate in \eqref{C01_estimate}, we observe that the order $\gamma$ fractional Laplacian of the probing function can be written as
$$
(-\Delta)^{\gamma}\eta_0^h = (-\Delta)^{\gamma-1}\big((-\Delta)\eta_0^h\big)\quad \mbox{with~~} \gamma \in (1,\frac{n}{2})\,.
$$
This indicates $(-\Delta)^{\gamma}\eta_0^h\in L^{\infty}(\mathbb{R}^n)$, by replacing $g$ by $(-\Delta)\eta_0^h$ and $\beta$ by $\gamma - 1\in (0,1/2)$ in \eqref{C01_estimate}.

\smallskip

To show part (b) that $(-\Delta)^{\gamma}\eta_0^h \in L^2(\mathbb{R}^n)$ for $\gamma \in (0,1)$, it suffices to show that $|(-\Delta)^{\gamma}\eta_0^h(x)|\leq c |x|^{-n-\epsilon}$ for $|x|>2h$, $\epsilon>0$, and some constant $c$ that is independent of $x$. This property is satisfied by $I_1$ due to \eqref{I_1_est}. Now we investigate $I_2$ in \eqref{eta_integration} more carefully. By the definition of the probing function, we have 
\begin{equation}
\label{D2_eta}
||D^2\eta_0^h||_{L^{\infty}(B(x,|x|/2))} \leq 4n^2|x/2|^{-n-3} \text{~~when~~} |x|>2h \,.
\end{equation}
Substituting this estimate into \eqref{I_2_est} implies that $|I_2| \leq c_2(n,\gamma)4n^2|x/2|^{-n-1-2\gamma}$. Hence, we can conclude that $(-\Delta)^{\gamma}\eta_0^h \in L^2(\mathbb{R}^n)$.

\smallskip

To show part (c) that $(-\Delta)^{\gamma}\widetilde{\eta}_0^h \in L^2(\mathbb{R}^n)$ for $\gamma \in (0,1)$. We first notice that for $|x|>2h$, all above estimates in \eqref{I_1_est}, \eqref{I_2_est}, and \eqref{D2_eta} hold after replacing $\eta_0^h$ by $\widetilde{\eta}_0^h$ since the point-wise value of $\widetilde{\eta}_0^h(x)$ for $|x|<h$ is not involved in those estimations. Therefore, we have $(-\Delta)^{\gamma}\widetilde{\eta}_0^h \in L^2(\mathbb{R}^n\setminus B(0,2h))$. Now, it only remains to show $(-\Delta)^{\gamma}\widetilde{\eta}_0^h$ belongs to $L^2(B(0,2h))$. 

For $\gamma \in (0,1/2)$, with \eqref{C01_estimate} and the definition in \eqref{def_zeta} which states that $\widetilde{\eta}_0^h \in C^{0,1}(\mathbb{R}^n)$, we have $(-\Delta)^{\gamma}\widetilde{\eta}_0^h \in L^{\infty}(B(0,2h))\subset L^2(B(0,2h))$. 

For $\gamma \in [1/2,1)$, denoting $\overline{\eta}_0^h := \eta_0^h - \widetilde{\eta}_0^h$ which satisfies $\overline{\eta}_0^h(x) =0$ if $|x|>h$. Then showing $(-\Delta)^{\gamma}\widetilde{\eta}_0^h$ belongs to $L^2(B(0,2h))$ is equivalent to showing $(-\Delta)^{\gamma}\overline{\eta}_0^h$ belongs to $L^2(B(0,2h))$. By definition, for $h<|x|<2h$, we have
\begin{equation*}
	|(-\Delta)^{\gamma}\overline{\eta}_0^h(x)| = c_{n,\gamma}\bigg|\int_{|y|<h}\frac{\overline{\eta}_0^h(y)}{|x-y|^{n+2\gamma}}dy \bigg| \leq 4\pi^2 c_{n,\gamma}||\eta_0^h - \widetilde{\eta}_0^h||_{L^{\infty}(\mathbb{R}^n)} (|x|-h)^{-2\gamma}\,.
\end{equation*}
For $|x|<h$, similarly to the decomposition in \eqref{eta_integration}, we have
{\footnotesize
\begin{equation*}
\begin{split}
	|(-\Delta)^{\gamma}\overline{\eta}_0^h(x)| \leq &\, c_{n,\gamma}\bigg|\int_{|y|<h-|x|}\frac{\overline{\eta}_0^h(x-y) + \overline{\eta}_0^h(x+y) - \overline{\eta}_0^h(x)}{2|y|^{n+2\gamma}}dy\bigg| + c_{n,\gamma}\bigg|\int_{h-|x|<|y|< 2h} \frac{\overline{\eta}_0^h(x) - \overline{\eta}_0^h(x+y)}{|y|^{n+2\gamma}}dy\bigg| \\
	\le & \,c_4(n,\gamma)\bigg[||D^2(\eta_0^h-\widetilde{\eta}_0^h)||_{L^{\infty}(B(0,h))}(h-|x|)^{2-2\gamma} + 2||\eta_0^h - \widetilde{\eta}_0^h||_{L^{\infty}(\mathbb{R}^n)}(h-|x|)^{2\gamma}\bigg]\,,
\end{split}
\end{equation*}}
for some constant $c_4(n,\gamma)$ that is independent of $x$. Combining estimates for $|x|>h$ and $|x|<h$, we conclude that $(-\Delta)^{\gamma}\overline{\eta}_0^h \in L^2(B(0,2h))$ which leads to part (c) of the lemma.
\end{proof}

We are now ready to introduce the crucial index function $I_{\gamma}^h$ that defines the direct sampling method, more accurately speaking, it generates the numerical image at all sampling points $z\in \Omega\,$:
\begin{equation}
\label{index_fun}
    I^{h}_{\gamma}(z) := \frac{\langle u_s, \,\eta_z^{h} \rangle_{H^{\gamma}(\mathbb{R}^n)}}{n(z)}\,,\, \quad 
    \text{ with }\,\,  n(z) := \langle d_n R^*R(\mathbbm{1}_{\Omega}), \eta_z^h\rangle_{H^{\gamma}(\mathbb{R}^n)}\, ,   
\end{equation}
where $\eta_z^h$ is the probing function introduced in \eqref{def_probing} and $d_n$ is defined in \eqref{dual_G_relation}. The normalization term $n(z)$ is taken to migrate the influence of the choice of $h$ and $\gamma$ on the magnitude of the index function, which is independent of the measurement data $u_s$. In particular, this choice of $n(z)$ ensures that our method is exact for constant valued target function. In the following sections \ref{sec_var_case1} and \ref{sec_var_case2}, we shall justify that the numerator of the index function, that is, the duality product between $u_s$ and $\eta_z^h$, will approximately recover the target function $f(z)$ up to a constant.
With this in mind, we observe that $n(z)$ is simply an approximation to the characteristic function of the sampling domain $\Omega$, hence $n(z)$ is nearly a constant for all sampling points in $\Omega$.

We remark that, since $u_s$ introduced in \eqref{inverse_formula} is not compactly supported,  the duality product involved in the index function \eqref{index_fun} is defined with respect to $\mathbb{R}^n$. However, the numerical implementation of the index function is still realized in a compact set due to the fact that the target function $f$ is often compactly supported in $\Omega$. The implementation of the new DSM will be presented in detail in section \ref{sec_implementation}

The proposed index function leverages upon the very important almost orthogonality property of the Green's function $G_x$ and the family of probing functions defined in \eqref{def_probing} in fractional order Sobolev duality products. Combining with the representation of the measurement data that we introduced in \eqref{inverse_formula}, this desired property helps reconstruct the target function $f$ with the index function (see the careful verification in section \ref{sec_verification}).
We now emphasize a very important feature of the novel DSM. By the definition of the index function \eqref{index_fun}, the evaluation of the index function does not involve any pseudo-differential operator applied to the noisy measurement data $u_s$, unlike many existing numerical methods in inverting the Radon transform.
This feature shall allow our DSM to be stable under high level noise and limited measurement data, which is evident from many numerical experiments in section \ref{sec_num_eg}.

Under the setting-ups above, the index function in \eqref{index_fun} gives rise to our new algorithm:

\medskip 
{Direct Sampling Method.}
Given the Radon transform $Rf(\theta,t)$ of the target function $f$ for a limited set of discrete angles $\theta\in\mathbb{S}^{n-1}$ and discrete points $t\in\mathbb{R}$, we evaluate $I_{\gamma}^h(z)$ numerically to approximate $f(z)$ for every 
sampling point $z$ in the domain $\Omega$.
\smallskip

\section{Verification of the index function}
\label{sec_verification}
In this section, we verify mathematically that the proposed index function in \eqref{index_fun} can indeed recover the target function $f$ in subsections \ref{sec_var_case1} and \ref{sec_var_case2}, in two separate scenarios that the Sobolev scale $\gamma\in ((n-1)/2,n/2)$ and $\gamma\in(0, (n-1)/2]$. 
In particular, the verification for the latter case relies on the alternative characterization of the index function that will be presented in subsection \ref{sec_alternative_chara}. Moreover, the choice of the key parameter $\alpha = n+1$ in the definition of the probing function \eqref{def_probing} will also be explained in the same subsection. In subsection \ref{sec_noise}, we will demonstrate the relationship between the Sobolev scale $\gamma$ in \eqref{def_inner_pro} and the variance of the index function under a particular noise model which provides crucial instruction on the choice of $\gamma$ during the reconstruction with noisy measurement data. In addition, the conclusion from subsection \ref{sec_noise} implies the choice of $\gamma \geq n/2$ is not preferable in real applications and hence we only consider the possibility of $\gamma <n/2$ in the following discussion.

Throughout this section, with the help of the  remarks that we mentioned after \eqref{inverse_formula}, we further assume that the target function $f$ is a smooth function.

\subsection{Verification of the index function for ${(n-1)}/{2}< \gamma < {n}/{2}$}
\label{sec_var_case1}
We will first focus on the numerator of the index function defined in \eqref{index_fun}, and verify that the duality product between $u_s = d_n R^*Rf$ and the probing function defined in \eqref{def_probing} can recover the target function as $h\rightarrow 0$ for any sampling point $z\in\Omega$. Especially, the parameter $h$ could be considered as the sampling interval in real applications. 
\begin{Lemma}
\label{lemma_ptwise_limit}
 For any $f\in C^{0,1}(\Omega)$ and $z\in \Omega$, it holds that 
\begin{equation}
\label{ptwise_limit}
  \lim_{h\rightarrow 0} h^{1+2\gamma} \langle u_s\,,\,\eta_{z}^h\rangle_{H^{\gamma}(\mathbb{R}^n)}=  ||(-\Delta)^{\gamma-\frac{n-1}{2}}\widetilde{\eta}_{0}^1||_{L^1(\mathbb{R}^n)} f(z)\,,
\end{equation}
In particular, the convergence is uniform for all $z\in \Omega$.
\end{Lemma}
\begin{proof}

Firstly, based on the inversion formula \eqref{inverse_formula} and the self-adjointness of the fractional Laplacian which holds due to part (b) and (c) of Lemma \ref{lemma_L1eta}, we can write 
\begin{equation}
    \label{int_by_pats_1}
    \begin{split}
    \langle u_s,\eta_z^{h}\rangle_{H^{\gamma}(\mathbb{R}^n)} 
    =&\int_{\mathbb{R}^n}u_s(x)(-\Delta)^{\gamma}\eta_z^h(x)dx 
    =\int_{\mathbb{R}^n}(-\Delta)^{\gamma-\frac{n-1}{2}}f(x)\eta_z^{h}(x)dx \\
    =& \int_{\mathbb{R}^n}f(x)(-\Delta)^{\gamma-\frac{n-1}{2}}\widetilde{\eta}_z^{h}(x)dx + \int_{\mathbb{R}^n}(-\Delta)^{\gamma-\frac{n-1}{2}}f(x)[\eta_z^h(x) - \widetilde{\eta}_z^h(x)]dx \,.
    \end{split}
\end{equation}
For the second integration above, by definitions of $\eta_z^h$ and $\widetilde{\eta}_z^h$ in \eqref{def_probing} and \eqref{def_auxiliary_probing}, we have
\begin{equation}
\label{diff_smooth}
\int_{\mathbb{R}^n}(-\Delta)^{\gamma-\frac{n-1}{2}}f[\eta_z^h - \widetilde{\eta}_z^h ]dx \leq ||(-\Delta)^{\gamma-\frac{1}{2}}f||_{L^{\infty}(\mathbb{R}^n)}||\zeta_{n+1}^h - \widetilde{\zeta}_{n+1}^h||_{L^1(\mathbb{R}^n)}\leq  ||(-\Delta)^{\gamma-\frac{1}{2}}f||_{L^{\infty}(\mathbb{R}^n)}h\,,
\end{equation}
where the boundness of the term $||(-\Delta)^{\gamma-\frac{1}{2}}f||_{L^{\infty}(\mathbb{R}^n)}$ follows from $f\in C^{0,1}(\mathbb{R}^n)$ and the estimate  \eqref{C01_estimate}.

Moreover, by substituting $ \widetilde{\eta}_{0}^1\big(\frac{x}{h}\big) = h^{n+1}\widetilde{\eta}_0^h(x)$ which comes from \eqref{def_zeta} and \eqref{def_auxiliary_probing} into the definition of the fractional Laplacian operator, we further have the following rescaling property:
\begin{equation}
   (-\Delta)^{\gamma-\frac{n-1}{2}}(\widetilde{\eta}_0^h)(x) = h^{n+1+2\gamma}\big((-\Delta)^{\gamma-\frac{n-1}{2}}\widetilde{\eta}_{0}^1\big)\bigg(\frac{x}{h}\bigg)\,,
\end{equation}
For simplicity, we write $\beta = \gamma - \frac{n-1}{2}$. By part (c) of Lemma \ref{lemma_L1eta} that $(-\Delta)^{\beta}\widetilde{\eta}_0^1\in L^2(\mathbb{R}^n)\subset L^1(\mathbb{R}^n)$, we can define a family of approximations to the identity for all $f\in C^{0,1}({\Omega})$:
\begin{equation}
\label{approx_identity}
    \tau^{h}_{\beta}(x)=  \frac{h^{-n}(-\Delta)^{\beta}\widetilde{\eta}_{0}^1(\frac{x}{h})}{||(-\Delta)^{\beta}\widetilde{\eta}_0^1||_{L^1(\mathbb{R}^n)}} = h^{1+2\gamma}\frac{(-\Delta)^{\beta}(\widetilde{\eta}_0^h)(x)}{||(-\Delta)^{\beta}\widetilde{\eta}_0^1||_{L^1(\mathbb{R}^n)}}\,, 
    \quad \lim_{h\rightarrow 0}\int_{\mathbb{R}^n}f(y)  \tau^{h}_{\beta}(z-y)dy = f(z)\,.
\end{equation}
Then, combining \eqref{int_by_pats_1} and \eqref{approx_identity}, we conclude for a fixed sampling point $z\in\Omega$ that
\begin{equation}
\begin{split}
\lim_{h\rightarrow 0}h^{1+2\gamma}\langle u_s, {\eta}_z^h\rangle_{H^{\gamma}(\mathbb{R}^n)} = &\lim_{h\rightarrow 0} h^{1+2\gamma}\int_{\mathbb{R}^n}f(y) (-\Delta)^{\gamma-\frac{n-1}{2}}\eta_z^{h}(y)dy \\
=& ||(-\Delta)^{\gamma-\frac{n-1}{2}}\widetilde{\eta}_0^1||_{L^1(\mathbb{R}^n)}\lim_{h\rightarrow 0} \int_{\mathbb{R}^n}f(y) \tau_{\beta}^h(z-y)dy\\
=& ||(-\Delta)^{\gamma-\frac{n-1}{2}}\widetilde{\eta}_0^1||_{L^1(\mathbb{R}^n)}f(z)\,.
\end{split}
\end{equation}
In particular, for all $z\in \Omega$, the convergence is uniform as the limit in \eqref{approx_identity} is uniform.
\end{proof}

Lemma \ref{lemma_ptwise_limit} indicates that the numerator of the index function can recover the target function $f$ up to a constant, when $h$ is small enough.

We move the justification of the index function for the case of the Sobolev scale $\gamma \le (n-1)/2$ to section \ref{sec_var_case2}. Before that, we next present an alternative characterization of the index function which explains the choice of $\alpha$ in \eqref{def_probing}. This alternative characterization will be also essential to our subsequent justification of the index function for $\gamma\leq (n-1)/2$.

\subsection{Alternative characterization of the index function}
\label{sec_alternative_chara}
In this subsection, we present an alternative characterization of the index function defined in \eqref{index_fun} for all possible choices of $\gamma \in (0,\frac{n}{2})$. The characterization is mainly to obtain a dominating term in the index function with respect to the small parameter $h$ involved in the probing function \eqref{def_zeta}. More specifically, we shall show that the index function at the sampling point $z$ approximately equals to the average of $(-\Delta)^{\gamma}u_s$ at the neighborhood of $z$. This characterization will be used in twofold:
\begin{itemize}
	\item We shall justify that the preferable choice of the key parameter $\alpha \in \mathbb{R}$ involved in the probing function is $\alpha = n+1$, as we suggested in \eqref{def_probing}. To do so, we will estimate and investigate the dominating term of the index function when the probing function \eqref{def_probing} is used or replaced by other functions $\zeta_{\alpha}^h$ with $\alpha \ne n+1$;
	\item The dominating term in the index function will provide an essential tool to help us justify that the proposed DSM can approximately recover the target function $f$ when $\gamma \in (0,(n-1)/2]$ in the next subsection \ref{sec_var_case2}.
\end{itemize}

Let us first assume $\alpha = n+1$, which is the one used in the definition \eqref{def_probing}. To obtain a dominating term of the numerator of the index function, we rewrite it, by using a direct addition and subtraction,  as
\begin{align}
\label{decomposition}
    &\langle u_s,\eta_z^{h} \rangle_{H^{\gamma}(\mathbb{R}^n)} 
    = \underbrace{\int_{\mathbb{R}^n}(-\Delta)^{\gamma}u_s\big[\eta_z^h - \widetilde{\eta}_z^{h}\big] dx}_{\phi_1(z)} -\frac{1}{n-1}\underbrace{\int_{\mathbb{R}^n}(-\Delta)^{\gamma+1}u_s\bigg[\widetilde{\zeta}_{n-1}^h(x-z)-\frac{1}{|x-z|^{n-1}}\bigg]dx}_{\phi_2(z)}\notag\\
    & - \frac{1}{n-1}\underbrace{\int_{\mathbb{R}^n}(-\Delta)^{\gamma+1}u_s\frac{1}{|x-z|^{n-1}}dx}_{\phi_3(z)} 
     +\underbrace{\int_{\mathbb{R}^n}\bigg[(-\Delta)^{\gamma}u_s\widetilde{\eta}_z^h+\frac{1}{n-1}(-\Delta)^{\gamma+1}u_s\widetilde{\zeta}_{n-1}^h(x-z)\bigg]dx}_{\phi_4(z)}
       \,,
\end{align}
where $\widetilde{\zeta}^h_{n-1}$ is defined in \eqref{def_zeta} and the derivation of the constant $1/(n-1)$ appeared in \eqref{decomposition} will be introduced in the analysis of $\phi_4(z)$. 

We now investigate the properties of the terms $\phi_i(z)$ ($1\le i\le 4$) one by one. For $\phi_1(z)$, the estimate is identical with \eqref{diff_smooth}, so we have
\begin{equation}
\label{phi_1}
|\phi_1(z)| \leq  ||(-\Delta)^{\gamma-\frac{1}{2}}f||_{L^{\infty}(\mathbb{R}^n)}h\,.
\end{equation}

Next, for $\phi_2(z)$, recalling the definition of $\zeta_{n-1}^h$ in \eqref{def_zeta}, we notice the integrand vanishes if $|x-z|>h$, which leads to 
\begin{equation}
\label{phi_2}
\begin{split}
        |\phi_2(z)| = \bigg|\int_{|x-z|\leq h}(-\Delta)^{\gamma+1}u_s\bigg[\frac{1}{|x-z|^{n-1}}-\frac{1}{h^{n-1}})\bigg]dx\bigg| 
        \leq  4\pi||(-\Delta)^{\gamma+\frac{1}{2}}f||_{L^{\infty}(\Omega)}h\,.    
\end{split}
\end{equation}

To consider the term $\phi_3(z)$, we know from \cite{stinga2018user} that for a smooth function $g$, the negative order fractional Laplacian can be represented by
\begin{equation*}
	(-\Delta)^{-\gamma}g(z) = c_{n,-\gamma}\int_{\mathbb{R}^n}\frac{g(x)}{|x-z|^{n-1}}dx \quad \mbox{with~~}  c_{n,-\gamma} = \frac{\Gamma(\frac{n}{2}-\gamma)}{4^{\gamma}\pi^{\frac{n}{2}}\Gamma(\gamma)}\,.
\end{equation*}
Using this property, taking $g=(-\Delta)^{\gamma+1}u_s$ and $\gamma=1/2$, we have
\begin{equation}
\label{phi_3}
  \phi_3(z) = \frac{1}{c(n,-\frac{1}{2})} (-\Delta)^{-\frac{1}{2}}\big((-\Delta)^{\gamma+1}u_s(z)\big) = \frac{1}{c(n,-\frac{1}{2})} (-\Delta)^{\gamma+\frac{1}{2}}u_s(z) \,.
\end{equation}
To summarize, we notice that the orders of $\phi_1(z)$ and $\phi_2(z)$ are $\mathcal{O}(h)$, and the magnitude of $\phi_3(z)$ is independent of the choice of $h$.

Finally, we come to analyse $\phi_4(z)$. Since $(-\Delta)\widetilde{\zeta}_{n-1}^h(x) = - (n-1)\widetilde{\zeta}_{n+1}^h(x)$ for $|x|>h\,$,  the Green's identity leads to
\begin{equation}
\label{phi_4}
\begin{split}
    \phi_4(z) = &\int_{\mathbb{R}^n}\bigg[(-\Delta)^{\gamma}u_s \widetilde{\zeta}_{n+1}^h(x)+\frac{1}{n-1}(-\Delta)^{\gamma+1}u_s\widetilde{\zeta}_{n-1}^h(x-z)\bigg]dx\\
    =&\int_{\partial B(z,h)}(-\Delta)^{\gamma}u_s \frac{\partial}{\partial n^-}\frac{1}{|x-z|^{n-1}}dx_s + \frac{1}{h^{n+1}}\int_{B(z,h)}(-\Delta)^{\gamma}u_sdx \\
    =& \,\frac{1}{h^{n}}\int_{\partial B(z,h)}(-\Delta)^{\gamma}u_sdx_s + \frac{1}{h^{n+1}}\int_{B(z,h)}(-\Delta)^{\gamma}u_sdx \,, 
\end{split}
\end{equation}
where $n^-$ denotes the normal vector pointing towards to $z$. The simplification of the integration on $\partial B(z,h)$ comes from the definition of $\widetilde{\zeta}_{n-1}^h$, since it is a constant inside $B(z,h)$. 

We can easily observe from \eqref{phi_4} that the order of $\phi_4(z)$ is $\mathcal{O}\big(h^{-1}\big)$, which is clearly larger than $\phi_i(z)$ $(1\le i \leq 3)$. Hence we can now conclude that the dominating term in the duality product $\langle u_s , \eta_z^h \rangle_{H^{\gamma}}$ is $\phi_4(z)$, and it can be readily seen as a good approximation of the average of $(-\Delta)^{\gamma}u_s$ in a close neighborhood of the sampling point $z$. This fact will be used in the next section \ref{sec_var_case2}.

\medskip

We are now ready to justify our choice of $\alpha = n+1$ in the definition of the probing function in \eqref{def_probing}. 
Firstly, the choice of $\alpha \leq n$ is not applicable as $\zeta_\alpha^h \notin L^1(\mathbb{R}^n)$. In this case, the index function \eqref{index_fun} which involves integration in $\mathbb{R}^n$ might not be always well defined and we can not ensure its accuracy and stability of reconstruction.

Secondly, for the choice of $\alpha>n+1$, there are two reasons that this option is not preferable. The first one is that the $L^{\infty}$-norm of the auxiliary function $\zeta_{\alpha}^h$ in \eqref{def_zeta} is of order $h^{-\alpha}$. Hence, a larger choice of $\alpha>n+1$ may lead to an issue of numerical instability. Let us discuss a special case of $\alpha> n+1$ below, i.e., $\alpha = n+3$, and we will conclude that the dominating term in the duality product between $u_s$ and $\zeta_{n+3}^{h}$ is the same as the dominating term in the duality product between $u_s$ and $\zeta_{n+1}^{h}$ (the probing function we employed in DSM).
We will compare the numerical reconstructions (see Example 1, section \ref{sec_num_eg}), with $\alpha$ being $n+1$,  $n+2$, and $n+3$, to justify the choice of $\alpha = n+1$ in our DSM.

Let us now consider the case $\alpha = n+3$, that is, the probing function \eqref{def_probing} used in the index function \eqref{index_fun} is replaced by $\zeta_{n+3}^h$.
We first observe that $(-\Delta)^2\widetilde{\zeta}_{n-1}^h(x) = 1/e(n)\widetilde{\zeta}_{n+3}^h(x)$ for $x>h$ with $e(n) = 1/3(n^2-1)$. We rewrite the duality product between $u_s$ and $\zeta_{n+3}^h$ like in \eqref{decomposition}: 
{\scriptsize
\begin{align*}
\label{decomposition_alpha_n+3}
    &\langle u_s,\zeta_{n+3}^{h}(x-z) \rangle_{H^{\gamma}}
    = \underbrace{\int_{\mathbb{R}^n}(-\Delta)^{\gamma}u_s\big[\zeta_{n+3}^h(x-z) - \widetilde{\zeta}_{n+3}^{h}(x-z)\big] dx}_{\widetilde{\phi}_1(z)} -e(n)\underbrace{\int_{\mathbb{R}^n}(-\Delta)^{\gamma+2}u_s\bigg[\widetilde{\zeta}_{n-1}^h(x-z)-|x-z|^{-n+1}\bigg]dx}_{\widetilde{\phi}_2(z)}\\
    & - e(n)\underbrace{\int_{\mathbb{R}^n}(-\Delta)^{\gamma+2}u_s|x-z|^{-n+1}dx}_{\widetilde{\phi}_3(z)} 
     +\underbrace{\int_{\mathbb{R}^n}\bigg[(-\Delta)^{\gamma}u_s\widetilde{\zeta}_{n-1}^h(x-z)+e(n)(-\Delta)^{\gamma+2}u_s\widetilde{\zeta}_{n+3}^h(x-z)\bigg]dx}_{\widetilde{\phi}_4(z)}
       \,.\notag
\end{align*}}

The estimates for $\widetilde{\phi}_i(z)$ $(1\le i \le 3)$ are basically the same as the above estimates for $\phi_i(z)$ $(i\le 1 \le 3)$, expect the minor differences in replacing the order of the fractional Laplacian from $\gamma+1/2$ to $\gamma + 3/2$ in the right hand side of \eqref{phi_2} and \eqref{phi_3}. For $\widetilde{\phi}_4(z)$, we can apply the Green's identity twice to derive
\begin{equation}\label{eq:orders}
    \widetilde{\phi}_4(z) = \frac{c_5(n,\gamma)}{h^{n+1}}\int_{\partial B(z,h)}(-\Delta)^{\gamma}u_sdx_s +\frac{1}{h^{n+2}}\int_{B(z,h)}(-\Delta)^{\gamma}u_sdx + \mathcal{O}(h^{-1})\,, 
 \end{equation}
where $c_5(n,\gamma)$ is a positive constant independent of $z$ and $h$. $\widetilde{\phi}_4(z)$ now still represents the average of $(-\Delta)^{\gamma}u_s$ over the neighborhood of $z$. We can conclude that the dominating term of the index function with $\alpha = n+1$ and $\alpha = n+3$ are approximately the same.
Although the order of $\widetilde{\phi}_4(z)$ in \eqref{eq:orders} with $\alpha = n+3$ is higher than $\phi_4(z)$ in \eqref{phi_4} with $\alpha = n+1$, we point out that the difference in order has minor influence on the accuracy of the reconstruction as the magnitude of $\phi_i(z)$ and $\widetilde{\phi}_i(z)$, $i=1,2,3$, are much smaller than both of $\phi_4(z)$ and $\widetilde{\phi}_4(z)$. Moreover, as we are particularly interested in reconstruction with noisy and inadequate measurement data, it is preferable to choose a probing function that is smoother and has a smaller $L^{\infty}$-norm. 

With the above considerations, in order to maintain the appropriate regularity of the probing function as well as to minimize numerical instability, we shall, from now on, only consider a choice of $\alpha$ in the range $\alpha \in (n, n+1]$. From numerical experiments, we do not observe much difference in the quality of numerical reconstruction for any choice of $\alpha \in (n, n+1]$, and therefore for simplicity, we always choose the probing function \eqref{def_probing} with $\alpha = n + 1$ instead of some other probing functions $\zeta_{\alpha}^h$ with $\alpha \ne n+1$.

\subsection{Verification of the index function for $0<\gamma\leq (n-1)/2$ and the frequency domain representation of the probing function}\label{sec_var_case2}
In this section, we shall first verify that our proposed index function $I_{\gamma}^h$ approximately recovers the target function $f$ when $0<\gamma\le(n-1)/2$, and then present a frequency domain representation of the function $\widetilde{\eta}_z^h$. This representation reveals the fact that the application of the probing function can be regarded as applying a low pass filter on the measurement data, which helps us better understand the importance and necessity of computing the duality product between the measurement data and the chosen probing function $\eta_z^h$. 

To verify that the index function can properly recover the target function $f$, we first recall the critical motivation for direct sampling type methods in \eqref{dual_G_relation}, that $u_s= d_nR^*Rf$ can be represented by the convolution of $f$ and a fast decaying kernel function $G_x(y) = 1/|x-y|$. It can be observed that as $G_x(y)$ is very large when $x\approx y$ and is relatively small otherwise. Hence, if we are given noisy or inadequate measurement data, $u_s$ is already an approximation to the target function $f$. Furthermore, considering the reconstruction by the proposed DSM with $\gamma \leq (n-1)/2$, we next show that our method can improve the approximation to the target function $f$ compared with the approximation provided by $u_s$ without applying any pseudo-differential operator on the noisy measurement data. 

Firstly, by the definition of the fractional Laplacian through a Fourier multiplier with a frequency variable $\omega$ and the Plancherel theorem, we derive
\begin{equation*}
\label{eqn_L2_converge}
    ||(-\Delta)^{\gamma}u_s- f||_{L^2(\mathbb{R}^n)} = \int_{\mathbb{R}^n}|(|\omega|^{2\gamma-(n-1)}-1)|^2|\mathcal{F}(f)|^2d\omega\,,
\end{equation*}
where $\mathcal{F}(f)$ denotes the Fourier transform of $f$. Therefore, we have
\begin{equation}
\label{eqn_L2_converge}
||(-\Delta)^{\gamma_2}u_s- f||_{L^2(\mathbb{R}^n)}\le||(-\Delta)^{\gamma_1}u_s- f||_{L^2(\mathbb{R}^n)} \le ||u_s -f||_{L^2(\mathbb{R}^n)} \quad \mbox{for~~} 0 < \gamma_1 < \gamma_2 < \frac{n-1}{2}\,.
\end{equation}
Combining \eqref{eqn_L2_converge} with the inversion formula $(-\Delta)^{\frac{n-1}{2}}u_s = f$ in  \eqref{inverse_formula}, we notice that the $L^2$-norm of $(-\Delta)^{\gamma}u_s- f$ coverages to $0$ as $\gamma \rightarrow (n-1)/2$. We can conclude that for $\gamma \in [0, (n-1)/2]$, as $\gamma$ becomes larger, $(-\Delta)^{\gamma}u_s$ recovers the target function $f$ more accurately with $u_s$ which is already a reasonable approximation to the target function $f$.

Now we recall the alternative characterization of the index function that we obtained through the discussion following \eqref{phi_4} in the previous section \ref{sec_alternative_chara}. The dominating term of the duality product  $\langle u_s,\eta_z^{h} \rangle_{H^{\gamma}(\mathbb{R}^n)}$ is the average of $(-\Delta)^{\gamma}u_s$ at the neighborhood of the sampling point $z$. Hence, this justifies that our index function can approximately recover $f$ due to the approximation property of $(-\Delta)^{\gamma}u_s$ and the alternative characterization of the index function.

We shall remark that, although computing $(-\Delta)^{\frac{n-1}{2}}u_s$ recovers $f$ exactly in the noise free case, the choice of $\gamma = (n-1)/2$ is not preferable in applications that we mentioned in the Introduction due to numerical instability. This theoretical prediction will also be justified in the following section \ref{sec_noise} and example 1 of section \ref{sec_num_eg}.

\medskip
In the remaining part of this subsection, we would like to investigate the frequency domain representation of the probing function. The main motivation for this part is that the discussion following \eqref{eqn_L2_converge} implies that the reconstruction solely with $(-\Delta)^{\gamma}u_s$ is already an approximation to the target function $f$. Therefore, it is necessary for us to justify that the introduction of the duality product and the probing function in the new DSM are essential in recovering the target function $f$ more stably. Firstly, by the definitions \eqref{def_inner_pro} and \eqref{index_fun}, the duality product allows us to avoid applying a pseudo-differential operator directly on the noisy measurement data $u_s$. Moreover, we will now show that our choice of the probing function induces a low pass filter in the frequency domain. For this reason, it helps improve the quality of reconstruction with noisy measurement data. To justify the low pass filtering property of the probing function, we consider the numerator of the index function in the frequency domain which yields
\begin{equation}
\label{frequency_domain}
    \langle u_s,\eta_z^{h} \rangle_{H^{\gamma}} = \int_{\mathbb{R}^n}(-\Delta)^{\gamma}u_s\widetilde{\eta}_z^hdx +  \int_{\mathbb{R}^n}(-\Delta)^{\gamma}u_s[\eta_z^h - \widetilde{\eta}_z^h]dx  = \mathcal{F}^{-1}\bigg\{\mathcal{F}(\widetilde{\eta}_0^h)\mathcal{F}((-\Delta)^{\gamma}u_s)\bigg\} + \phi_1(z)\,,
\end{equation} 
where $\mathcal{F}^{-1}$ denotes the inverse Fourier transform and $\phi_1(z)$ is defined in \eqref{decomposition} which is of the order $\mathcal{O}(h)$ by \eqref{phi_1}. The representation of the duality product in \eqref{frequency_domain} implies that the reconstruction by the proposed DSM can be regarded as applying the filtering function induced by $\widetilde{\eta}_0^h$ on $(-\Delta)^{\gamma}u_s$. Therefore, we now investigate the Fourier transform of $\widetilde{\eta}_0^h$ explicitly. 

In $\mathbb{R}^2$, it follows after converting it into a Hankel transform that
\begin{align}
   &\mathcal{F}\big(\widetilde{\eta}_0^h\big)(\omega) =  \int_0^h \frac{J_0(2\pi |\omega| r)r}{h^3}dr + \int_h^{\infty}\frac{J_0(2\pi|\omega| r)}{r^2}dr\\
   = &\frac{1}{4\pi^2|\omega|^2h^3}\int_0^{2\pi|\omega|h} J_0(t)tdt  + 2\pi|\omega|\bigg[\int_0^{\infty}\frac{J_0(t)-1}{t^2}dt + \int_{2\pi|\omega|h}^{\infty}\frac{1}{t^2}dt  + \int_0^{2\pi|\omega|h}\frac{1-J_0(t)}{t^2}dt\bigg]\,.\notag  
\end{align}
We notice the following integrals regarding Bessel functions of the first kind in \cite{handbook}:
\begin{align}
 &\int_0^x J_0(t)tdt = xJ_1(x)\,, \quad \int_{0}^{\infty}\frac{1-J_0(t)}{t^2}dt = 1\,, \\
  &\int_{0}^{x} \frac{1-J_0(t)}{t^2} = -\frac{1}{x} - \bigg[1 -  \frac{\pi x}{2}\bm{H}_0(x) \bigg]J_1(x) + \bigg[\frac{x^2+1}{x} - \frac{\pi x}{2}\bm{H}_1(x)\bigg]J_0(x) \,.\notag
\end{align}
where $\bm{H}_{\nu}$ is the Struve function of order $\nu$. Combining the above computations, we conclude that
\begin{equation}
	\mathcal{F}(\widetilde{\eta}_0^{h})(\omega) = \frac{1}{h}\frac{J_1(\lambda)}{\lambda} + \frac{\lambda}{h}\bigg[\bigg(\frac{\lambda^2+1}{\lambda} - \frac{\pi \lambda}{2}\bm{H}_1(\lambda)\bigg)J_0(\lambda) - \bigg(1-\frac{\pi\lambda}{2}\bm{H}_0(\lambda)\bigg)J_1(\lambda) - 1\bigg] \,,\quad \lambda = 2\pi|\omega|h\,.
\end{equation}

In the first plot of Fig.\,\ref{low_pass}, we plot the frequency domain representation of $\widetilde{\eta}_0^h$ with respect to $|\omega|$. $h$ is chosen as $0.1$, and we suppose the data is band-limited to $1/(2h)$.  
We observe that the frequency domain representation of $\widetilde{\eta}_0^h$ decays smoothly to $0$ as $|\omega|$ becomes larger. Hence, the probing function can be approximately considered as a low pass filter since it cuts off the high frequency component and smoothes the low frequency component of the measurement data in $\mathbb{R}^2$.

For the case $n=3$, we first consider the identity regarding the Fourier transform of a radial function in $\mathbb{R}^3$:
\begin{equation*}
    \int_{\mathbb{R}^3}f(|x|)e^{-2\pi i x\cdot\omega}dx = 2\pi\int_0^{\infty}\int_0^{\pi}f(r)e^{- 2\pi i r \cos\theta|\omega|}r^2d(-\cos\theta)dr= \frac{2}{|\omega|}\int_0^{\infty}f(r)r\sin(2\pi r |\omega|) dr\,,
\end{equation*}
then we can derive
\begin{equation*}
\begin{split}
	\mathcal{F}(\widetilde{\eta}_0^h)(\omega) = &\frac{2}{|\omega|}\bigg[\int_0^{h}\frac{r}{h^4}\sin(2\pi r|\omega|)dr + \int_{h}^{\infty}\frac{1}{r^3}\sin(2\pi r|\omega|)dr\bigg]\\
 =& \frac{2}{|\omega|}\bigg[\frac{1}{h^{4}4\pi^2|\omega|^2}\int_0^{2\pi |\omega|h}t\sin(t)dt +4\pi^2|\omega|^2\int_{2\pi|\omega|h}^{\infty}\frac{\sin(t)}{t^3}dt\bigg]\\
 =& \frac{4\pi}{h}\bigg[\frac{\sin(\lambda) -\lambda \cos(\lambda)}{\lambda^3} + \lambda\bigg(\frac{-\pi+2\text{Si}(\lambda)}{4} + \frac{\sin(\lambda)}{2\lambda^2} + \frac{\cos(\lambda)}{2\lambda}\bigg)\bigg]\,, \quad \lambda = 2\pi h|\omega|\,,
\end{split}
\end{equation*}
where $\text{Si}$ is the sine integral function. We now draw the frequency domain representation of $\widetilde{\eta}_0^h$ in the second plot of Fig.\,\ref{low_pass} with respect to $|\omega|$. We also assume $h = 0.1$ and the data is band-limited to $1/(2h)$, 
and the plot implies that the function $\widetilde{\eta}_0^h$ can still be regarded as a low pass filter since it cuts off the high frequency component and smoothes the low frequency component of the measurement data in $\mathbb{R}^3$.

\begin{figure}
 \begin{subfigure}[b]{0.499\textwidth}
    \centering
    \includegraphics[scale = 0.42]{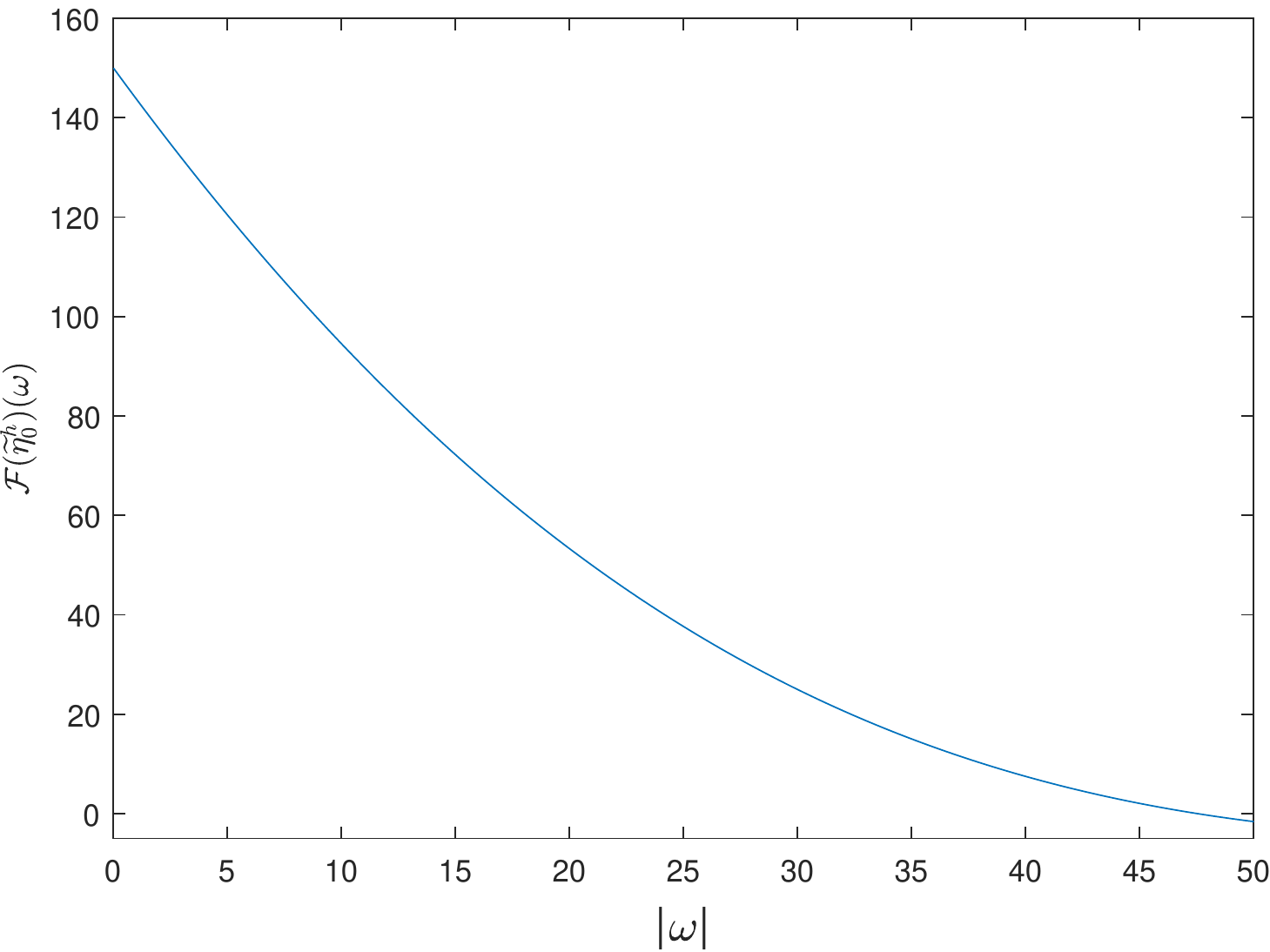}
    \end{subfigure}
 \begin{subfigure}[b]{0.499\textwidth}
    \centering
    \includegraphics[scale = 0.42]{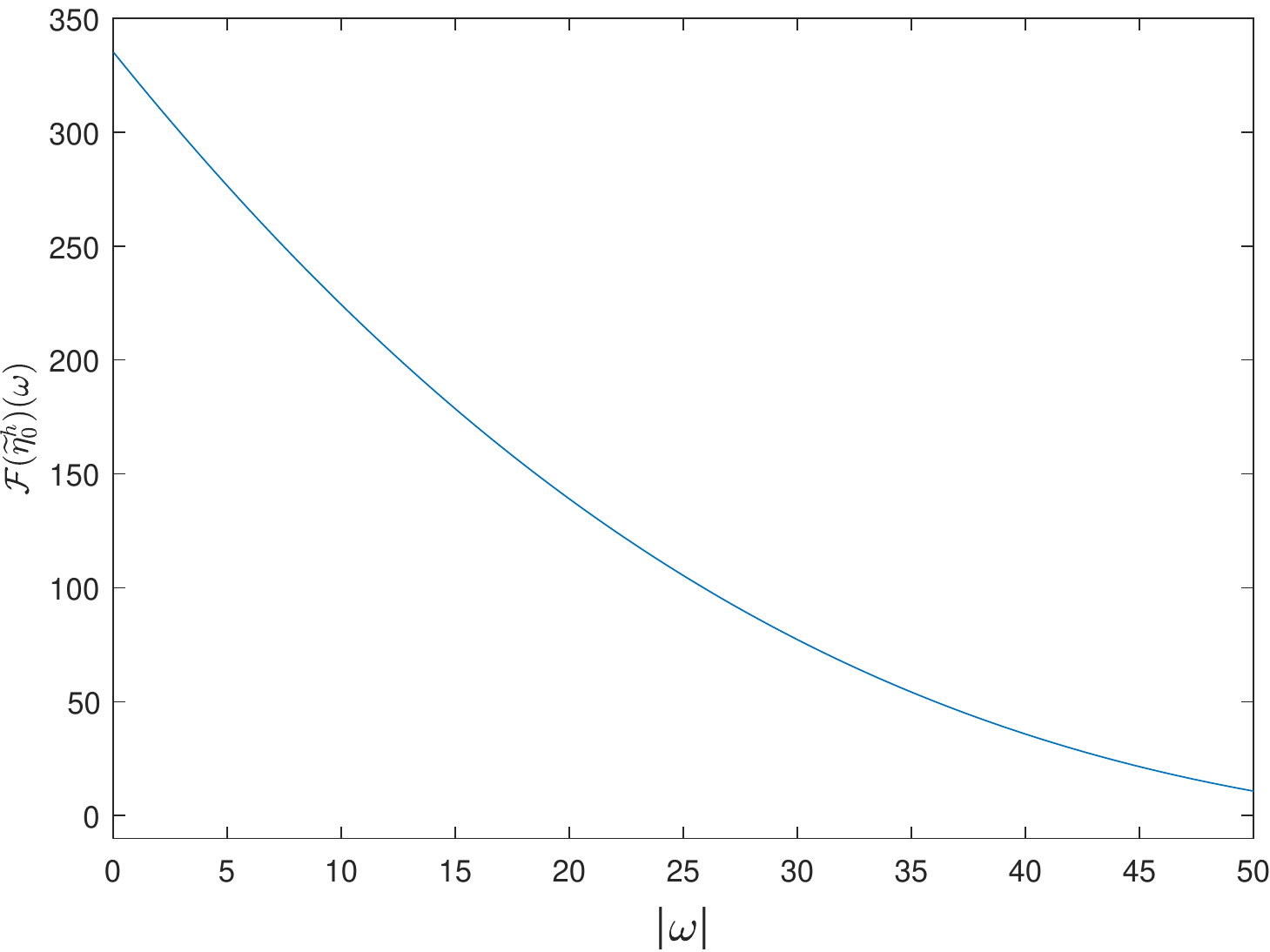}
    \end{subfigure}
     \caption{
   Frequency domain representation of $\widetilde{\eta}_0^h$ (cf.\,\eqref{def_zeta}) in $\mathbb{R}^2$ (left) and $\mathbb{R}^3$ (right), with the data being band-limited to $1/(2h)$ and $h = 0.1$.  The horizontal axis is $|\omega|$ and the vertical axis is 
$\mathcal{F}(\widetilde{\eta}_0^h)(|\omega|)$.}
\label{low_pass}
\end{figure}

To conclude, the crucial family of probing functions defined in \eqref{def_probing} for the new DSM allows our reconstruction to be very stable under highly noisy measurement data since the application of the probing function can be regarded as applying a low pass filter on the measurement data as illustrated in Fig.\,\ref{low_pass}.

\subsection{Relationship between the Sobolev scale and the variance of the index function}
\label{sec_noise}
In this subsection, we consider a particular noise model from \cite{principle_CT} that approximates the measurement process to showcase some close relationship between the Sobolev scale $\gamma$ in the duality product \eqref{def_inner_pro} and the variance of the index function for $\gamma \in (0,1)$. We only consider the case $\gamma<1$ since the $L^2$-norm of $(-\Delta)^{\gamma}\widetilde{\eta}_z^h$ is bounded for $\gamma < 1$ by part (c) of Lemma \ref{lemma_L1eta} and the boundedness of the $L^2$-norm is essential in our following discussion.

Suppose the collected data is polluted by a stationary zero-mean additive Gaussian noise, and the noise distribution is independent of each other for projections on different hyperplanes, namely, the noisy measurement takes the form:
\begin{equation}
    Rf(\theta,t) = Rf_{e}(\theta,t) + n(\theta,t)\,,
     ~~n(\theta,t) \sim N(0,\sigma_0^2)\,,
      ~~\mathbb{E}\big[n(\theta_1,t_1)n(\theta_2,t_2)\big] = \sigma_0^2\delta(\theta_1-\theta_2 )\delta(t_1-t_2)\,,
\end{equation}
where $\mathbb{E}$ represents the expectation operator, and $N(\mu,\sigma^2)$ stands for the normal distribution with mean $\mu$ and standard deviation $\sigma$. And $\delta$ is the delta measure and the subscript $e$ denotes the exact value. 

Recalling the numerator of our proposed index function in \eqref{index_fun}, we can rewrite it as
\begin{equation*}
\langle u_s,\eta_z^{h}\rangle_{H^{\gamma}(\mathbb{R}^n)}  = \int_{\mathbb{R}^n} (-\Delta)^{\gamma}u_s\widetilde{\eta}_z^h dx + \int_{\mathbb{R}^n} (-\Delta)^{\gamma}u_s [\eta_z^h - \widetilde{\eta}_z^h] dx \,.
\end{equation*}
 Since $Rf\in L^{\infty}(S^{n-1}\times \mathbb{R})$, the expectation of the product of measurements is given by 
 \begin{equation}
\begin{split}
    \mathbb{E}\big[u_s(x)u_s(y)\big] 
    =& \int_{S^{n-1}\times S^{n-1}}\bigg[Rf_{e}(\alpha,\alpha\cdot x)Rf_{e}(\beta,\beta\cdot y) + \sigma_0^2\delta(\alpha-\beta)\delta(\alpha\cdot x- \beta\cdot y)\bigg] d\alpha d\beta \\
   = &\,\,  u_e(x)u_e(y) + \big|S^{n-1}\big|^2\sigma_0^2\delta(x-y)\,,
\end{split}
\end{equation}
where $u_e$ represents the exact value, and $u_s$ is the measurement data with noise. From the above, we see the variance of the index function at $z\in \Omega\,$:
\begin{equation}
\label{def_var_z}
\begin{split}
 \sigma^2_{\gamma}(z) = \mathbb{E}\big[(I_{\gamma}^h(z))^2\big]- \mathbb{E}\big[I_{\gamma}^h(z)\big]^2\,, \quad 
  \text{ with }\,\, \mathbb{E}\big[I_{\gamma}^h(z)\big] =\frac{\int_{\mathbb{R}^n}u_s (-\Delta)^{\gamma}\widetilde{\eta}_z^h dx + \phi_1(z)}{n(z)} \,,
 \end{split}
\end{equation}
where $\phi_1(z)$ is defined in \eqref{decomposition}, and the order of $\phi_1(z)$ is $\mathcal{O}(h)$ as we know from the estimate in \eqref{phi_1}. By part (c) of Lemma \ref{lemma_L1eta}, $(-\Delta)^{\gamma}\widetilde{\eta}_z^h$ belongs to $L^2(\mathbb{R}^n)$. Then one can derive the relationship between the variance of the index function and the Sobolev scale $\gamma$:
\begin{equation}
\label{var_z}
    \begin{split}
        &\big[n(z) \sigma_{\gamma}(z)\big]^2 + \mathcal{O}(h)\\ = &\,\,\int_{\mathbb{R}^{n}\times\mathbb{R}^n}\bigg[\mathbb{E}\big[u_s(x)u_s(y)\big](-\Delta)^\gamma\widetilde{\eta}_z^{h}(x)(-\Delta)^\gamma\widetilde{\eta}_z^{h}(y)\bigg]dxdy -\bigg[\int_{\mathbb{R}^n}\mathbb{E}[u_s](-\Delta)^{\gamma}\widetilde{\eta}_z^{h}dx\bigg]^2 \\
        =&\,\,\sigma_0^2\big|S^{n-1}\big|\int_{\mathbb{R}^n}|(-\Delta)^{\gamma}\widetilde{\eta}_z^{h}|^2dx = \sigma_0^2\big|S^{n-1}\big|\int_{\mathbb{R}^n}|\omega|^{4\gamma}|\mathcal{F}(\widetilde{\eta}_z^h)(\omega)|^2 d\omega\,.
         \end{split}
\end{equation}
We now substitute the representation of the normalization term $n(z)$ defined in \eqref{index_fun} into \eqref{var_z}. In Fig.\,\ref{fig_variance_gamma}, assuming $\Omega = [-0.5,0.5]^n$  for $n=2$ (left) and $n=3$ (right) with $h = 0.025$, we plot the nature logarithm of the variance of the index function at the origin, i.e., $\ln(\sigma^2_{\gamma}(0))$, with respect to $\gamma \in [0.2,0.975]$ where the step size of $\gamma$ equals to $0.025$. The constant $\sigma_0$ in \eqref{var_z} is chosen such that $\max_{\gamma\in [0.2,0.975]}\sigma_{\gamma}^2(0) = 1$ for all $\gamma$. Our computation only considers $\gamma \geq 0.2$ is due to \eqref{eqn_L2_converge} which implies the accuracy of the reconstruction is not satisfactory for relatively small $\gamma$. From Fig.\,\ref{fig_variance_gamma}, for both reconstructions in $\mathbb{R}^2$ and $\mathbb{R}^3$, the variance of the index function increases exponentially with respect to $\gamma$. Hence, we shall not consider the possibility of very large $\gamma$, i.e., $\gamma \ge n/2$, in real applications. This conclusion is also consistent with the motivation of DSM in section \ref{sec_principle} that we expect a smaller choice of the Sobolev scale $\gamma$ will improve the robustness of the reconstruction under high level of random noise.
\begin{figure}
 \begin{subfigure}[b]{0.499\textwidth}
    \centering
    \includegraphics[scale = 0.42]{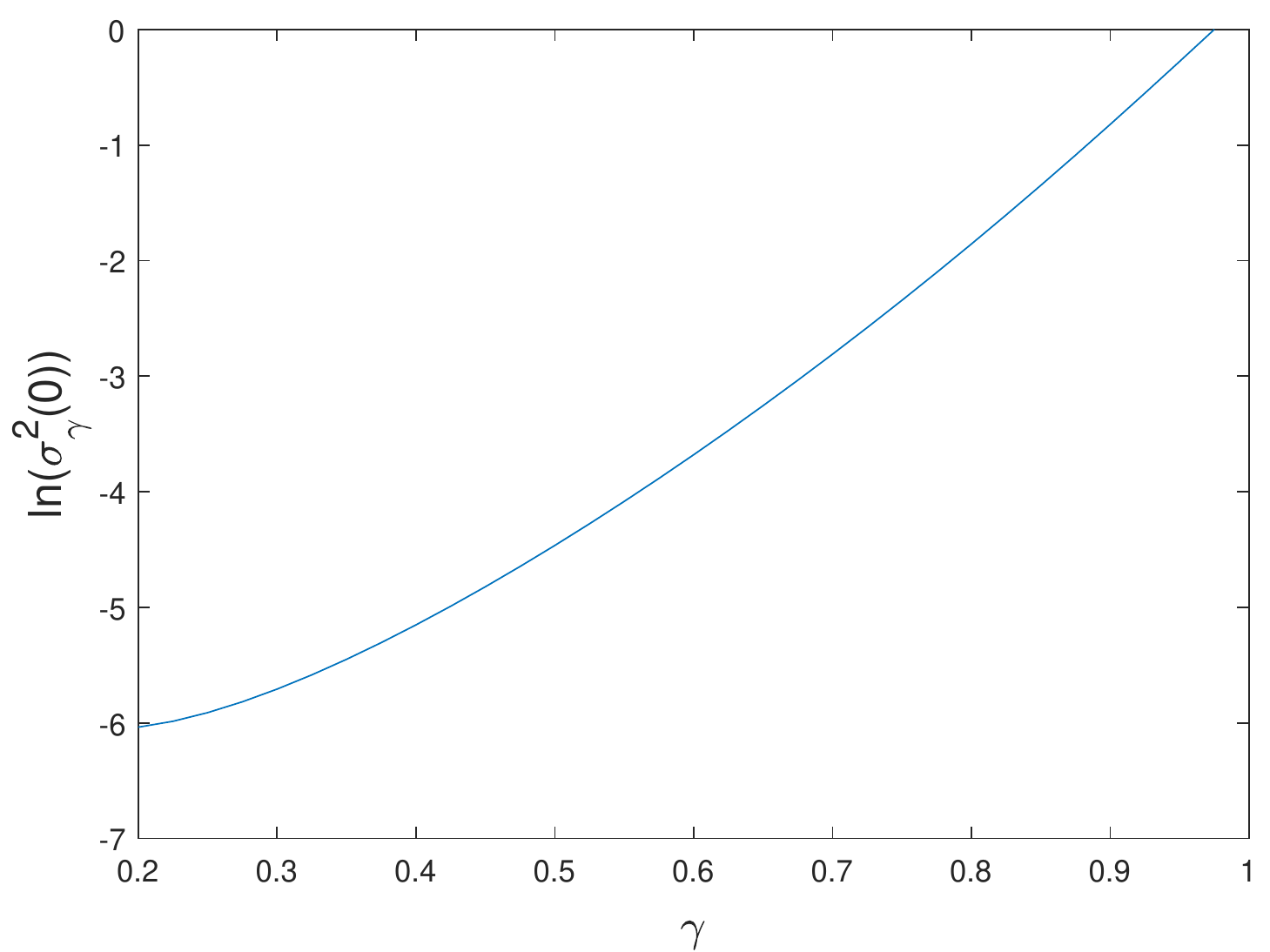}
    \end{subfigure}
 	\begin{subfigure}[b]{0.499\textwidth}
    \centering
    \includegraphics[scale = 0.42]{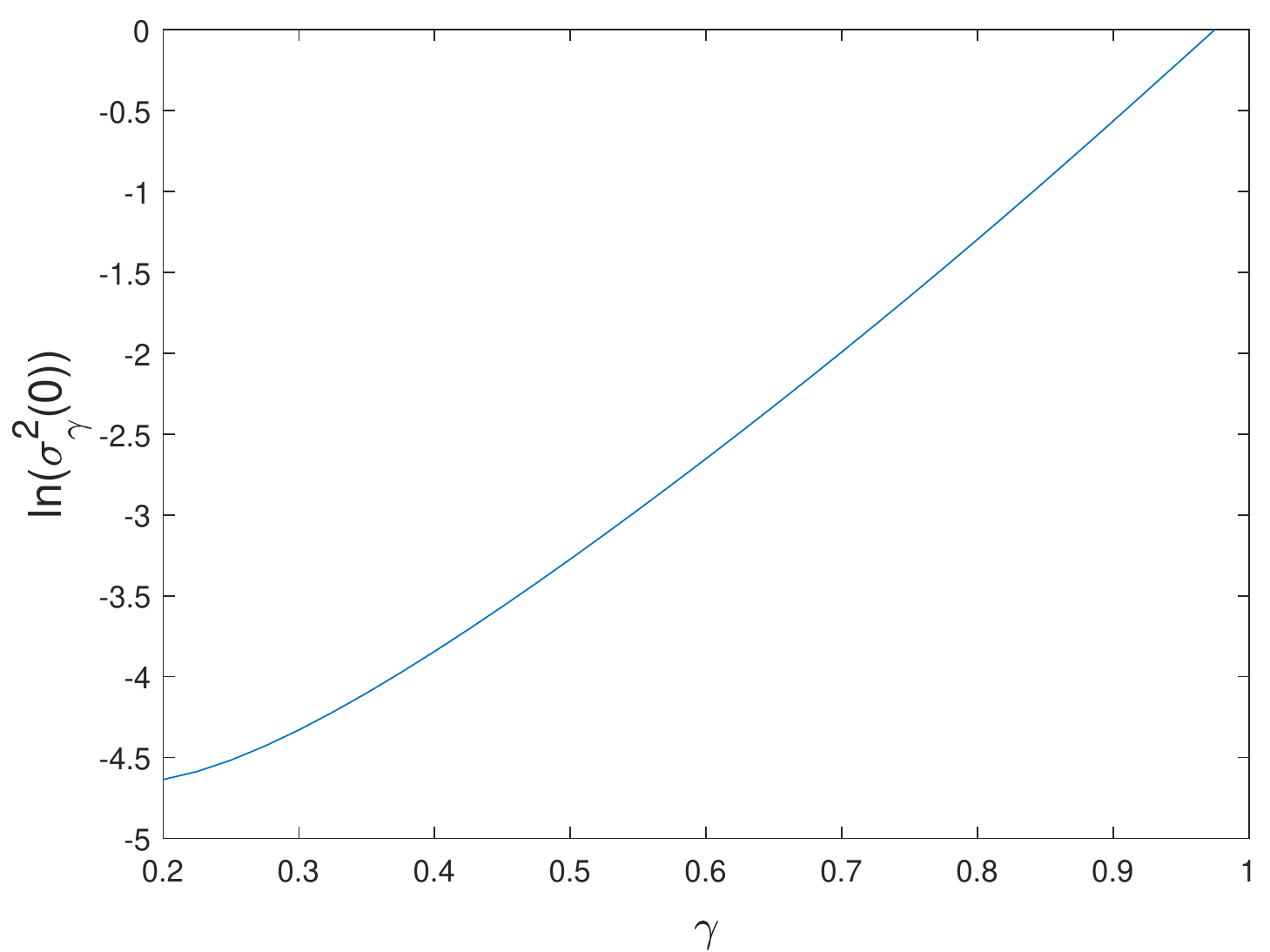}
    \end{subfigure}
    \caption{{$\ln(\sigma^2_{\gamma}(0))$ (cf.\,\eqref{def_var_z}) with respect to $\gamma \in [0.2,0.975]$ in $\mathbb{R}^2$ (left) and in $\mathbb{R}^3$(right).}}
\label{fig_variance_gamma}
\end{figure}

We end this section with a summary of our theoretical predictions on the appropriate choice of the Sobolev scale $\gamma$ for applications, based on the discussions in section \ref{sec_var_case2} and \ref{sec_noise}. For the reconstruction from noisy and inadequate measurement data that we are particularly interested in, we shall choose $\gamma$ that is relatively small considering the relationship revealed in Fig.\,\ref{fig_variance_gamma}. Moreover, for $\gamma>0$ but is much smaller than $(n-1)/2$, from \eqref{eqn_L2_converge} in section \ref{sec_var_case2}, the reconstruction results may not be accurate enough. Hence, we may prefer to choose $\gamma$ that is close to $(n-1)/2$ for our DSM, for instance, $\gamma = 0.4$ in $\mathbb{R}^2$ and $\gamma = 0.9$ in $\mathbb{R}^3$. This theoretical prediction will be verified in example 1 of section \ref{sec_num_eg}.

\section{DSMs for some other tomography problems}\label{sec_other}
\subsection{Limited angle tomography}
\label{section_Limited_Angle}
As we pointed out in section \ref{sec_verification}, the new DSM is expected to be robust against noise, due to the property of the probing function as a low pass filter and the choice of the duality product which avoids applying any pseudo-differential operator on the noisy measurement data. Therefore, we are interested in whether the DSM also performs reasonably in the limited angle tomography, which is another closely related and highly ill-posed inverse problem associated with the Radon transform. 

We will focus on the two-dimensional case when detectors are distributed in the parallel geometry. Recall the Radon transform for a smooth target function $f$:
\begin{equation*}
    Rf(\theta,t) = \int_{x\cdot \theta = t}f(x)dx_L\,, \quad \theta \in S^{1}\,, \quad t\in\mathbb{R}\,,
\end{equation*} 
where we assume $\theta \in [-\Phi,\Phi]$ for $\Phi<\pi/2$, and $s \in I_{\theta}$, with $I_{\theta}$ being the same as in \eqref{def_Itheta}. In this case, the dual of the Radon transform with limited angle measurement is given by 
\begin{equation}
\label{parral_back}
   R^*_{\Phi}g (x) := \int_{\mathbb{S}^1}g(\theta,x\cdot \theta)\mathcal{X}_{V_{\Phi}}(\theta)d\theta\,,  \quad V_{\Phi} = [-\Phi,\Phi]\cup [\pi-\Phi,\pi +\Phi] \,, 
  \end{equation}
where $\mathcal{X}_{V}(\theta) = 1$ if $\theta \in V$ and $\mathcal{X}_{V}(\theta) = 0$ otherwise.

The limited angle tomography will be very different from the case where we have measurements from all directions as in section \ref{sec_principle}.
In particular, the extra discontinuity of the characteristic function in \eqref{parral_back} will create undesirable artifacts when we apply a pseudo-differential operator on the measurement data, including the proposed fractional Laplacian.  A classification of artifacts was deduced in the work \cite{reduction_WF} with an argument using microlocal analysis and the wavefront set. 

If we employ the same index function as in equation \eqref{index_fun}, with $u_s$ replaced by $u_s^{\Phi} = 1/2 R^*_{\Phi}Rf$, the numerator of the index function becomes
\begin{equation}
   \langle u_{s}^{\Phi}, \eta_z^{h} \rangle_{H^{\gamma}(\mathbb{R}^2)} =\int_{\mathbb{R}^2}u_s^{\Phi}(-\Delta)^{\gamma}\eta_z^h(x)dx\,.
 \end{equation}

We now consider the above duality product in the frequency domain. 
Recalling the Fourier slice theorem, i.e., $\mathcal{F}_t(Rf)(\theta,t) = \mathcal{F}(f)(t\theta)$ \cite{math_of_CT}, where $\mathcal{F}_t$ is the one-dimensional Fourier transform with respect to $t$, we can rewrite $u_s^{\Phi}$ as
\begin{equation}
 \begin{split}
   u_s^{\Phi}(x) = & \int_{\mathbb{S}^1} \mathcal{F}^{-1}_t\{\mathcal{F}_tRf\}(\theta,x\cdot\theta) \mathcal{X}_{V_{\Phi}}(\theta)d\theta 
   =  \int_{\mathbb{S}^1} \int_{\mathbb{R}}\mathcal{F}(f)(\theta\tau)e^{2\pi i\tau(x\cdot \theta) }d\tau \mathcal{X}_{V_{\Phi}}(\theta)d\theta \\
   =& \int_{\mathbb{R}^2} \mathcal{F}(f)(\omega) \frac{\mathcal{X}_{V_{\Phi}}(\omega/|\omega|)}{|\omega|}e^{2\pi i\omega \cdot x} d\omega
   = \bigg(f\ast \mathcal{F}^{-1}\bigg(\frac{\mathcal{X}_{V_{\Phi}}(\omega/|\omega|)}{|\omega|}\bigg)\bigg)(x)\,.
  \end{split}
 \end{equation}
Hence, the duality product between the measurement data and the probing function with a small choice of the Sobolev scale $\gamma < 1/2$ becomes
\begin{equation}
\label{Limited_anlge_fourier}
    \langle u_{s}^{\Phi}, \eta_z^{h} \rangle_{H^{\gamma}(\mathbb{R}^2)} = \mathcal{F}^{-1}\bigg(\mathcal{X}_{V_{\Phi}}(\omega/|\omega|)|\omega|^{2\gamma-1}\mathcal{F}(\widetilde{\eta}_0^h)\mathcal{F}(f) \bigg)(z) + \int_{\mathbb{R}^2}(-\Delta)^{\gamma}u_s^{\Phi}(\widetilde{\eta}_z^h - \eta_z^h )dx \,,
\end{equation}
where the order of the second integration in the right hand side is $\mathcal{O}(h)$ with an estimate similar to \eqref{diff_smooth}. In this case, the duality product with a relatively small Sobolev scale  combining with the probing function will serve as a low pass filter in the frequency domain to improve the numerical stability of the reconstruction. 

Similarly to \cite{reduction_WF}, we will further employ the modified back projection operator to improve the accuracy of the reconstruction:
\begin{equation}
\label{def_dual_pat}
   \widetilde{R}^*_{\Phi}g (x) := \int_{\mathbb{S}^1}g(\theta,x\cdot \theta)\Psi_{V_{\Phi}}(\theta)d\theta\,, 
 \end{equation}
 where $\Psi_{V_{\Phi}}$ is defined as
\begin{equation*}
  \Psi_{V_{\Phi}}(\theta) = \begin{cases}
        1\,, \, &\theta\in[-\Phi,\Phi]\cup [\pi - \Phi,\pi)\cup[-\pi,-\pi+\Phi]\,;\\
        1-\frac{|\theta| - \Phi}{\lambda}\,, \, &\theta \in[-\Phi-\lambda,-\Phi)\cup (\Phi,\Phi+\lambda]\,;\\
        1-\frac{(\pi-\Phi)-|\theta|}{\lambda}\,, \, &\theta\in[\pi-\Phi-\lambda, \pi -\Phi) \cup (-\pi+\Phi, -\pi+\Phi+\lambda]\,;\\
        0\,, \, &\text{otherwise}\,.
    \end{cases}
\end{equation*}
where $\lambda$ is a fixed value representing the range of data that is smoothed. This modified back projection operator basically extends the original characteristic function $\mathcal{X}_{V_{\Phi}}$ to a linear function. The realization of $\Psi_{V_{\Phi}}$ is achieved by a direct extension of the measurement data. 

To conclude, the index function for reconstructing $f$ at the sampling point $z$ reads now as 
\begin{equation}
\label{index_fun_limited}
	I^{\gamma}_h(z) = \frac{\langle u_s^{\Phi}, \eta_z^h\rangle_{H^{\gamma}(\mathbb{R}^2)}}{n(z)}\,\quad\mbox{with~~} u_s^{\Phi} = \frac{1}{2}\widetilde{R}^*_{\Phi}Rf\quad \mbox{and~~} n(z) =\langle   \frac{1}{2}R^*R(\mathbbm{1}_{\Omega}),\eta_z^h \rangle_{H^{\gamma}(\mathbb{R}^2)}\,.
\end{equation}
We shall demonstrate the robustness of the DSM in this limited angle tomography numerically in  example 5 of section \ref{sec_num_eg}.

Furthermore, our above discussion applies to the case that the measurement is in the fan beam geometry when the range of measurement angles is limited. The proposed index function can be employed after replacing $u_s^{\Phi}$ in \eqref{index_fun_limited} by the back-projected data obtained from the fan beam measurement. 
 
\subsection{Exponential Radon transform}
We now discuss the application of the DSM to a special inverse problem of the exponential Radon transform. The exponential Radon transform appears in the radionuclide imaging and can be regarded as a generalization of the Radon transform \cite{exponential}. 

First, assuming $f$ is smooth and compactly supported in $\Omega$, we denote $T_{\mu}f(\theta, t)$ and $T_{\nu}^{\ast}g(x)$ as
\begin{equation}
\label{def_exp}
    T_{\mu}f(\theta,t) := \int_{\mathbb{R}^n}f(x)e^{\mu x \cdot\theta^{\perp}}\delta(t-x\cdot \theta)dx\,,\quad 
    T_{\nu}^{\ast}g(x) := \int_{S^{n-1}}g(\theta,\theta\cdot x)e^{\nu x \cdot \theta^{\perp}}d\theta\,,
\end{equation}
for $x\in \mathbb{R}^n$, $\theta\in S^{n-1}$, and $t\in\mathbb{R}$. We note that $\theta^{\perp}$ can be defined through a fixed rotation rule, for instance, rotating $\theta$ clockwise for $\pi/2$ in $\mathbb{R}^2$. The Radon transform is a special case of \eqref{def_exp} with $\mu = 0$. With a change of variable, the measurement data after back projection becomes
\begin{equation}
    u^{\mu}(x) := T_{\mu}^{\ast}T_{-\mu}f(x) = \int_{\mathbb{R}^n}\int_{S^{n-1}}f(y)e^{\mu (y-x)\cdot \theta^{\perp}}\delta(y\cdot\theta - x\cdot \theta) d\theta dy = \int_{\mathbb{R}^n}f(y)\frac{e^{\mu|x-y|}}{|x-y|}dy \,.    
\end{equation}

Considering a special case of the exponential Radon transform, that is, $n=3$ and $\mu = ik$ with $k>0\,$:
\begin{equation}
	u^{ik}(x) = T_{ik}^*T_{-ik}f(x) =\int_{\Omega}f(y)\frac{e^{ik|x-y|}}{|x-y|}dy = (f\ast \tilde{G}_0)(x) \,,
\end{equation}
where $\tilde{G}_x$ satisfies $(\Delta+k^2)\tilde{G}_x = 4\pi\delta_x$. Hence, an inversion formula for the measurement is 
\begin{equation}
	f(x) = (4\pi)^{-1}(\Delta+k^2)u^{ik}(x)\,.
\end{equation}
We observe that, with the index function defined in \eqref{index_fun}, $f(x)$ can be reconstructed by employing $\tilde{I}(x) = I_{\gamma}^h(x) + k^2u^{ik}(x)$ and all our early discussions could be extended to this scenario.

\section{Numerical implementations}\label{sec_implementation}
In this section, we introduce some numerical implementations of the proposed DSM, especially the evaluation of the duality product \eqref{def_inner_pro} between the measurement data and the probing function. 
With several strategies that are employed to reduce the computational time of our method, we will compare the computational complexity of DSM with the popular FBP method.

We first recall the definition of the index function in \eqref{index_fun}, since both $(-\Delta)^{\gamma}\eta_z^h$ and $f$ are contained in $L^2(\mathbb{R}^n)$, the numerator of $I_{\gamma}^h(z)$ can be written as 
\begin{align}
	\langle u_s, \eta_z^h \rangle_{H^{\gamma}(\mathbb{R}^n)} 
	=&\, d_n\int_{\mathbb{R}^n} R^*Rf(x)(-\Delta)^{\gamma}\eta_z^h dx 
	= d_n \int_{\mathbb{R}^n}\bigg[\int_{\mathbb{S}^{n-1}}Rf(\theta, x \cdot \theta) d\theta)	 
	\bigg](-\Delta)^{\gamma}\eta_z^h(x)dx \nonumber\\ 
	=& \,d_n\int_{\mathbb{S}^{n-1}} \bigg[\int_{\mathbb{R}^n}Rf(\theta,x\cdot\theta)  
	(-\Delta)^{\gamma}\eta_z^h(x) dx  \bigg]d\theta \,.  \label{numerical_formulation_index}
\end{align}

Now we investigate more carefully the integration of the product between the Radon transform of the target function $f$ and the fractional Laplacian of the probing function in $\mathbb{R}^n$. We first notice that if supp$\{f\} \subseteq \Omega \subseteq B(0,r_2)$, then $Rf(\theta,t) = 0$ for $|t|>r_2$. With this observation, we know the integral part with respect to $\theta$ in \eqref{numerical_formulation_index} equals to
\begin{equation}
\label{def_Htz}
\begin{split}
 &\int_{\mathbb{R}^n}Rf(\theta,x\cdot\theta)(-\Delta)^{\gamma}\eta_z^h(x)dx = \int_{\mathbb{R}} \int_{x\cdot\theta = t}(-\Delta)^{\gamma}\eta_z^h(x) Rf(\theta,t) dx dt \\
	=& \int_{|t|<r_2}R((-\Delta)^{\gamma}\eta_0^h)(\theta_0,t-z\cdot\theta)Rf(\theta,t)dt
	= \int_{|t|<r_2}(-\Delta_{t-\tau})^{\gamma}R(\eta_0^h)(\theta_0,t-\tau)Rf(\theta,t)dt\,
\end{split}
\end{equation}
for a fixed angle $\theta_0$, where we have employed in the second equality of \eqref{def_Htz} the following property regarding the Radon transform for an arbitrary radial function $g_0 \in L^2(\mathbb{R}^n)$ that satisfies $g_0(x) = g_0(|x|)$ and $g_z(x) = g(x-z)\,$:
$$R(g_z)(\theta,t) = R(g_0)(\theta_0,t-z\cdot\theta)\,$$
for a fixed angle $\theta_0$ and arbitrary angles $\theta$. The last equality in \eqref{def_Htz} holds due to the intertwining property between the fractional Laplacian and the Radon transform, which can be derived through the Fourier slice theorem, i.e., $\mathcal{F}_t(Rf)(\theta,t) = \mathcal{F}(f)(t\theta)$, and the representation of the fractional Laplacian through a Fourier multiplier.

For the notational sake, we define $H(\theta,\tau) := \int_{|t|<r_2}(-\Delta_{t-\tau})^{\gamma}R(\eta_0^h)(\theta_0,t-\tau)Rf(\theta,t)dt$. Then \eqref{numerical_formulation_index} can be computed by
\begin{equation*}
	\langle u_s, \eta_z^h\rangle_{H^{\gamma}(\mathbb{R}^n)} = d_n R^*(H(\theta, \tau))(z)\,.
\end{equation*}  

To summarize, the implementation of the DSM for reconstructing the target function $f$ consists of the following steps: 

\begin{itemize}
	\item In the off-line computation, for a set of discrete sampling point $z_j \in \Gamma_z \subset \Omega\subseteq B(0,r_2)$, we take $h = \min_{z_{i},z_{j}\in \Gamma_{z}}|z_{i}-z_{j}|$. Then we choose a set of uniformly distributed points
$$\Gamma_{\tau} = \{\tau_k = -r_2 + hk\,;\,\, hk<2r_2+h\,, \,\, k \in \mathbb{N}\} \subset \mathbb{R}\,,$$
 and compute $(-\Delta_{\tau})^{\gamma}R(\eta_0^h)(\theta_0, \tau_k)$ with $\theta_0 = 0$ and $\tau_k\in \Gamma_{\tau} \cup r_2 + \Gamma_{\tau} \cup -r_2 + \Gamma_{\tau}$. Finally, for each sampling point $z_j$, we compute $n(z_j)$ defined in \eqref{index_fun}.

\item Given the measurement data $Rf(\theta_i,t_j)$ with measurement angles $\theta_i \in \Gamma_{\theta} \subset \mathbb{S}^{n-1}$ and discrete measurement points $t_j \in \Gamma_t(\theta) \subset I_\theta \subset \mathbb{R}$ defined by \eqref{def_Itheta}:
\begin{enumerate}
	\item For each $\theta_i\in \Gamma_{\theta}$, $\tau_k \in \Gamma_{\tau}$, we compute 
	\begin{equation}
	\label{Htz_dis}
	H(\theta_i,\tau_k) = h\sum_{j}(-\Delta_{t-\tau})^{\gamma}R(\eta_0^h)(\theta_0, t_j - \tau_k)Rf(\theta_i, t_j) \,;
	\end{equation} 
	\item For each sampling point $z_j$, we apply the back-projection operator $R^*$ on $H(\theta_i,\tau_k)$ to obtain $\langle u_s, \eta_{z_j}^h\rangle_{H^{\gamma}(\mathbb{R}^n)}$. Then we divide it by $n(z_j)$ to obtain the index function $I^h_{\gamma}(z_j)$ which recovers the target function $f(z_j)$.
\end{enumerate}
\end{itemize}

{\bf Comparison between computational complexities of DSM and FBP}. 
We now recall the implementation of the FBP method, which applies the ramp filter composed with a proper low pass filter on the $t$ variable of $Rf(\theta,t)$, and then back-projecting it to recover $f$. In general, for the standard case that measurement points $t_j$ are uniformly distributed, the step of filtering in an FBP reconstruction requires $\mathcal{O}(N\log N)$ flops for $N$ discretization points. Considering the computational complexity of our DSM, except for the step of back-projection that we share with the FBP method, the method only requires two extra steps. The first is to compute $H(\theta,\tau)$ with \eqref{Htz_dis}. In this step, we can observe that the matrix representation of $(-\Delta_{t-\tau})^{\gamma}R(\eta_0^h)(\theta_0,t_j-\tau_k)$ is a Toeplitz matrix since the value of entries only depend on $t_j-\tau_k$. Hence, with the fast Fourier transform, the computation of \eqref{Htz_dis} costs $\mathcal{O}(N\log N)$ flops. The second extra step required by DSM is to divide the duality product by the normalization term $n(z_j)$ which only costs $\mathcal{O}(N)$ flops. 
To conclude, the overall computational complexity of the DSM is of the same order as the traditional FBP method. However, as we shall observe from a series of numerical experiments in section \ref{sec_num_eg}, DSM provides more robust and accurate reconstructions.
We like to mention that the traditional methods which yield reasonable reconstructions in those challenging situations have much higher computational complexities, for instance, they often involve minimizing a functional with certain regularization \cite{sparse_tomo}.

\section{Numerical experiments}\label{sec_num_eg}
A series of numerical experiments are carried out in this section to illustrate the robustness and accuracy of the novel DSM for a number of representative applications in two and three dimensions.
For two-dimensional experiments, we take the sampling domain $\Omega = [-0.5,0.5]\times [-0.5,0.5]$, with the mesh size {$h = 2^{-1}\times 10^{-2}$. Detectors are placed in parallel arrays and the angular increment is {$0.25$} degree except for Example 4.

The Radon transform of the target function $f$ supported in $\Omega$ is available at a set of discrete angles $\Gamma_\theta$, which are uniformly distributed in $[-\pi/2,\pi/2)$ (except for Example 5) and at discrete points $\Gamma_t(\theta) \subset I_{\theta}$ defined by \eqref{def_Itheta}. Two original images are examined, with the first one being an image containing four objects with different shapes, and the second one being the classical head phantom image.

A stationary additive Gaussian random noise is added to the Radon transform of $f$ in all experiments: 
\begin{equation}\label{eq:data}
    Rf_s(\theta,t) := Rf_e(\theta,t)+\epsilon \delta\,, \quad \theta\in \Gamma_{\theta}\,,\quad t\in \Gamma_t(\theta)\,,
\end{equation}
where $\epsilon$ is the standard normal distribution, $Rf_e$ is the exact data, and $\delta = \text{mean}(Rf_e)\times(\text{noise level})$. We will also investigate the reliability of the proposed DSM under another type of random noise, i.e., the 'salt and pepper' noise, which corresponds to the dysfunction of detectors. This type of noise can be caused by mechanical issues or sudden disturbances on detectors, and as a result, a certain portion of data will be corrupted. The noise level, in this case, represents the percentage of the measurement data that is incorrect, and the incorrect data is randomly set to be the minimum or the maximum of all available data in a particular experiment.

In each of the following examples, we first generate the exact measurement data $Rf_{e}(\theta,t)$ and then impose the noise on the exact data as in \eqref{eq:data} to obtain $Rf_{s}(\theta,t)$. Then the index function \eqref{index_fun} is evaluated with the basic computational strategies introduced in section \ref{sec_implementation}. 
To compare the DSM with some existing methods, we choose the FBP method with the 'Hamming' filter for reconstructing the image in the \textsc{MATLAB} R2019B. This corresponds to adding the Hamming window on the classical ramp filter.

To compare the numerical reconstruction qualities, we compute the discrete $L^2$-norm error and $L^{\infty}$-norm error of the reconstruction. We write by $I_{DSM}$ and $I_{FBP}$ the images reconstructed by the new DSM and the FBP method, respectively, and by $I_{O}$ and $\overline{I}_{O}$ the original image and its average in $\Omega$, respectively. We further define
\begin{equation}
\label{def_L2_SNR}
\Err^2_{DSM} :=   \frac{||I_{DSM} -I_{O}||_{2}}{||I_O||_2}\,,\quad 
\Err^{\infty}_{DSM} :=   \frac{||I_{DSM} -I_{O}||_{\infty}}{||I_O||_\infty}\,,
\end{equation}
where $||\cdot||_{2}$ and $||\cdot||_{\infty}$ denote the discrete $L^2$-norm and $L^{\infty}$-norm. Similar quantities are also computed for the FBP method and denoted by $\Err^2_{FBP}$ and $\Err^{\infty}_{FBP}$.

To fairly compare the reconstruction quality of DSM and FBP, we plot the normalized index function $\tilde{I}_{DSM}(z) = I_{DSM}(z)/\max_{y\in \Omega}|I_{DSM}(y)|$ and $\tilde{I}_{FBP}(z) = I_{FBP}(z)/\max_{y\in \Omega}|I_{FBP}(y)|$ in each plot. 
In all the figures, images in the same row are generated with the same measurement data to demonstrate certain numerical phenomena; Plots with subtitles 'DSM', 'FBP', and '$f(x)$' plot $\tilde{I}_{DSM}(z)$, $\tilde{I}_{FBP}$, and the original image being recovered.

\textbf{Example 1.} 
We examine in this example the influence of the Sobolev scale $\gamma$ (cf.\,\eqref{def_inner_pro}) and parameter $\alpha$ (cf.\,\eqref{def_probing}) on the reconstruction to validate our previous theoretical predictions and also to provide some important practical guidance on their choice for the subsequent examples.  Reconstructions by DSM (with $\gamma = 0.3$, $0.4$, $0.5$, $0.6$) and reconstructions by DSM (with $\alpha=3$, $4$, $5$) are shown in Fig.\,\ref{example1}. 

We compute the four images in the first row of Fig.\,\ref{example1} with the same measurement data under different choices of $\gamma =0.3$, $0.4$, $0.5$, and $0.6$ with $\alpha = 3$ and $20\%$ additive Gaussian noise. We may observe that the reconstruction is sharper but less stable as $\gamma$ increases. Denoting $\Err^2_{\gamma = \lambda}$ as the discrete $L^2$-norm error of the reconstruction by DSM with $\gamma = \lambda$ as in \eqref{def_L2_SNR}, then the corresponding reconstruction errors are given by
$$
\Err^2_{\gamma = 0.3} = 0.171 \,,\quad
\Err^2_{\gamma = 0.4} = 0.135 \,,\quad
\Err^2_{\gamma = 0.5} = 0.153 \,,\quad
\Err^2_{\gamma = 0.6} = 0.342 \,.
$$
The above numerical results follow from our previous theoretical conclusions at the end of section \ref{sec_noise} that we expect a smaller $\gamma$ will provide more stable reconstruction results with noisy measurement data, i.e., comparing $\gamma = 0.4$, $0.5$, and $0.6$; but at the same time, the reconstruction is not accurate enough for $\gamma$ that is too small, i.e., comparing $\gamma = 0.3$ and $0.4$. Hence, for the following examples, we will mainly employ $\gamma = 0.4$ to enhance both the numerical stability and the accuracy of the reconstruction. Moreover, to illustrate the feasibility of the proposed DSM with other choices of $\gamma$, we also employ $\gamma = 0.55$ in the second case of example 2 to demonstrate that our method performs stably for a wide range of $\gamma$  due to the choice of the probing function which serves as a low pass filter as we discussed in section \ref{sec_var_case2}.

Next, we would like to justify our preference of choosing $\alpha = n+1$ ($\alpha = 3$ in $\mathbb{R}^2$) fo reconstruction. We compute the first three images in the second row with the same measurement data under different choices of $\alpha = 3$, $4$, and $5$ with $\gamma = 0.4$ and $30\%$ additive Gaussian noise. Denoting $\Err^2_{\alpha = \lambda}$ as the discrete $L^2$-norm error of the reconstruction with $\alpha = \lambda$, then the corresponding reconstruction errors are given by
$$
\Err^2_{\alpha = 3} = 0.151 \,,\quad
\Err^2_{\alpha = 4} = 0.173 \,,\quad
\Err^2_{\alpha = 5} = 0.189 \,
$$
We observe that the reconstruction becomes less accurate as $\alpha$ becomes larger under high level Gaussian noise. The above observation echoes with the analysis in section \ref{sec_alternative_chara}. This suggests the choice of $\alpha = 3$ in most real applications, namely, $\alpha = n+1$ in $\mathbb{R}^n$ as justified in section \ref{sec_alternative_chara}.

\begin{figure}
     \begin{subfigure}[b]{0.245\textwidth}
                 \centering
                 \includegraphics[scale = 0.23]{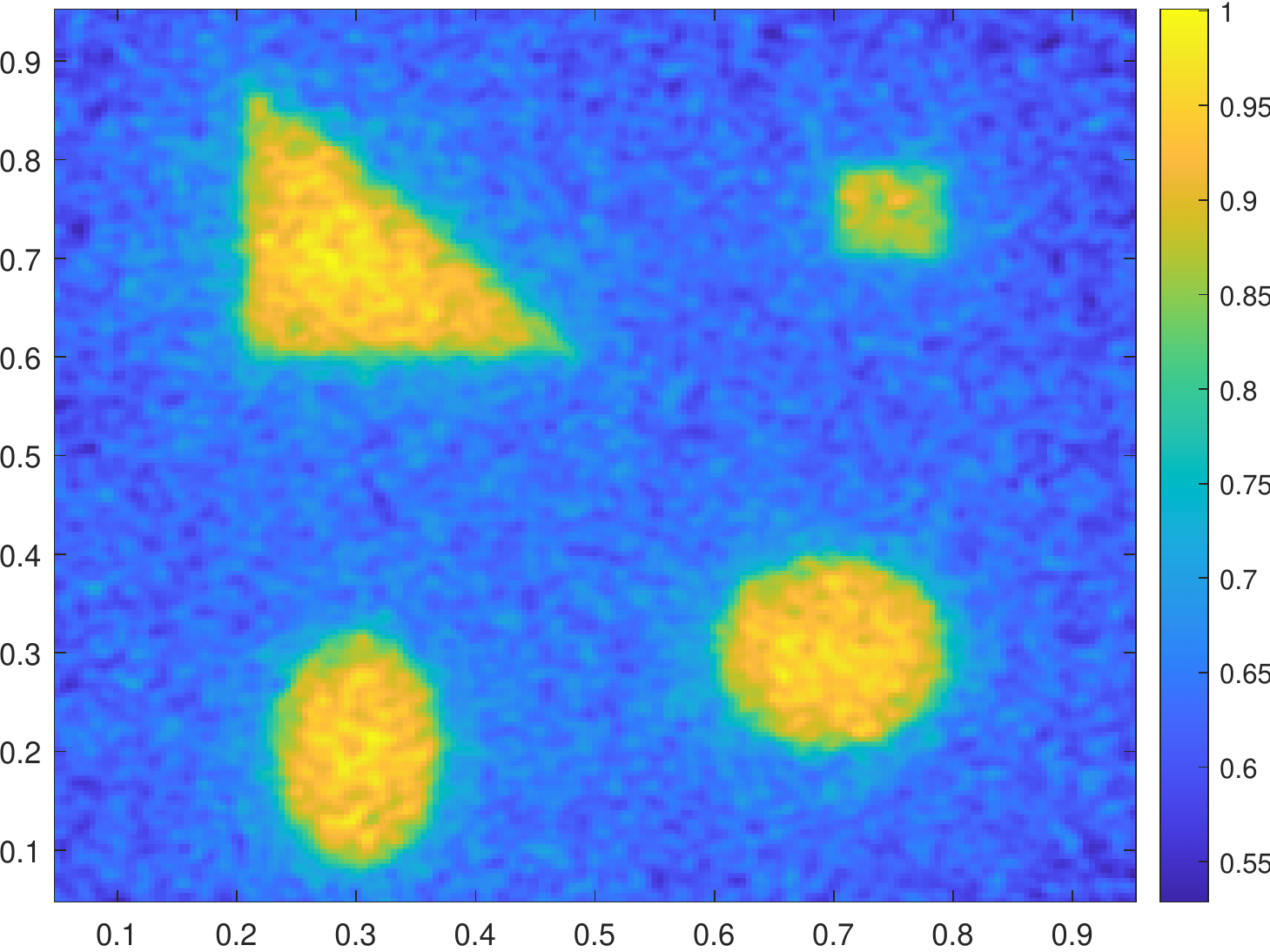}
                 \caption{DSM: $\gamma = 0.3$.}
         \end{subfigure}
     \begin{subfigure}[b]{0.245\textwidth}
                 \centering
                 \includegraphics[scale = 0.23]{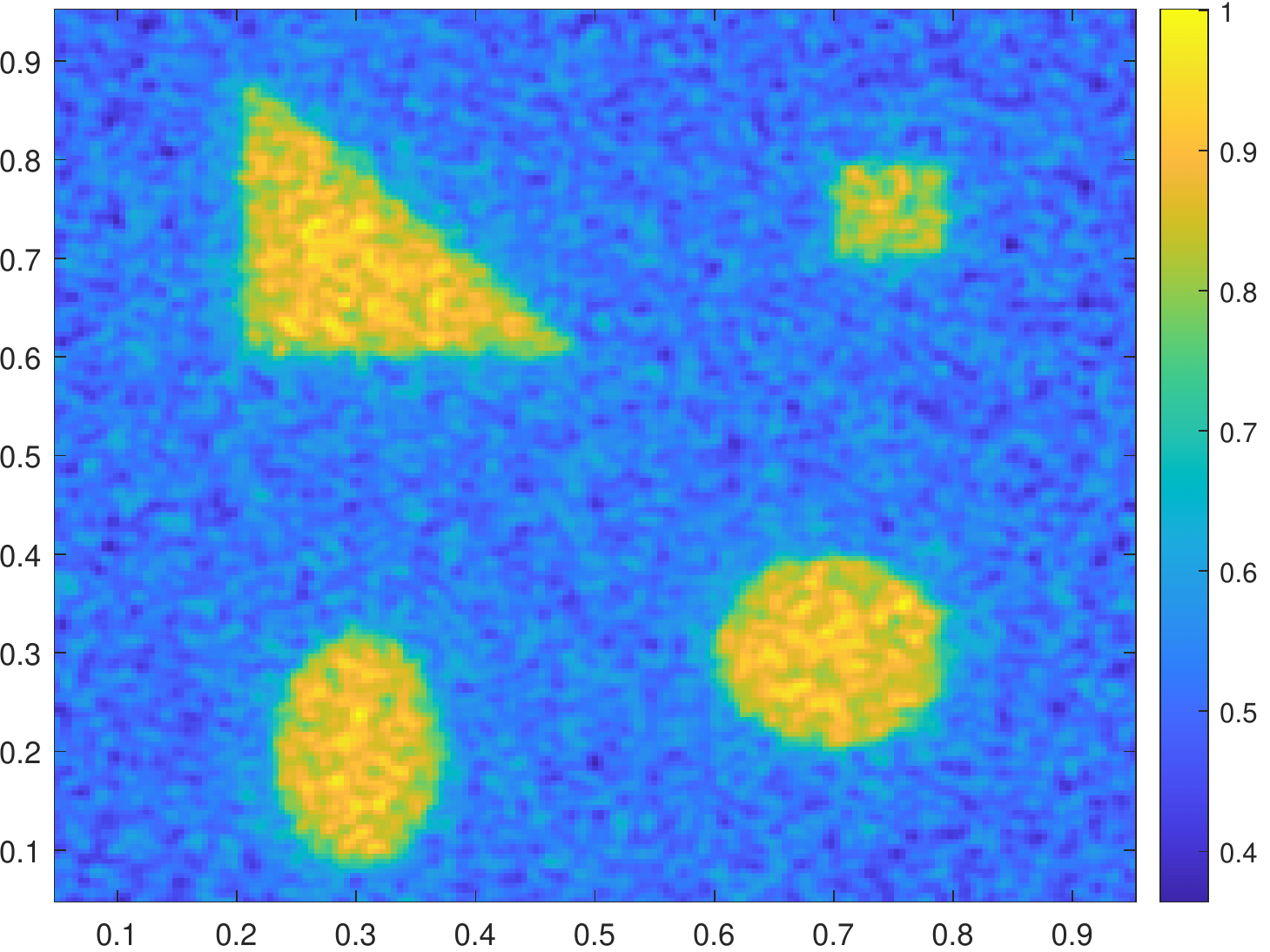}
                 \caption{DSM: $\gamma = 0.4$.}
         \end{subfigure}
   \begin{subfigure}[b]{0.245\textwidth}
                 \centering
                 \includegraphics[scale = 0.23]{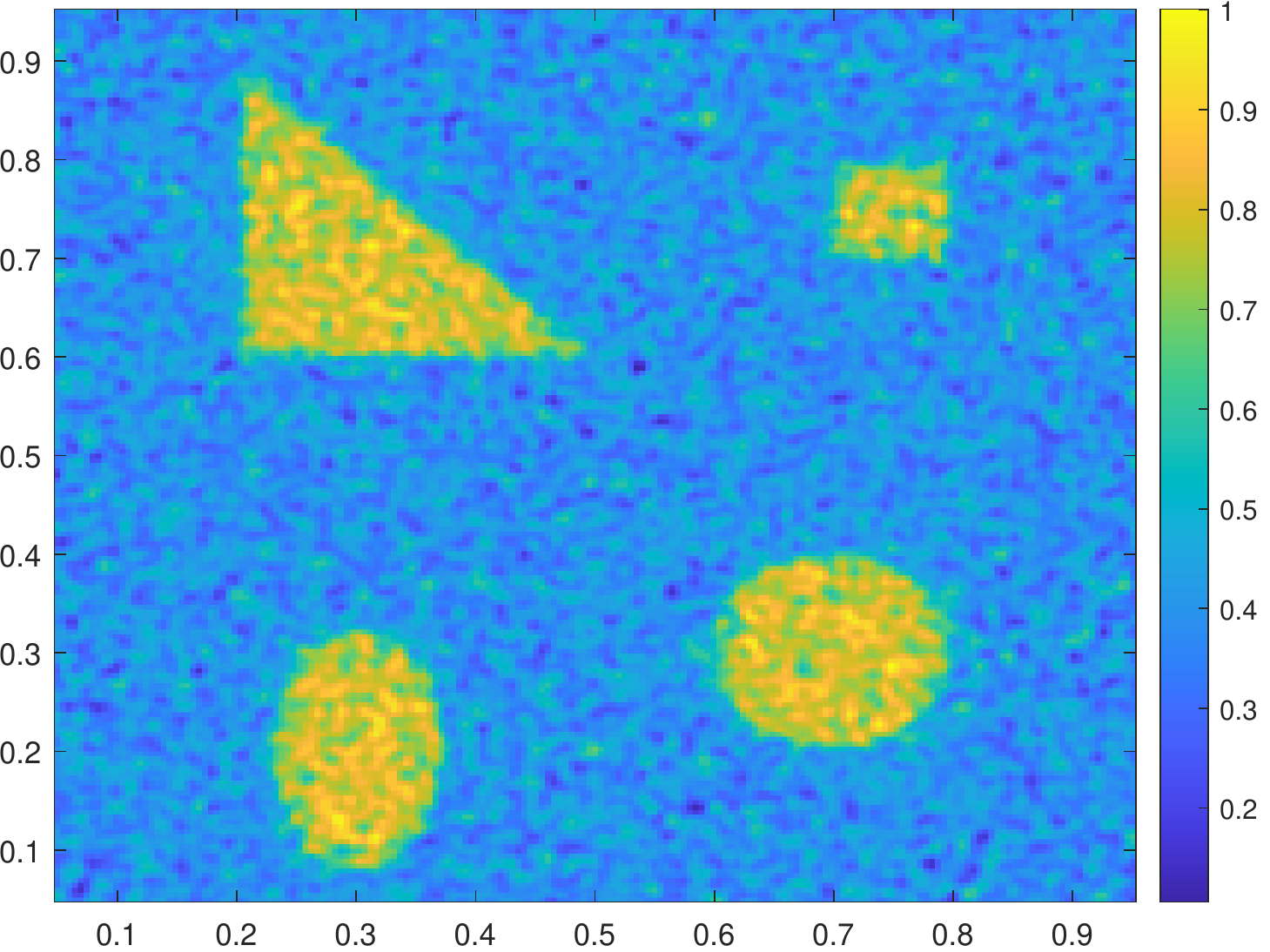}
                 \caption{DSM: $\gamma = 0.5$.}
         \end{subfigure}
   \begin{subfigure}[b]{0.245\textwidth}
                 \centering
                 \includegraphics[scale = 0.23]{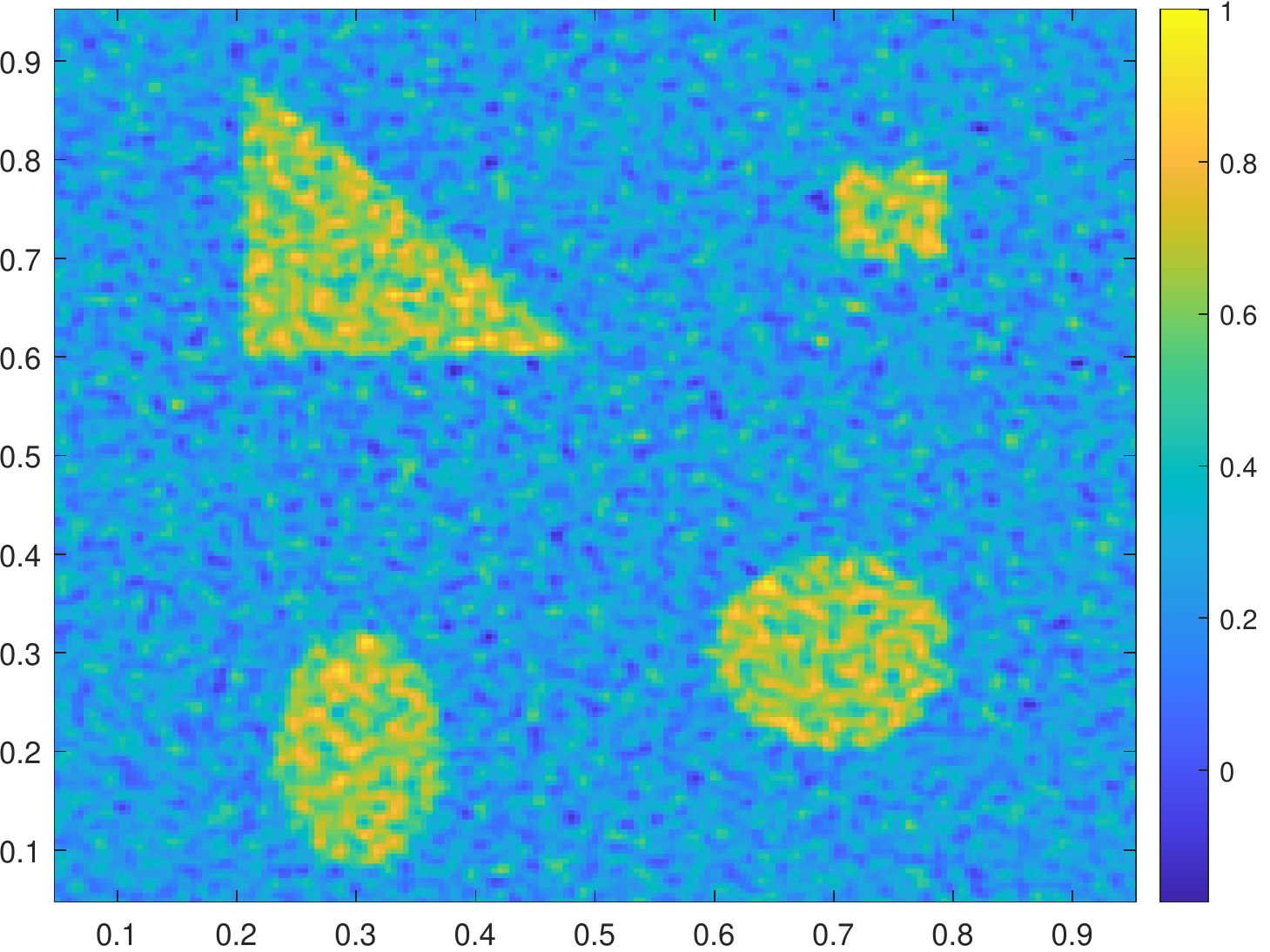}
                 \caption{DSM: $\gamma = 0.6$.}
   \end{subfigure}\newline
   \begin{subfigure}[b]{0.245\textwidth}
                 \centering
                 \includegraphics[scale = 0.23]{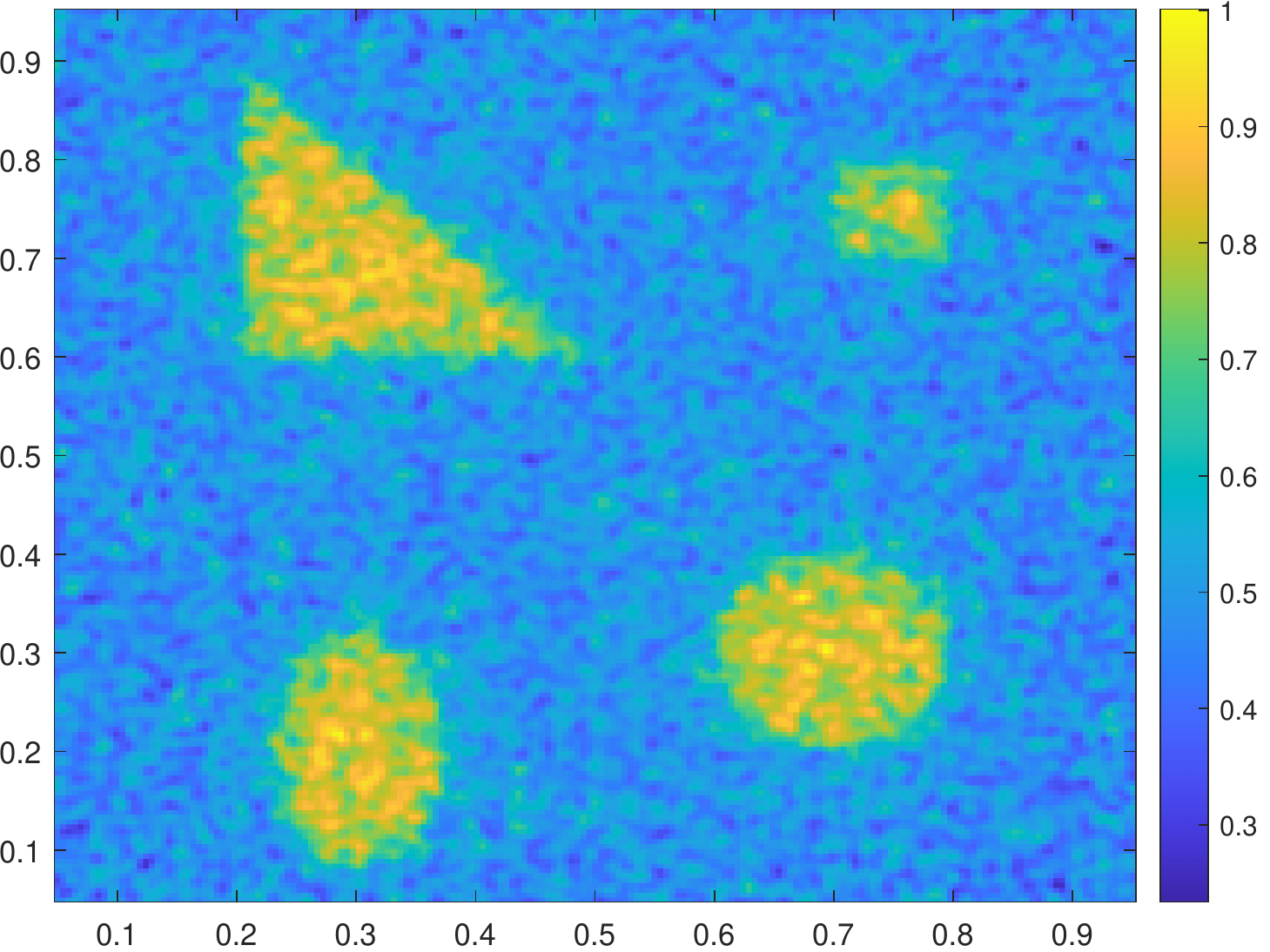}
                 \caption{DSM: $\alpha = 3$.}
         \end{subfigure}
   \begin{subfigure}[b]{0.245\textwidth}
                 \centering
                 \includegraphics[scale = 0.23]{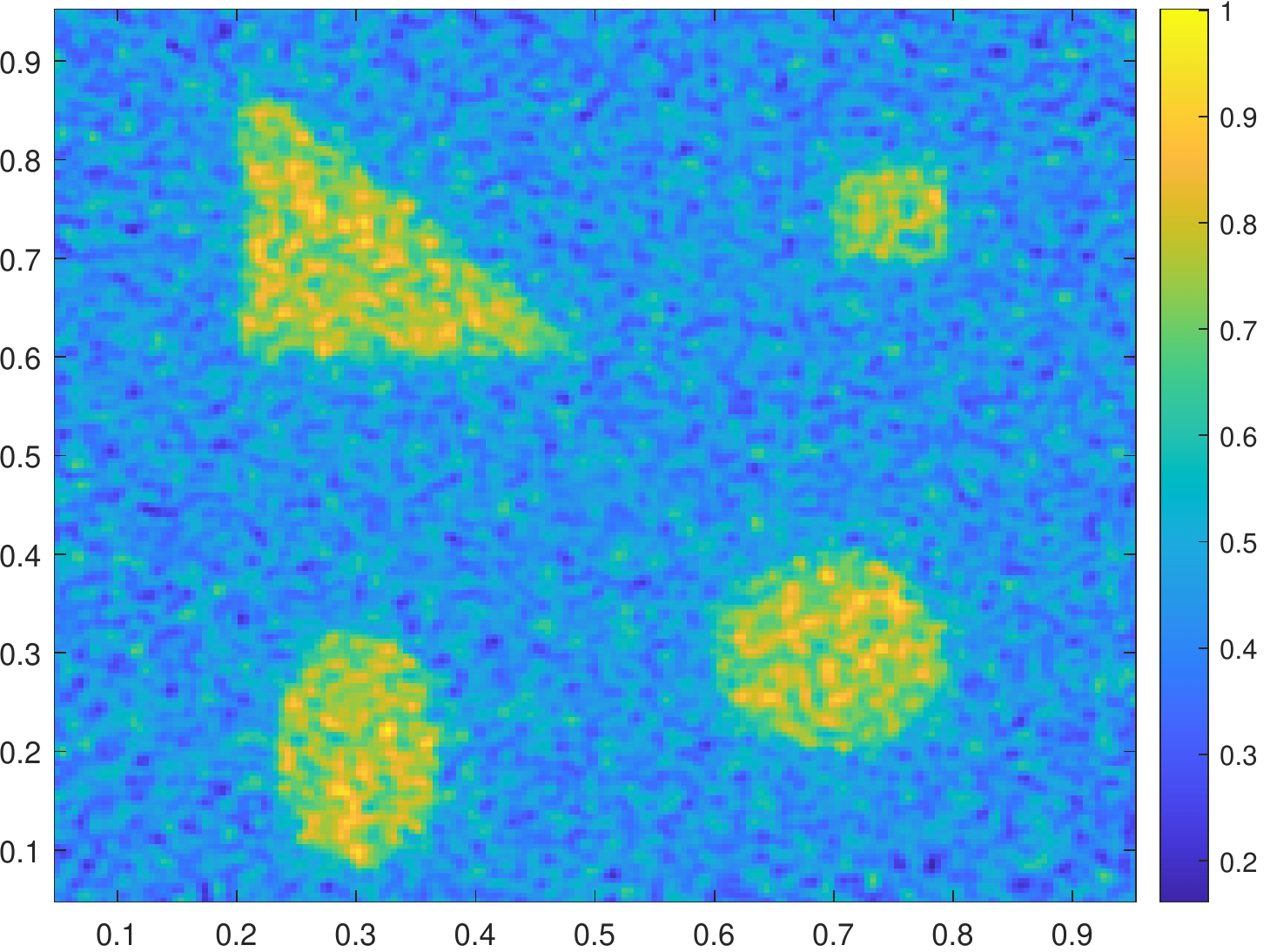}
                 \caption{DSM: $\alpha = 4$.}
         \end{subfigure}
   \begin{subfigure}[b]{0.245\textwidth}
                 \centering
                 \includegraphics[scale = 0.23]{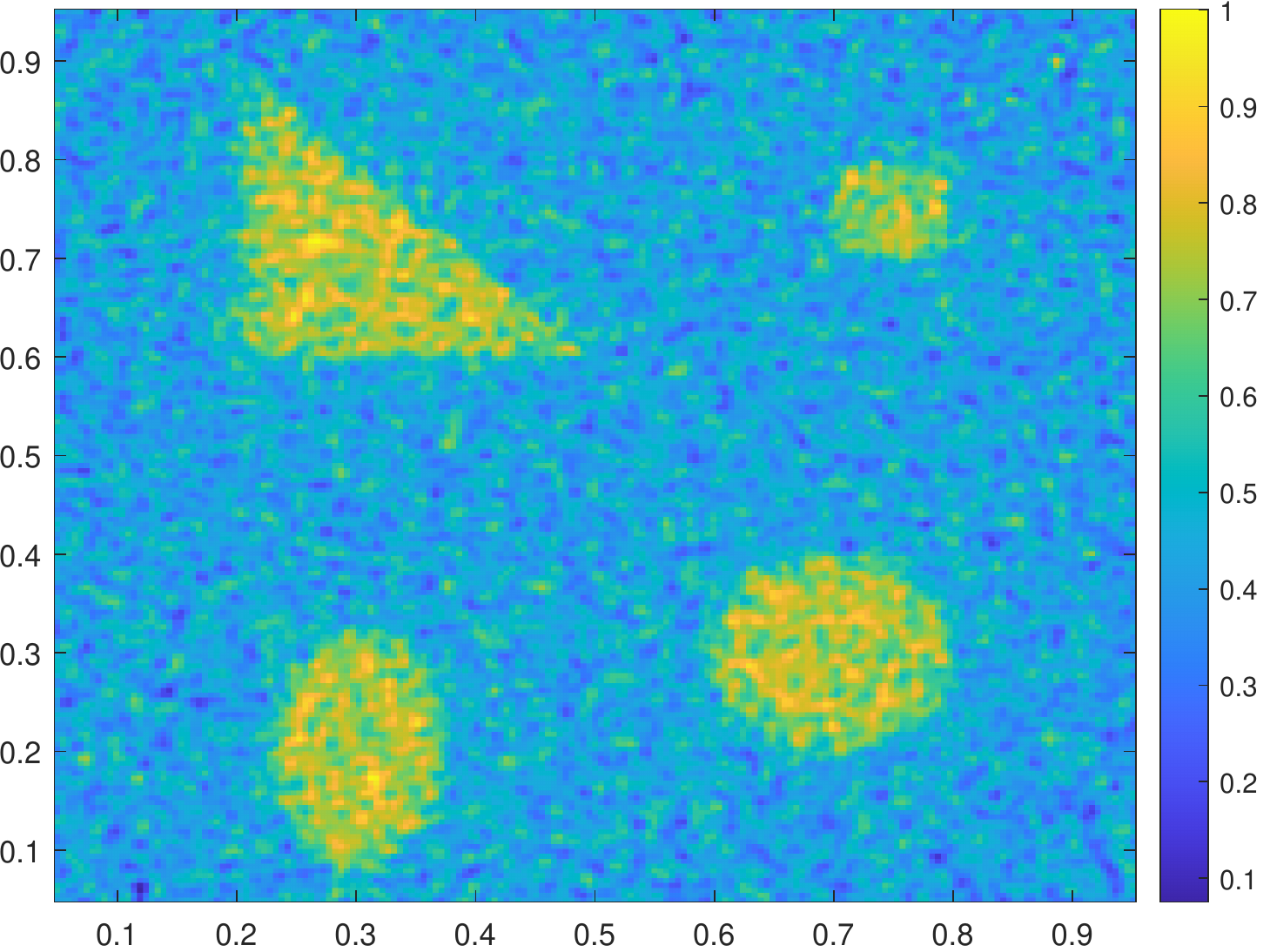}
                 \caption{DSM: $\alpha = 5$.}
       \end{subfigure}
    \begin{subfigure}[b]{0.245\textwidth}
                 \centering
                 \includegraphics[scale = 0.23]{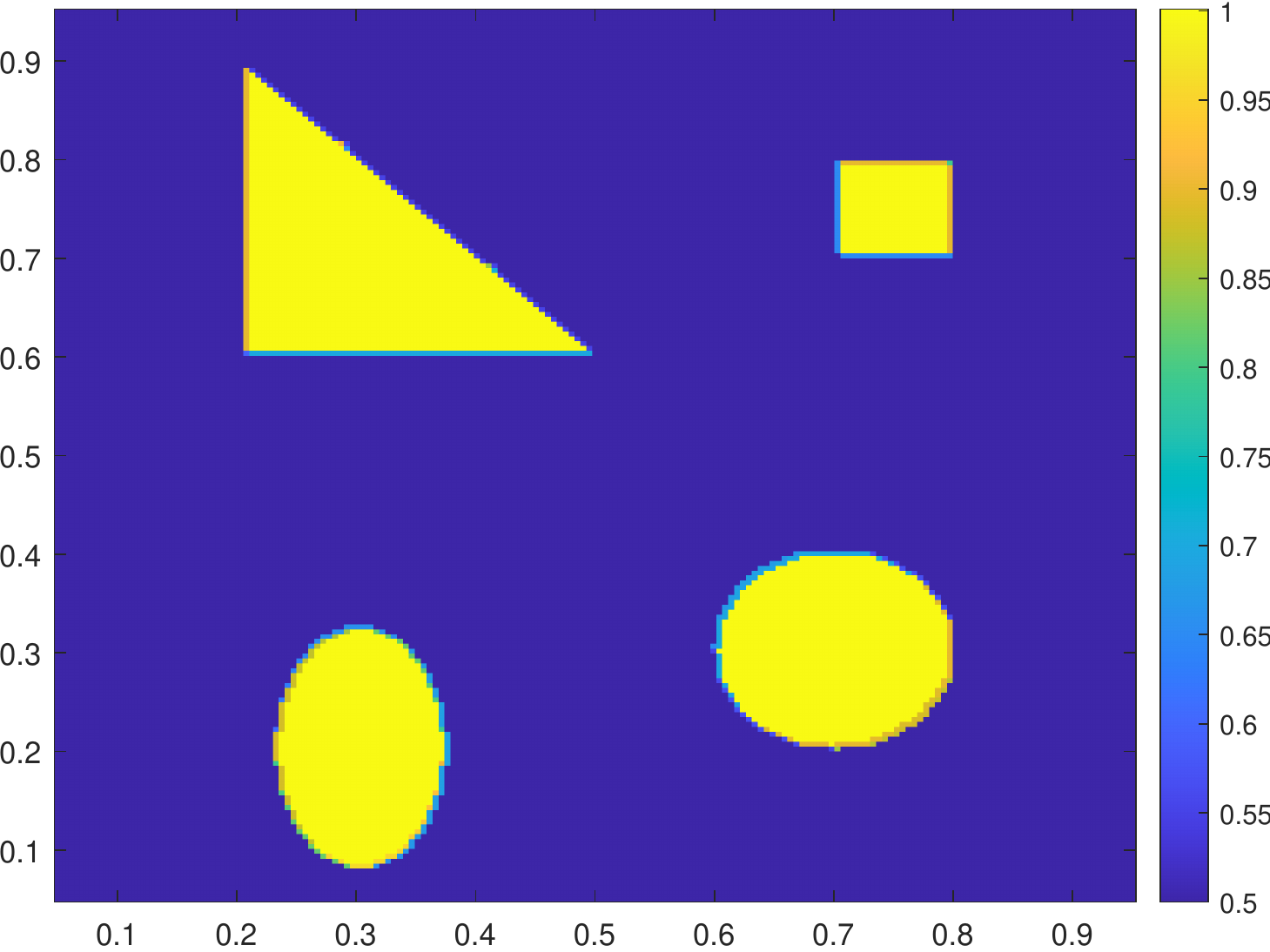}
                 \caption{$f(x)$.}
       \end{subfigure}
    \caption{Example 1. Influence of choices of $\gamma$ and $\alpha$: reconstruction by DSM under $20\%$ additive Gaussian noise and $\alpha = 3$ with $\gamma=0.3$, $0.4$, $0.5$, and $0.6$ (first row); reconstruction by DSM under $30\%$ additive Gaussian noise and $\gamma = 0.4$ with $\alpha=3$, $4$, $5$ and the original image (second row).}
\label{example1}
\end{figure}

\textbf{Example 2.} 
This example involves additive Gaussian noise in the data. The reconstructions by DSM (with $\gamma = 0.4$ in the first row and with $\gamma = 0.55$ in the second row) and FBP are shown in Fig.\,\ref{example2}. The corresponding reconstruction errors are given respectively by
\begin{equation*}
\Err^2_{DSM} =        0.135 \,,\quad 
\Err^2_{FBP}=        0.293 \,,\quad 
\Err^{\infty}_{DSM} = 0.143   \,,\quad 
\Err^{\infty}_{FBP}=   0.245   \,,
 \end{equation*}
for the reconstructions in the first row with the noise level being $20\%$, and by 
\begin{equation*}
\Err^2_{DSM}=     0.237 \,, \quad 
\Err^2_{FBP}=    0.279 \,,\quad 
\Err^{\infty}_{DSM} = 0.202   \,,\quad 
\Err^{\infty}_{FBP}=   0.223   \,,
\end{equation*}
for the reconstructions in the second row with the noise level being $20\%$.

From the numerical reconstructions, we can observe that the DSM is very robust against strong Gaussian noise in the measurement data. And based on the $L^2$-norm error and the $L^{\infty}$-norm error of the reconstruction, we can see that the DSM performs obviously better than FBP.
\begin{figure}
    \begin{subfigure}[b]{0.33\textwidth}
                 \centering
                 \includegraphics[scale = 0.301]{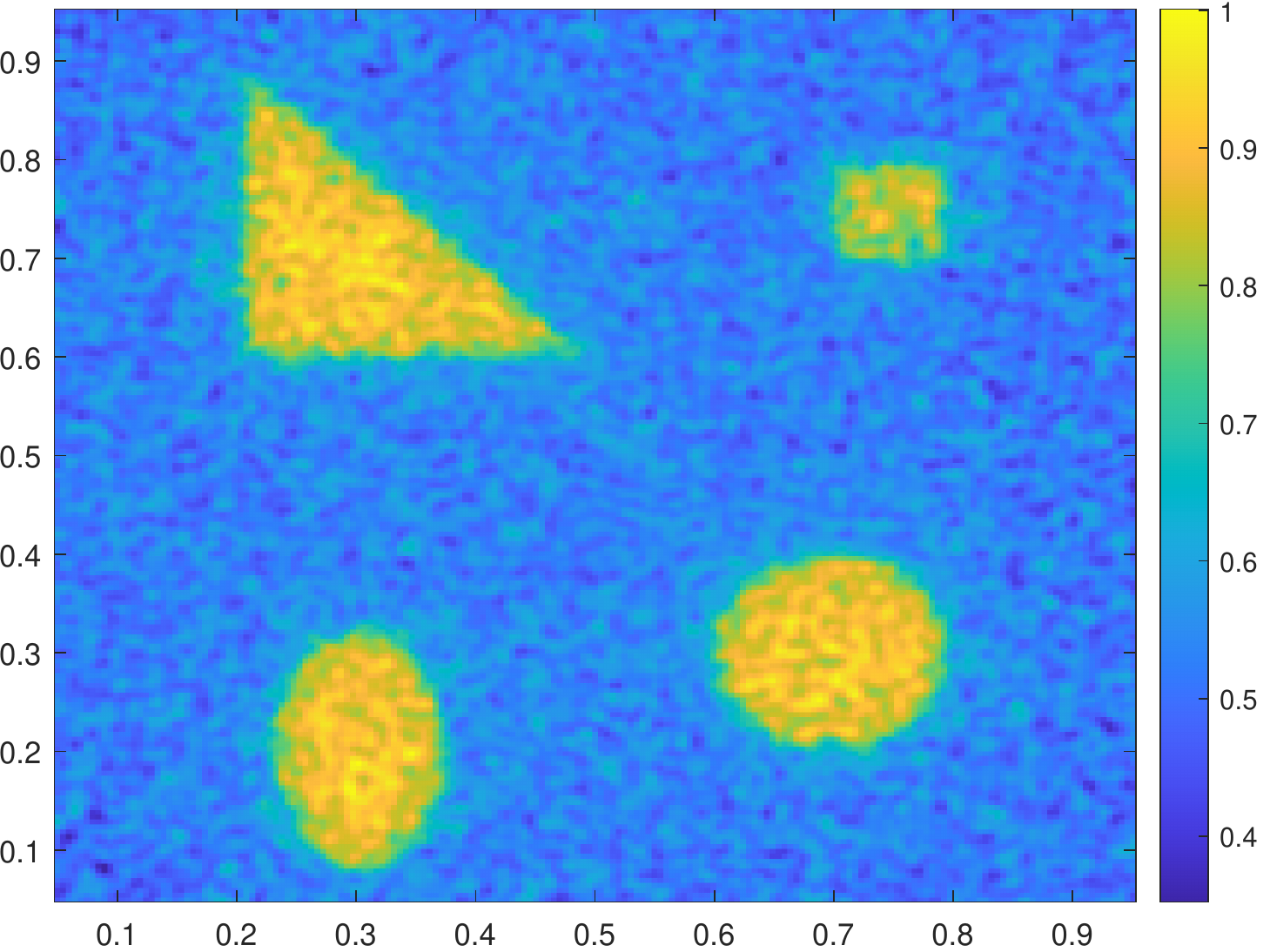}
                 \caption{DSM: $\gamma = 0.4$.}
      \end{subfigure}
   \begin{subfigure}[b]{0.33\textwidth}
      \centering
                 \includegraphics[scale = 0.301]{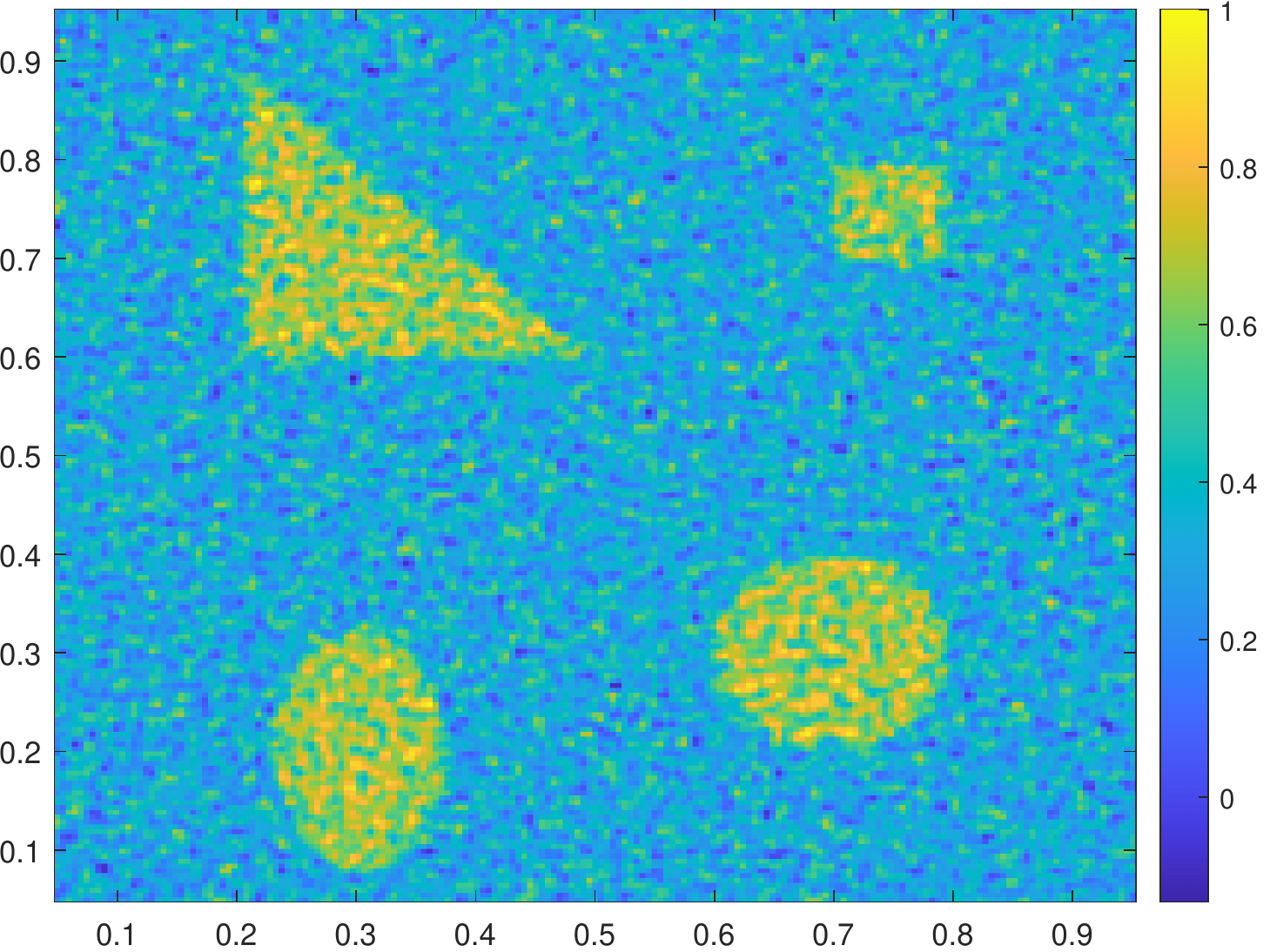}
                 \caption{FBP.}
         \end{subfigure}
   \begin{subfigure}[b]{0.33\textwidth}
                 \centering
                 \includegraphics[scale = 0.301]{example2afx.pdf}
                 \caption{$f(x)$.}
         \end{subfigure}
    \newline     \begin{subfigure}[b]{0.33\textwidth}
                 \centering
                 \includegraphics[scale = 0.301]{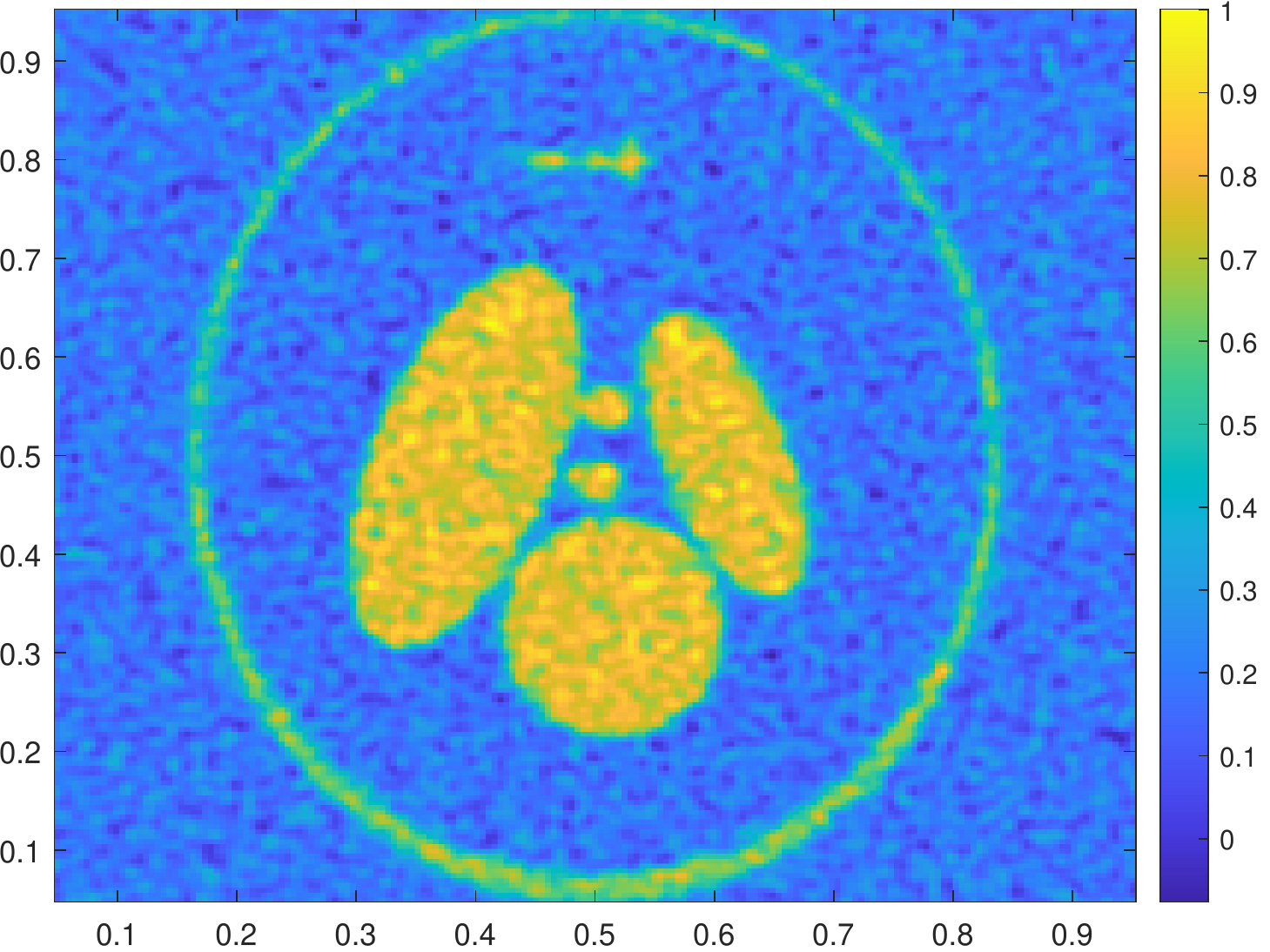}
                 \caption{DSM: $\gamma = 0.55$.}
         \end{subfigure}
   \begin{subfigure}[b]{0.33\textwidth}
                 \centering
                 \includegraphics[scale = 0.301]{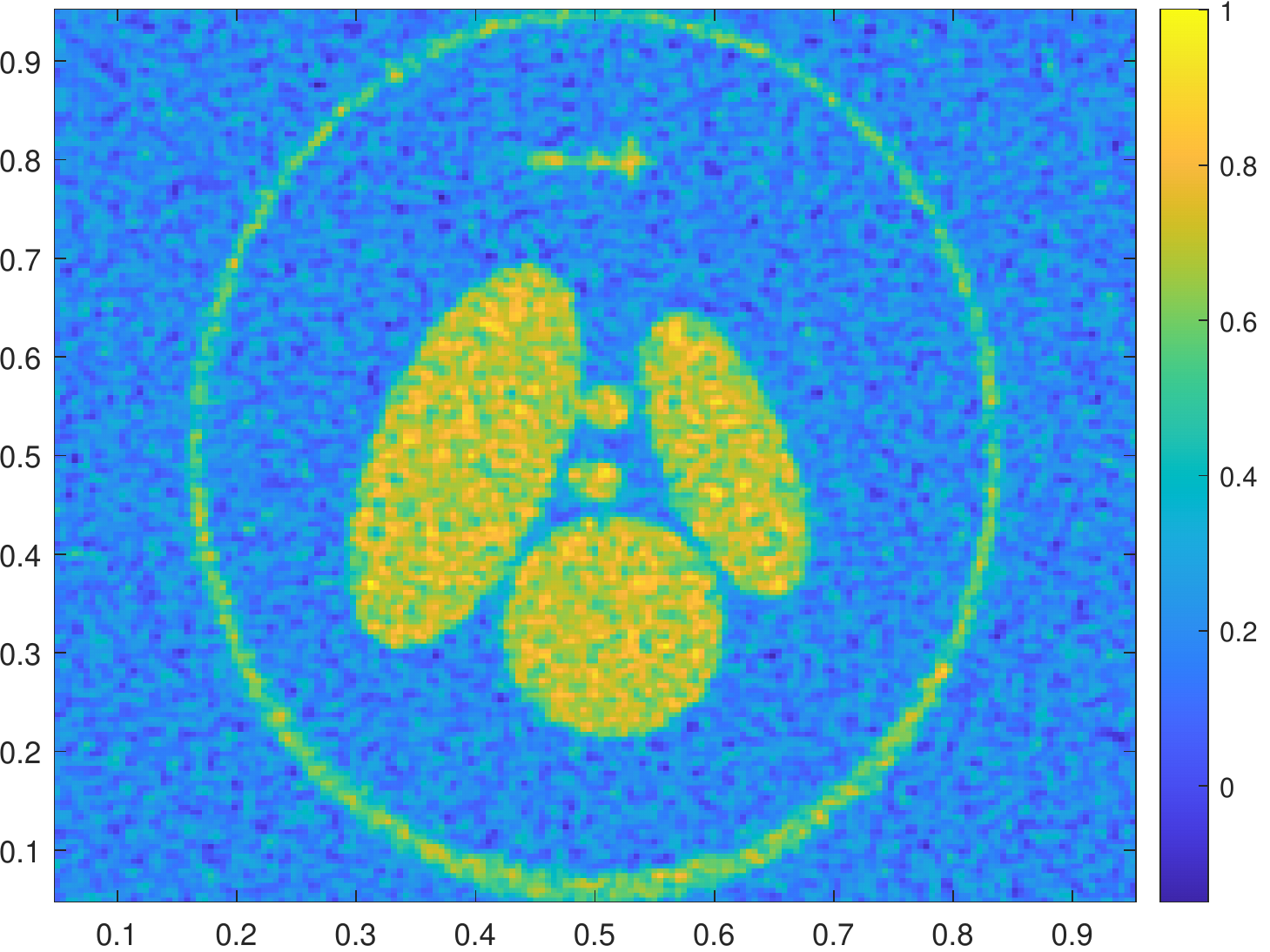}
                 \caption{FBP.}
         \end{subfigure}
   \begin{subfigure}[b]{0.33\textwidth}
                 \centering
                 \includegraphics[scale = 0.301]{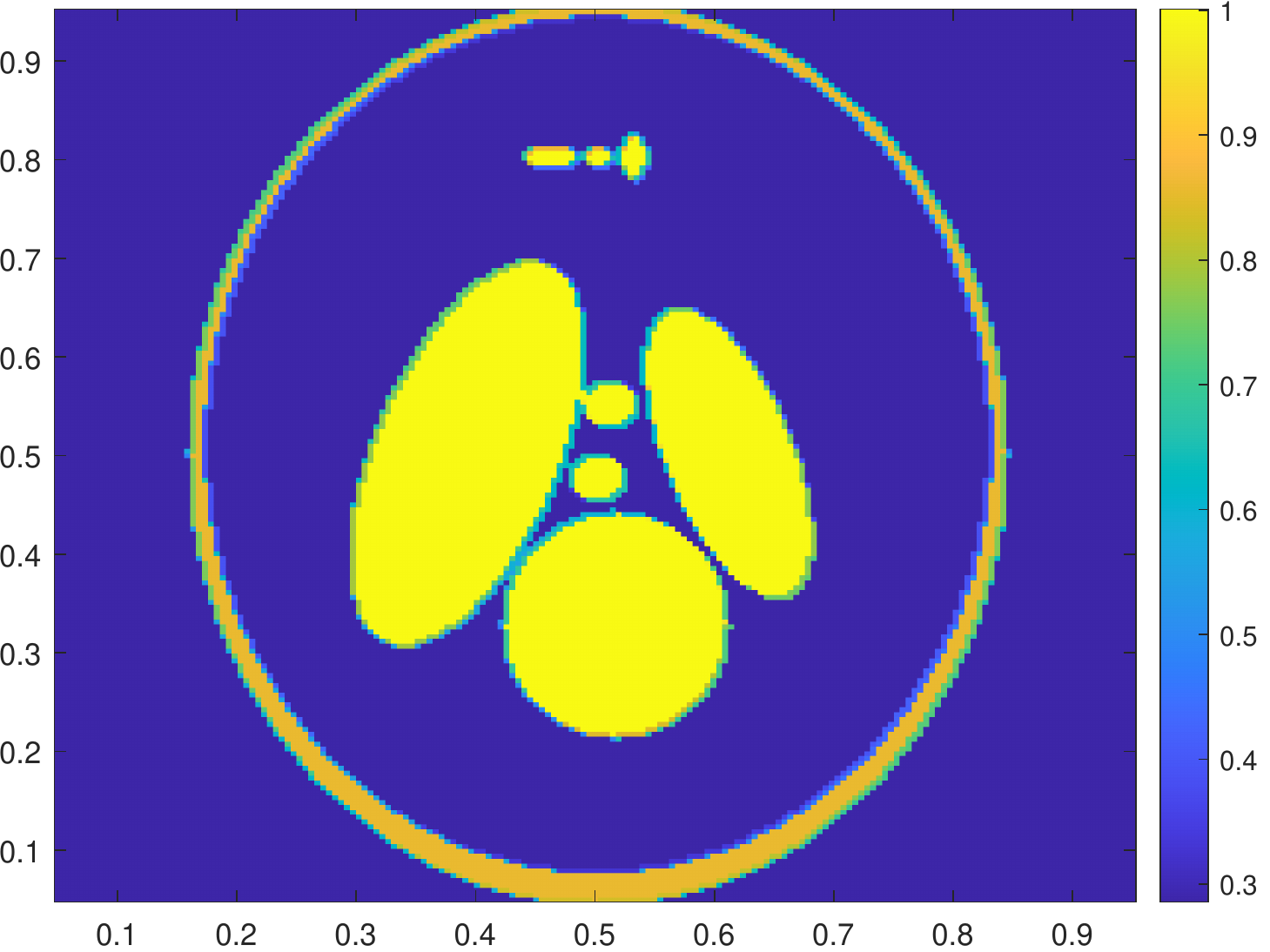}
                 \caption{$f(x)$.}
         \end{subfigure}
    \caption{Example 2. Under additive Gaussian noise: $20\%$.}
\label{example2}
\end{figure}

\textbf{Example 3.}  
In this example, we consider the 'salt and pepper' type noise in the measurement data. The reconstructions by DSM (with $\gamma = 0.4$) and FBP are shown in Fig.\,\ref{example3}. The corresponding reconstruction errors are given respectively by
\begin{equation*}
\Err_{DSM}^2=   0.180 \,,\quad 
\Err^2_{FBP}=   0.530 \,,\quad 
\Err^{\infty}_{DSM} = 0.173   \,,\quad 
\Err^{\infty}_{FBP}=   0.454   \,,
\end{equation*}
for the reconstruction in the first row with the noise level being $8\%$, and by 
    \begin{equation*}
\Err^2_{DSM}=    0.269 \,,\quad 
\Err^2_{FBP}=    0.369 \,,\quad 
\Err^{\infty}_{DSM} = 0.232   \,,\quad 
\Err^{\infty}_{FBP}=   0.300   \,,
\end{equation*}
for the reconstruction in the second row with the noise level being $8\%$.

From the numerical reconstructions, we notice that the DSM is quite stable and accurate when the measurement data is severely polluted by the 'salt and pepper' type noise. And based on the $L^2$-norm error and the $L^{\infty}$-norm error of the reconstruction, we can see that the DSM performs obviously better than FBP. Moreover, comparing reconstruction results of DSM and FBP in Fig.\,\ref{example3}(a) and (b), DSM can recover the shape of the triangle and the ellipse much more accurately. 

\begin{figure}
    \begin{subfigure}[b]{0.33\textwidth}
                 \centering
                 \includegraphics[scale = 0.301]{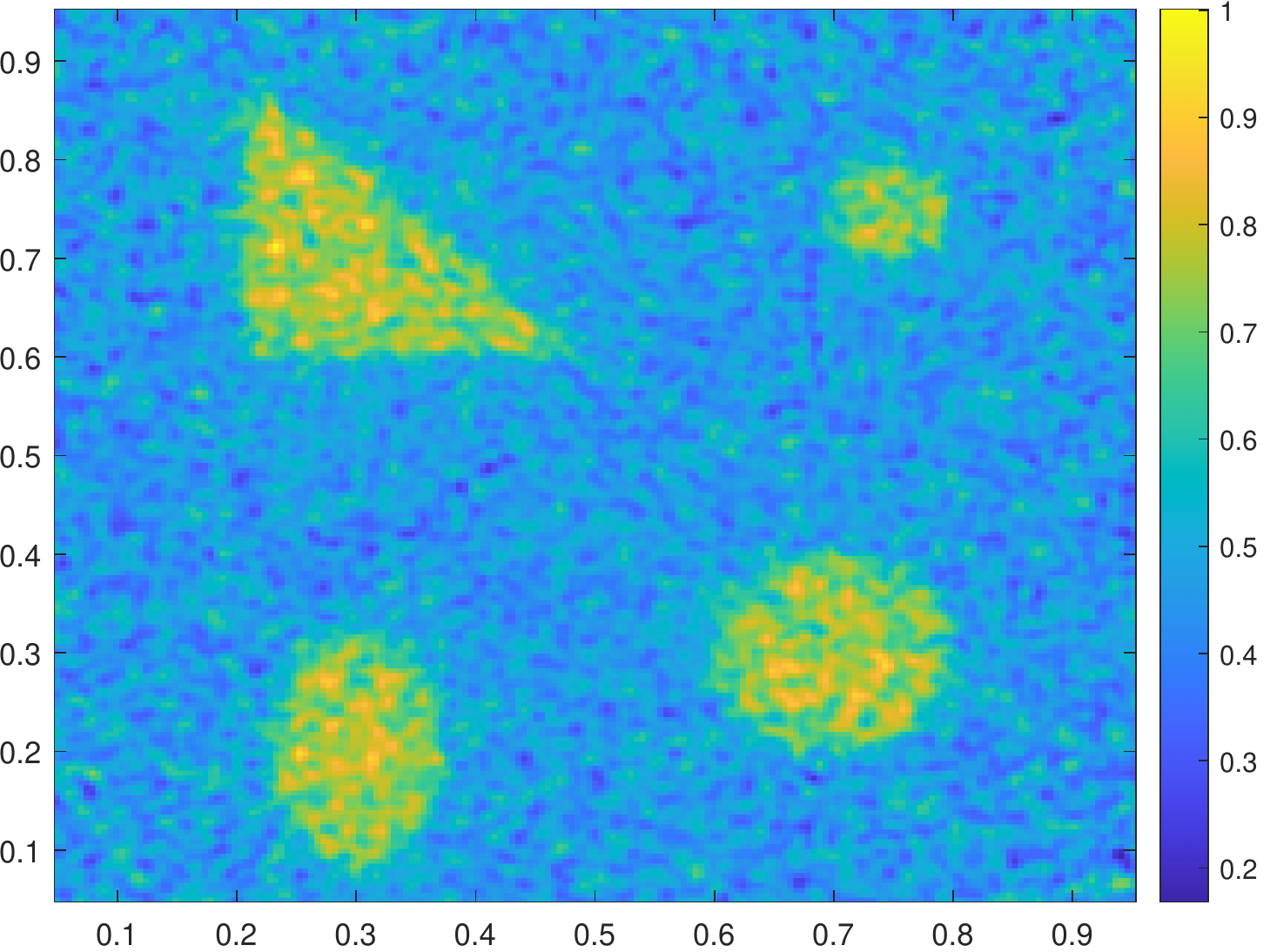}
                 \caption{DSM: $\gamma = 0.4$.}
         \end{subfigure}
   \begin{subfigure}[b]{0.33\textwidth}
                 \centering
                 \includegraphics[scale = 0.301]{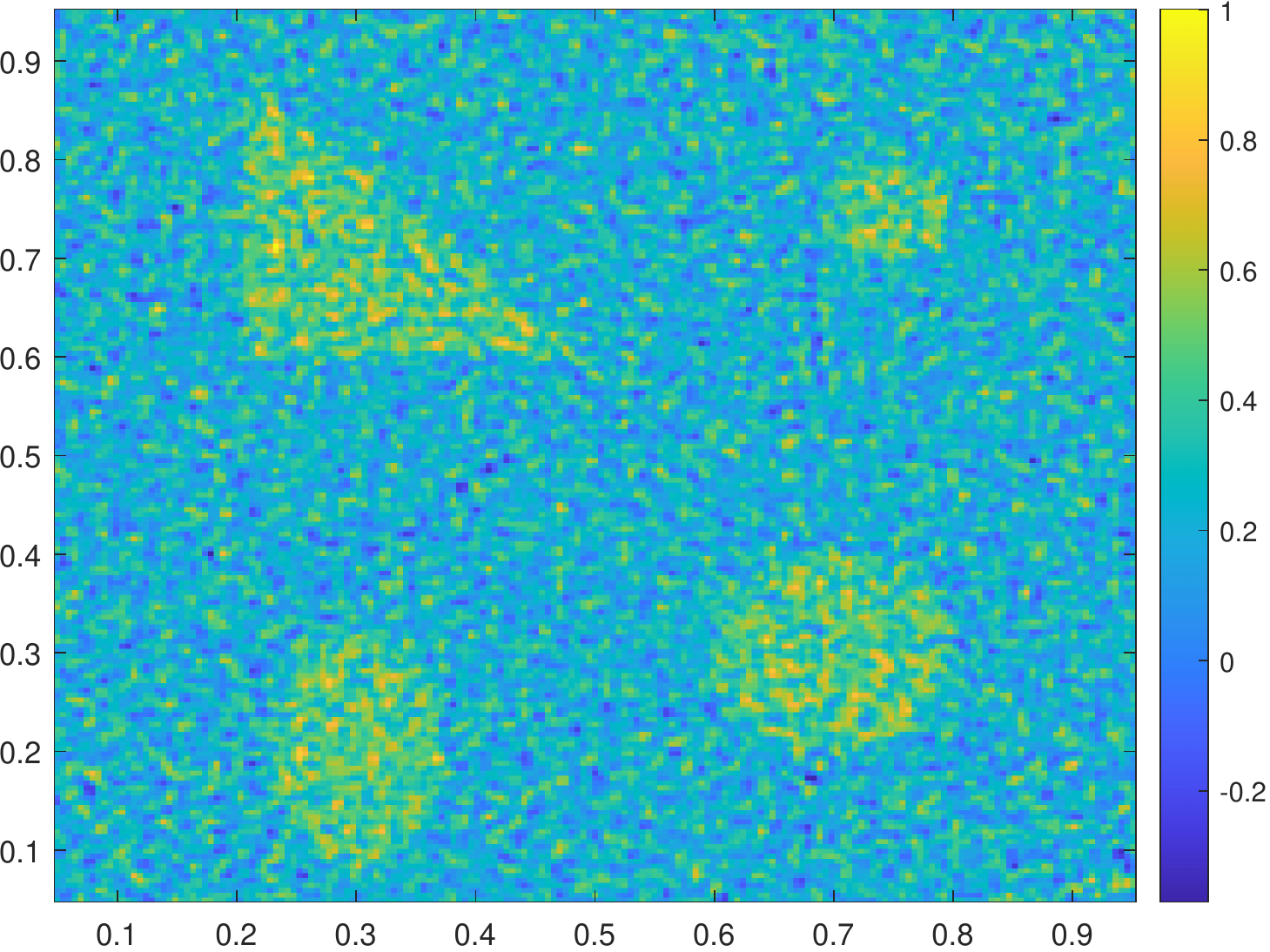}
                 \caption{FBP.}
         \end{subfigure}
   \begin{subfigure}[b]{0.33\textwidth}
                 \centering
                 \includegraphics[scale = 0.301]{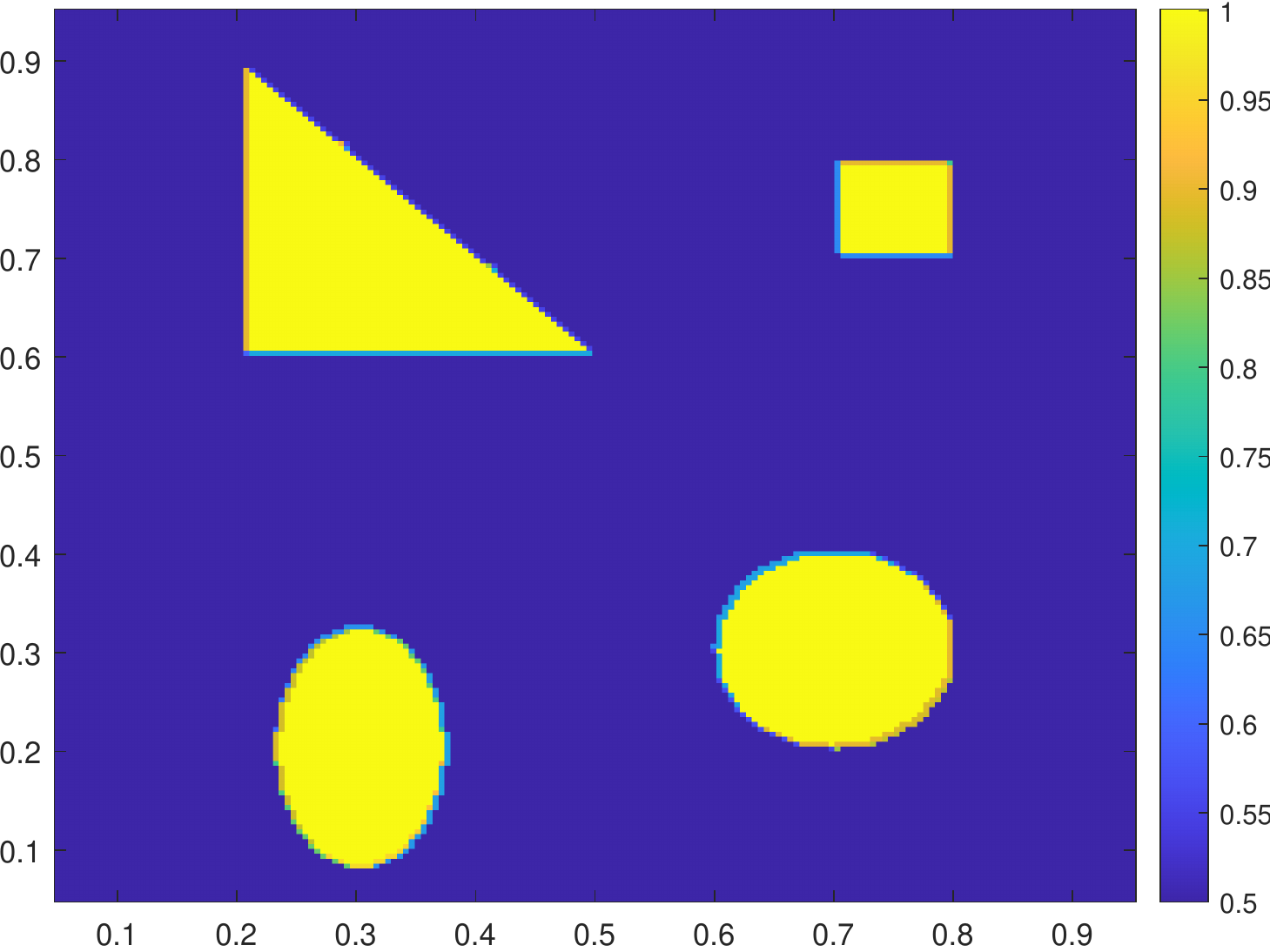}
                 \caption{$f(x)$.}
         \end{subfigure}
      \newline
     \begin{subfigure}[b]{0.33\textwidth}
                 \centering
                 \includegraphics[scale = 0.301]{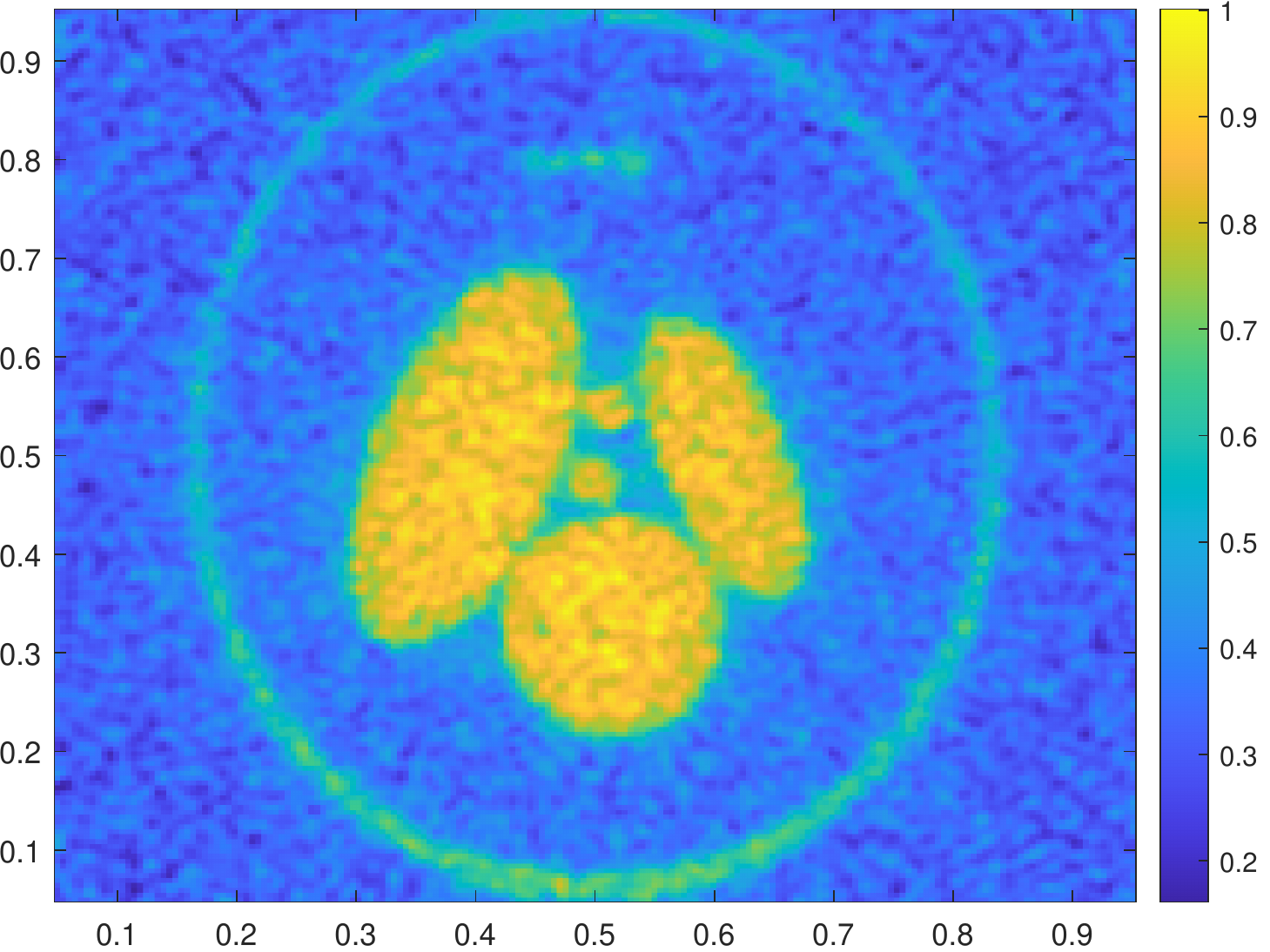}
                 \caption{DSM: $\gamma = 0.4$.}
         \end{subfigure}
   \begin{subfigure}[b]{0.33\textwidth}
                 \centering
                 \includegraphics[scale = 0.301]{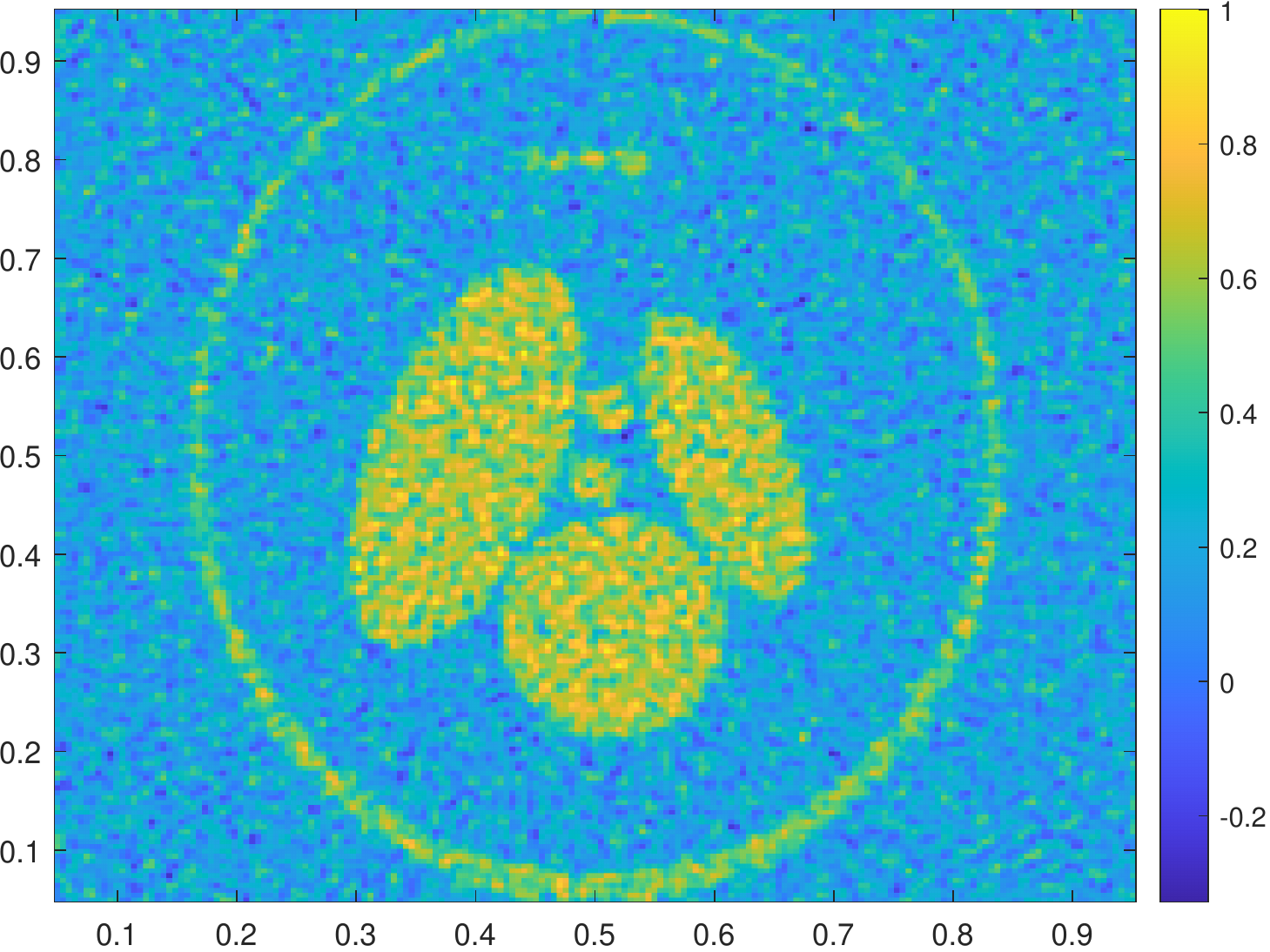}
                 \caption{FBP.}
         \end{subfigure}
   \begin{subfigure}[b]{0.33\textwidth}
                 \centering
                 \includegraphics[scale = 0.301]{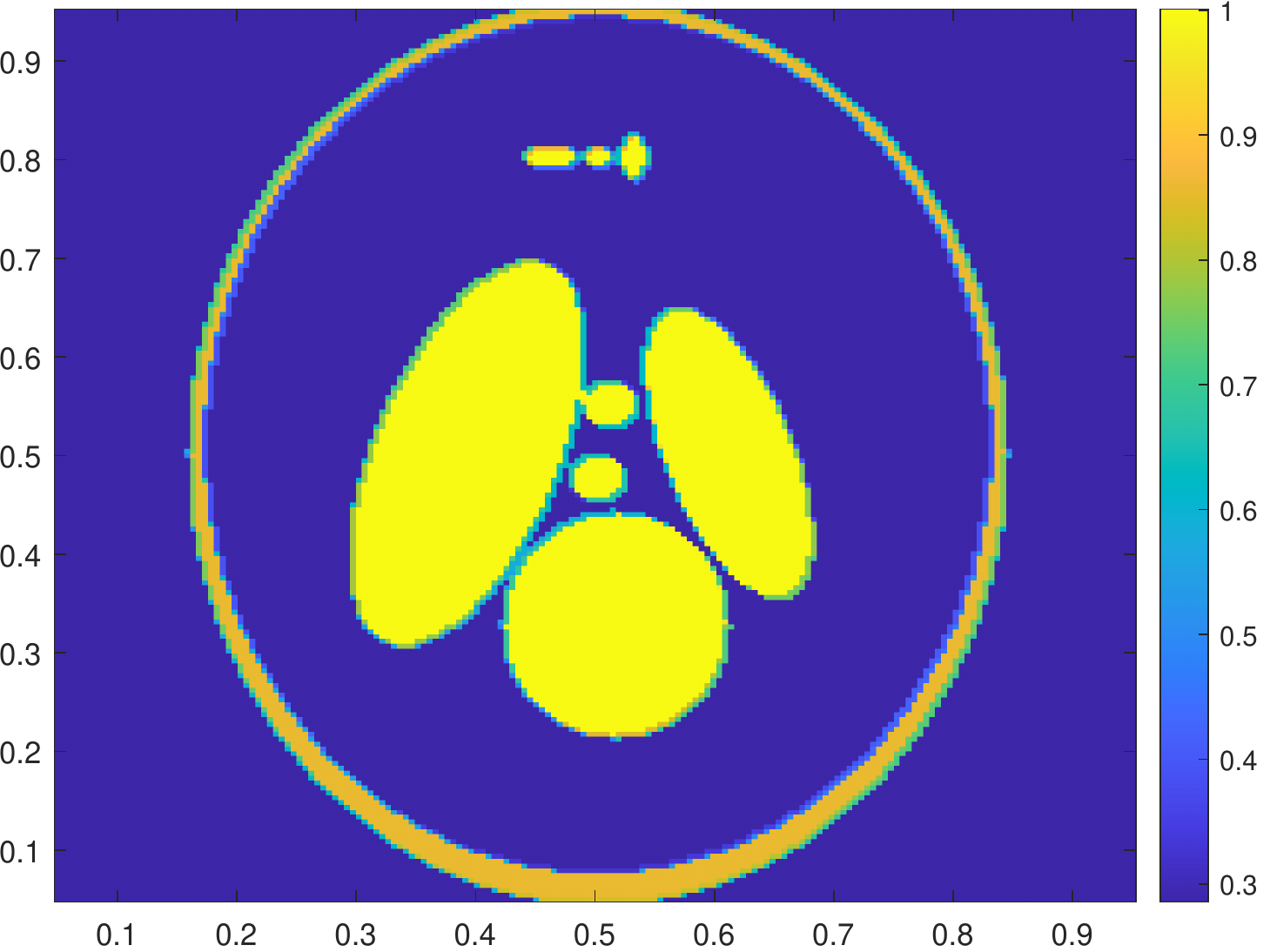}
                 \caption{$f(x)$.}
         \end{subfigure}
    \caption{Example 3. Under 'salt and pepper' noise: $8\%$.}
\label{example3}
\end{figure}

\textbf{Example 4.} 
This example studies a relatively challenging case with a limited number of projection angles in $\Gamma_\theta$, sparsely distributed over $[-\pi/2,\pi/2)$. The reconstructions by DSM (with $\gamma = 0.4$) and FBP are shown in Fig.\,\ref{example4}. The corresponding reconstruction errors are given respectively by
\begin{equation*}
\Err^2_{DSM} = 0.165 \,,\quad
\Err^2_{FBP} = 0.463 \,,\quad
\Err^{\infty}_{DSM} = 0.203   \,,\quad 
\Err^{\infty}_{FBP}=   0.478   \,,
\end{equation*}
for the reconstruction in the first row with projections from $18$ angles, and by  
\begin{equation*}
\Err^2_{DSM} = 0.214 \,,\quad
\Err^2_{FBP} = 0.650 \,,\quad
\Err^{\infty}_{DSM} = 0.266   \,,\quad 
\Err^{\infty}_{FBP}=   1.064   \,,
\end{equation*}
for the reconstruction in the second row with projections from $10$ angles.

As we may see from the reconstructions, the DSM demonstrates its strong robustness in this highly ill-posed scenario especially with respect to the $L^{\infty}$-norm error of the reconstruction. Moreover, for reconstructions in the second row with projections only from 10 directions, DSM still allows us to identify the shape and the location of objects in a reasonable manner while it is difficult to obtain useful information from the reconstruction by the FBP method. This shows a great potential of the DSM in real applications when projection angles are very sparsely distributed.

\begin{figure}
    \begin{subfigure}[b]{0.33\textwidth}
                 \centering
                 \includegraphics[scale = 0.301]{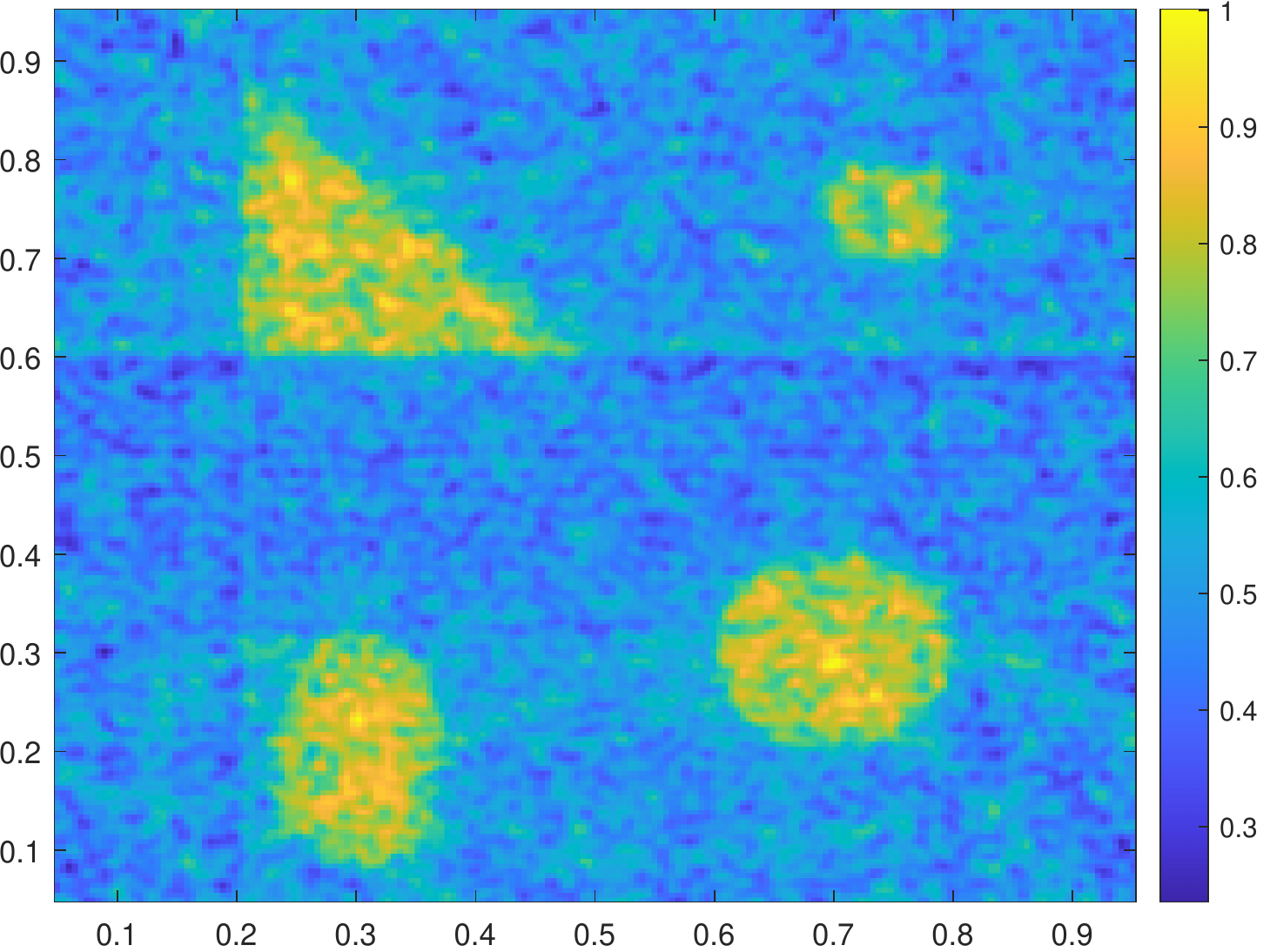}
                 \caption{DSM: $\gamma = 0.4$.}
         \end{subfigure}
   \begin{subfigure}[b]{0.33\textwidth}
                 \centering
                 \includegraphics[scale = 0.301]{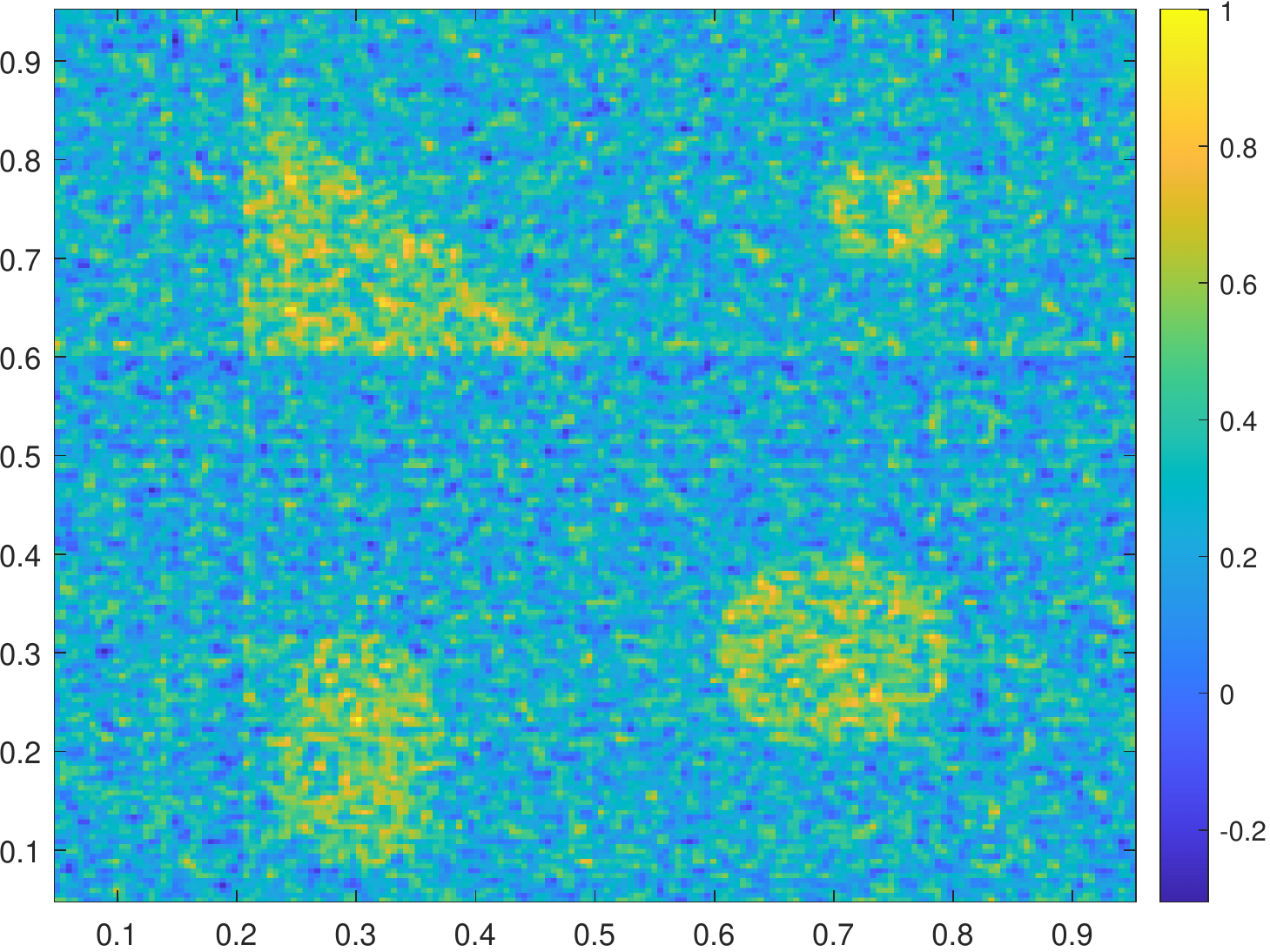}
                 \caption{FBP.}
         \end{subfigure}
   \begin{subfigure}[b]{0.33\textwidth}
                 \centering
                 \includegraphics[scale = 0.301]{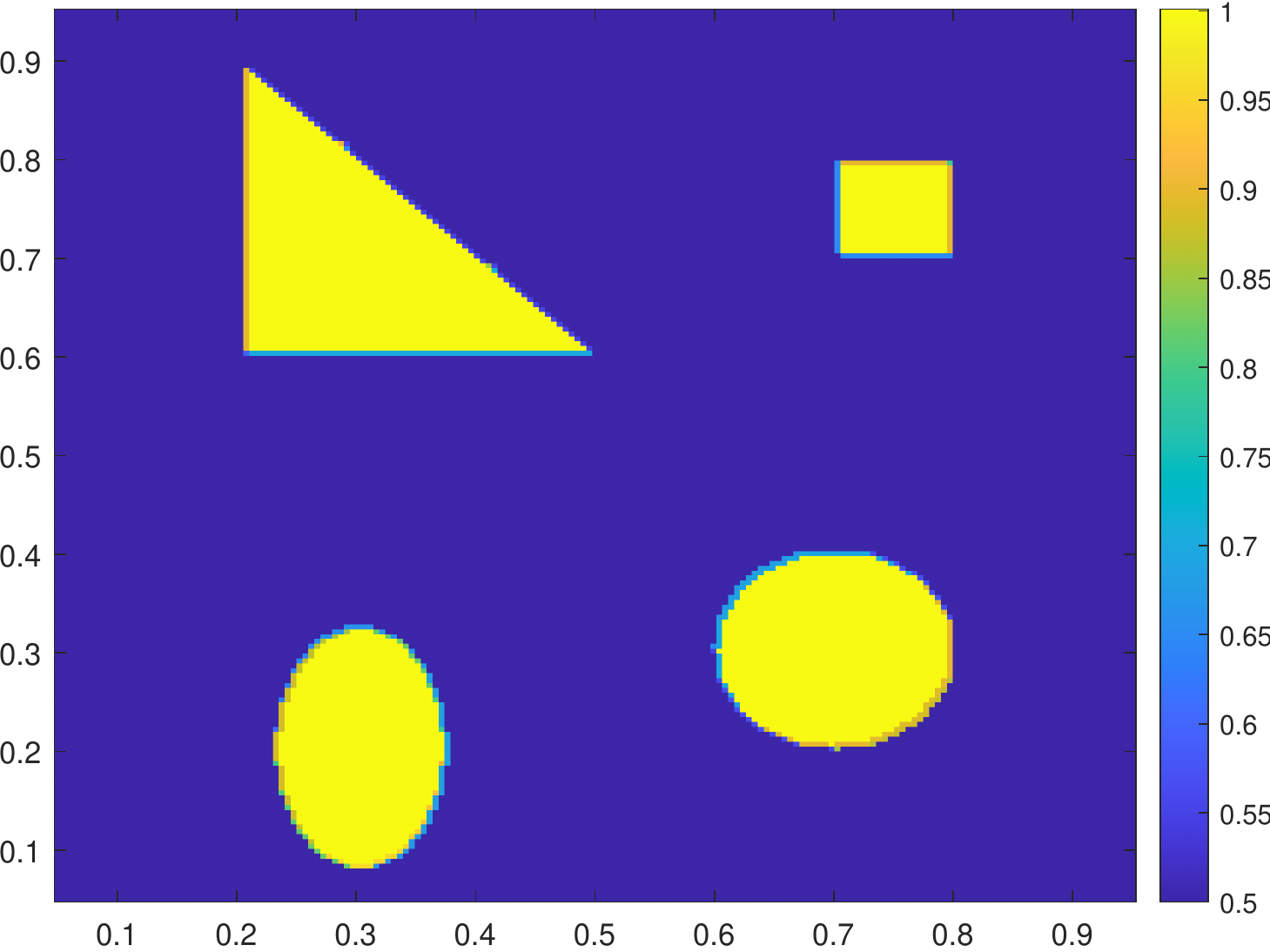}
                 \caption{$f(x)$.}
         \end{subfigure}
      \newline
     \begin{subfigure}[b]{0.33\textwidth}
                 \centering
                 \includegraphics[scale = 0.301]{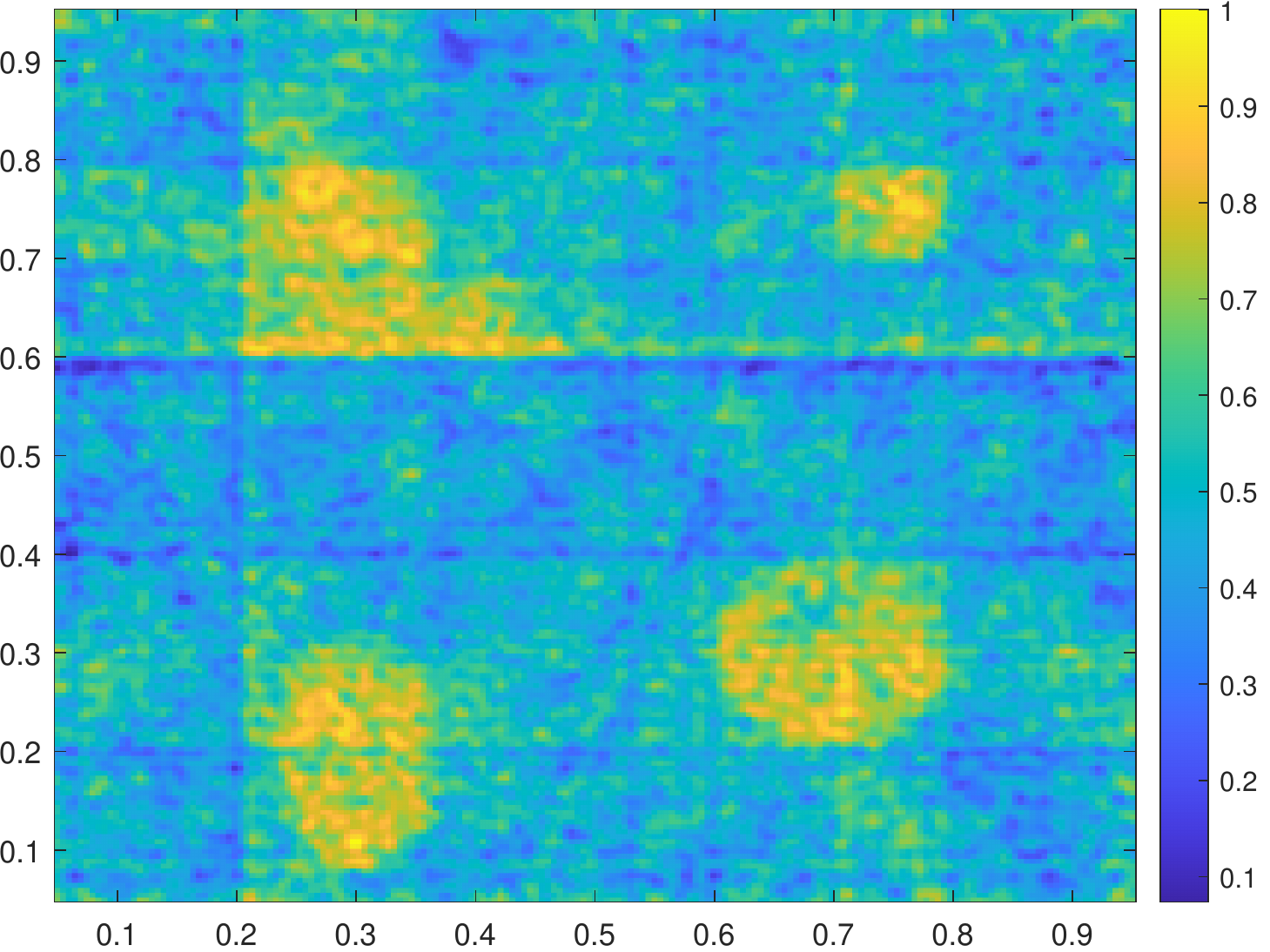}
                 \caption{DSM: $\gamma = 0.4$.}
         \end{subfigure}
   \begin{subfigure}[b]{0.33\textwidth}
                 \centering
                 \includegraphics[scale = 0.301]{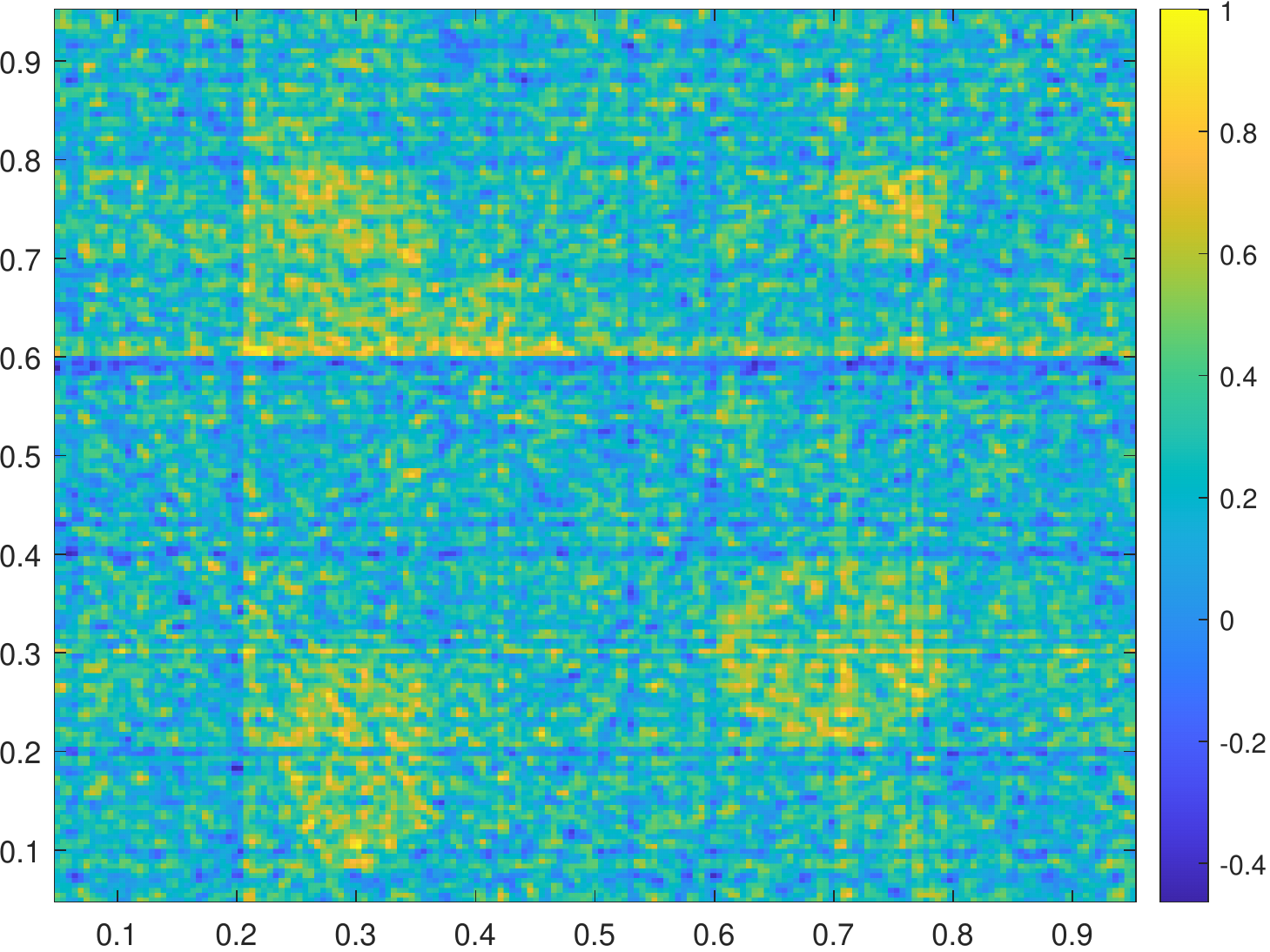}
                 \caption{FBP.}
         \end{subfigure}
   \begin{subfigure}[b]{0.33\textwidth}
                 \centering
                 \includegraphics[scale = 0.301]{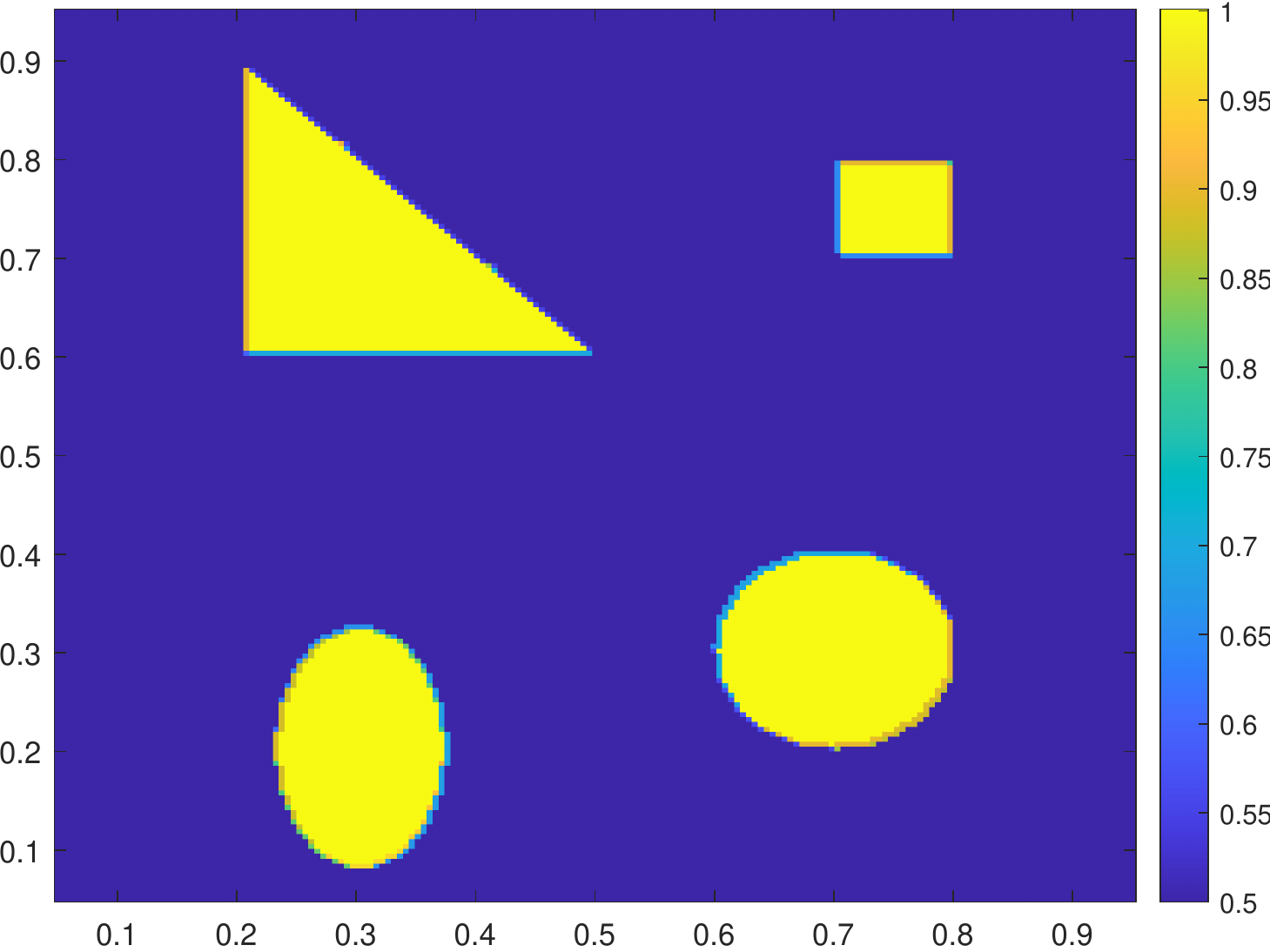}
                 \caption{$f(x)$.}
         \end{subfigure}
    \caption{Example 4. Sparse measurements with $5\%$ additive Gaussian noise}
\label{example4}
\end{figure}

\textbf{Example 5.} 
In this example, we consider the projection angles limited to a specific range as in section \ref{section_Limited_Angle}. The back projection operator needed in both the DSM and FBP reconstructions takes the form \eqref{def_dual_pat} with $\lambda = \pi/18$. The reconstructions by DSM (with $\gamma = 0.4$) and FBP are shown in Fig.\,\ref{example5}. The corresponding reconstruction errors are given respectively by
\begin{equation*}
\Err^2_{DSM} = 0.179 \,,\quad
\Err^2_{FBP} = 0.268 \,,\quad
\Err^{\infty}_{DSM} = 0.175   \,,\quad 
\Err^{\infty}_{FBP}=   0.239   \,,
\end{equation*}
for the reconstruction in the first row with $\Phi = \pi/3$ (cf. \eqref{parral_back}), and by 
\begin{equation*}
\Err^2_{DSM} = 0.217 \,,\quad
\Err^2_{FBP} = 0.348 \,,\quad
\Err^{\infty}_{DSM} = 0.211  \,,\quad 
\Err^{\infty}_{FBP}=   0.333   \,,
\end{equation*}
for the reconstruction in the second row with $\Phi = 2\pi/9$.

As we may see from the numerical reconstructions, especially from the second case where the projections are restricted only on a very narrow range with $\Phi = 2\pi/9$, we can see that the DSM performs obviously better than FBP, based on the $L^2$-norm error and the $L^{\infty}$-norm error of the reconstruction. As we can see from Fig.\,\ref{example5}(a), the shape of objects are recovered more accurately compared with FBP.

\begin{figure}
    \begin{subfigure}[b]{0.33\textwidth}
                 \centering
                 \includegraphics[scale = 0.301]{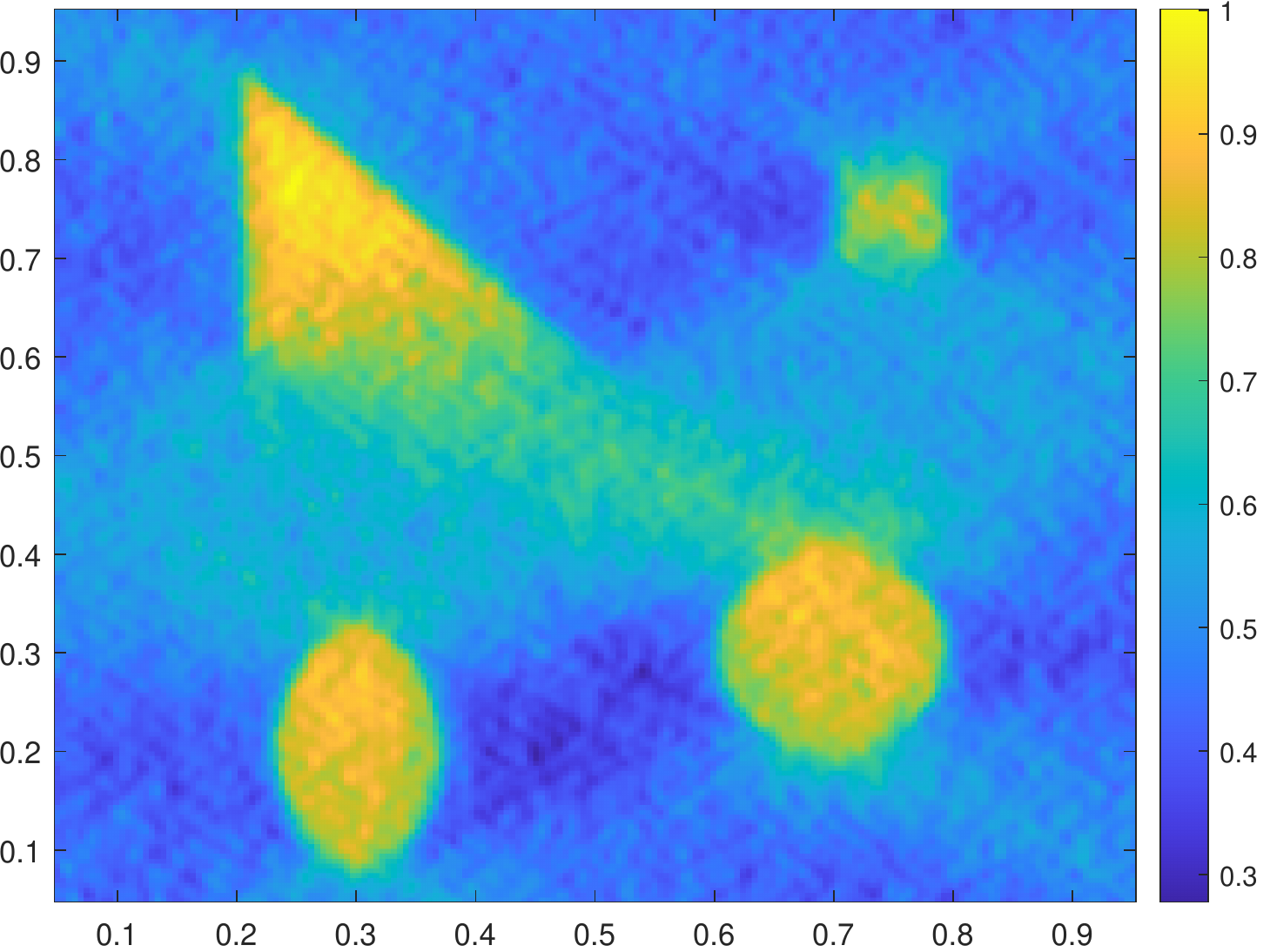}
                 \caption{DSM: $\gamma = 0.4$.}
         \end{subfigure}
   \begin{subfigure}[b]{0.33\textwidth}
                 \centering
                 \includegraphics[scale = 0.301]{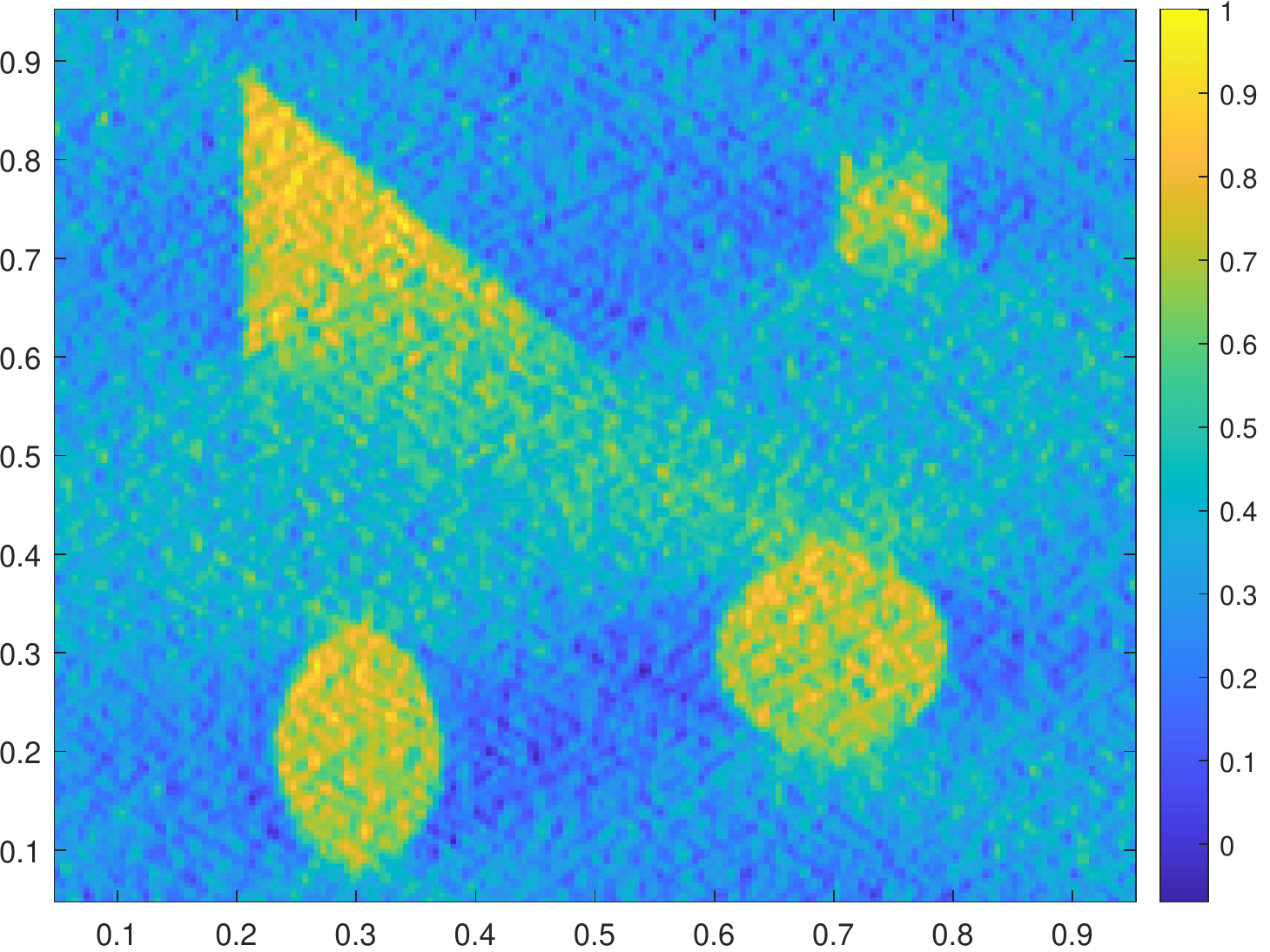}
                 \caption{FBP.}
         \end{subfigure}
   \begin{subfigure}[b]{0.33\textwidth}
                 \centering
                 \includegraphics[scale = 0.301]{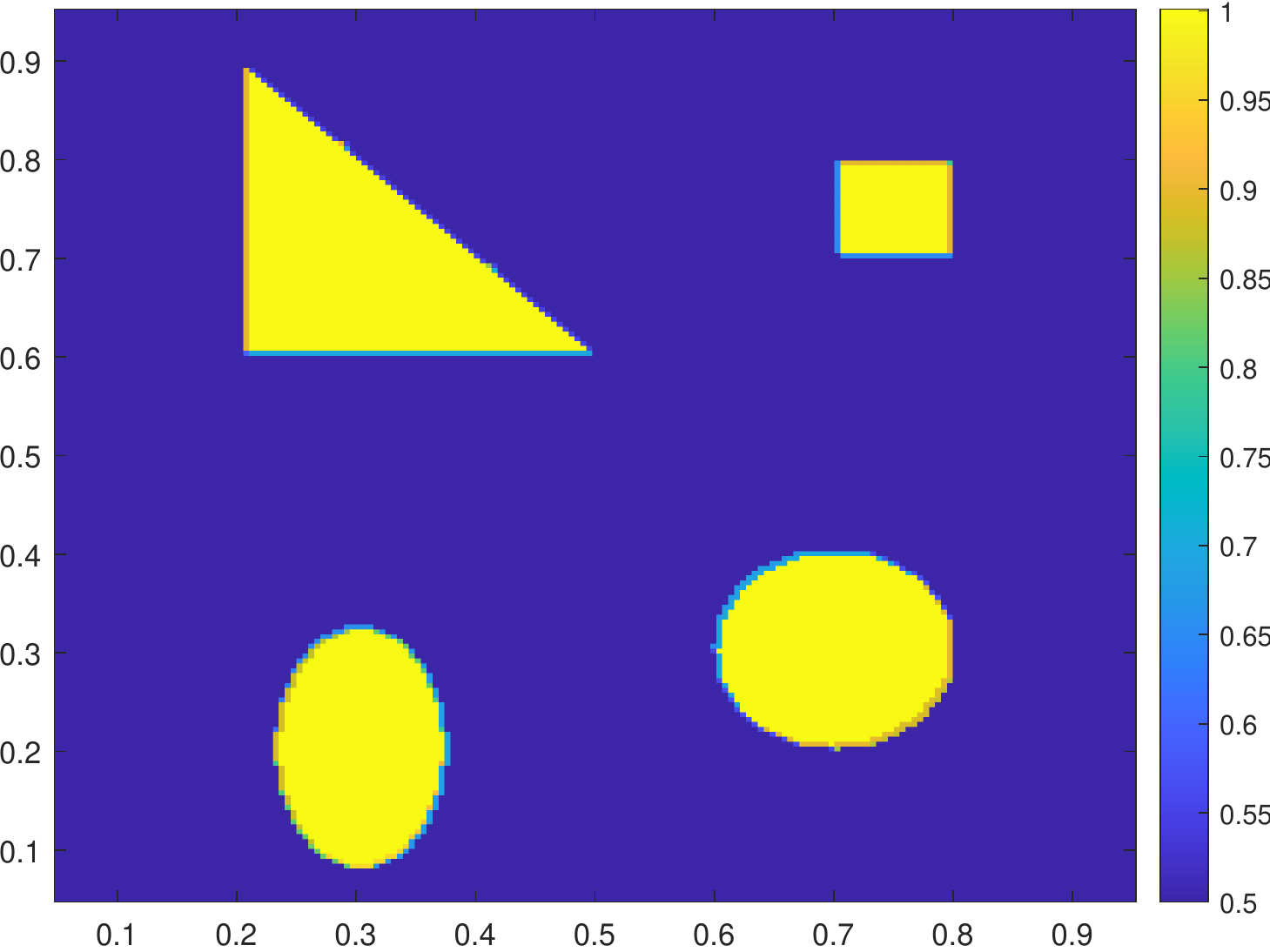}
                 \caption{$f(x)$.}
         \end{subfigure}
 \newline     
     \begin{subfigure}[b]{0.33\textwidth}
                 \centering
                 \includegraphics[scale = 0.301]{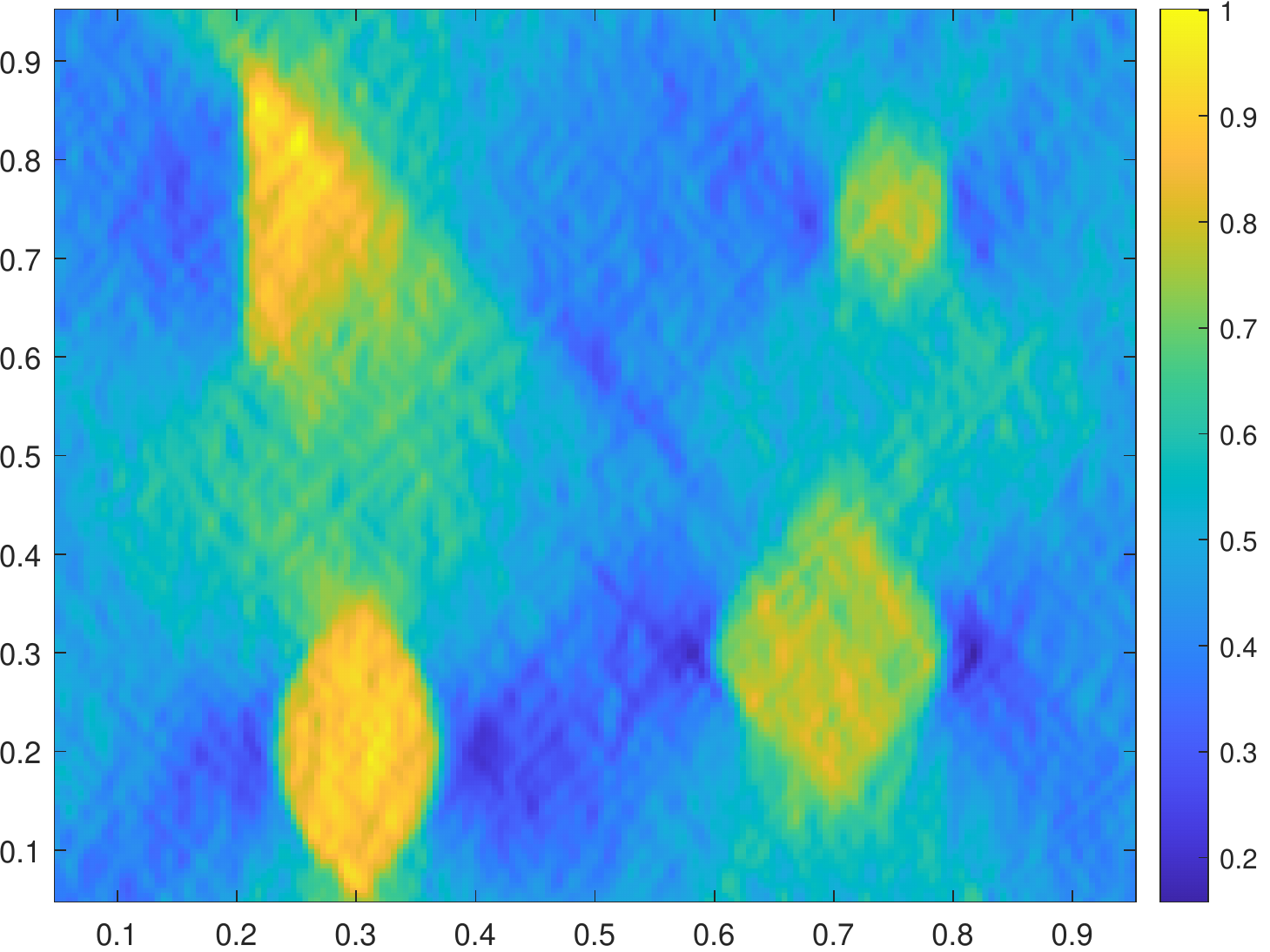}
                 \caption{DSM: $\gamma = 0.4$.}
         \end{subfigure}
   \begin{subfigure}[b]{0.33\textwidth}
                 \centering
                 \includegraphics[scale = 0.301]{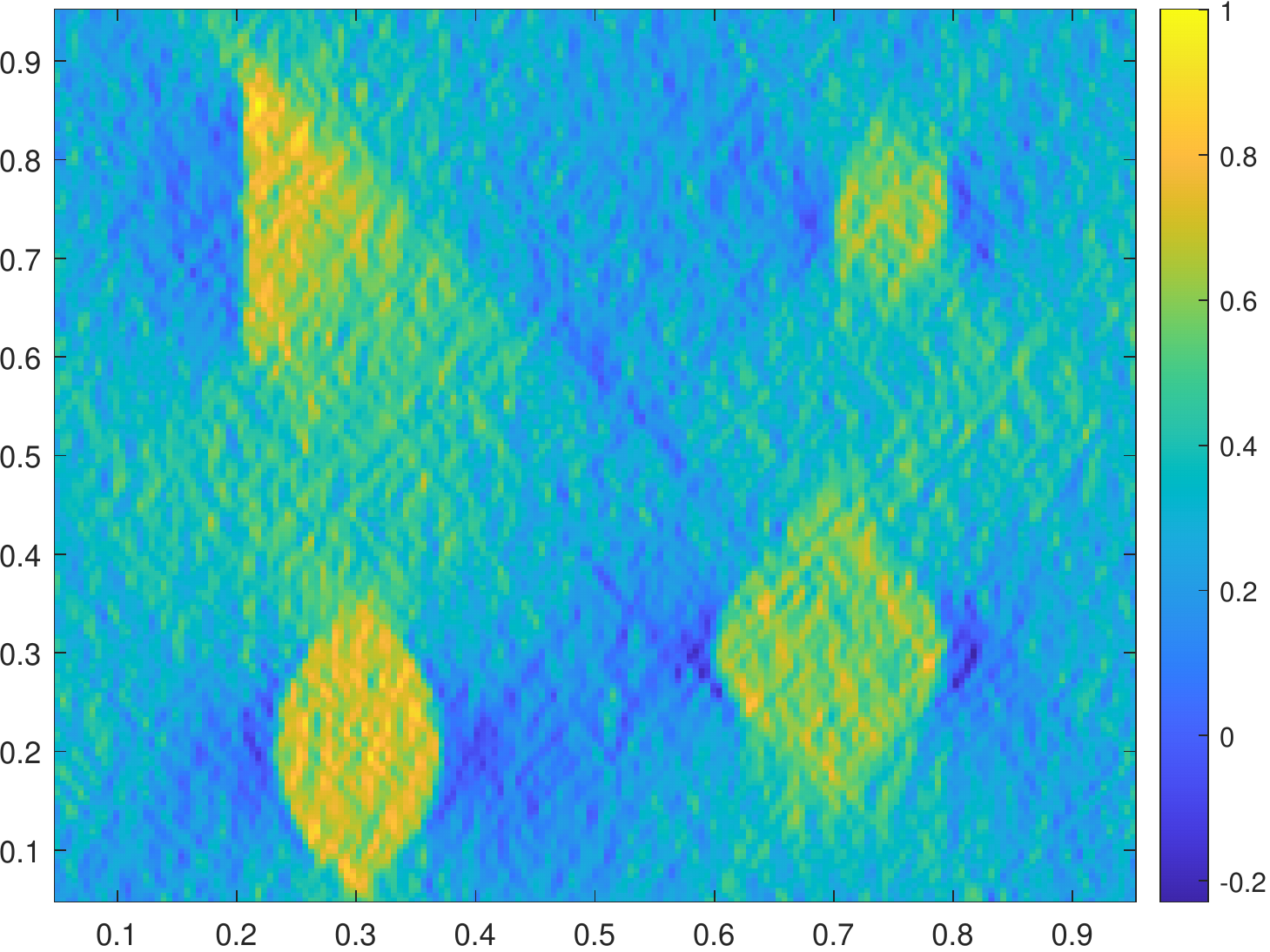}
                 \caption{FBP.}
         \end{subfigure}
   \begin{subfigure}[b]{0.33\textwidth}
                 \centering
                 \includegraphics[scale = 0.301]{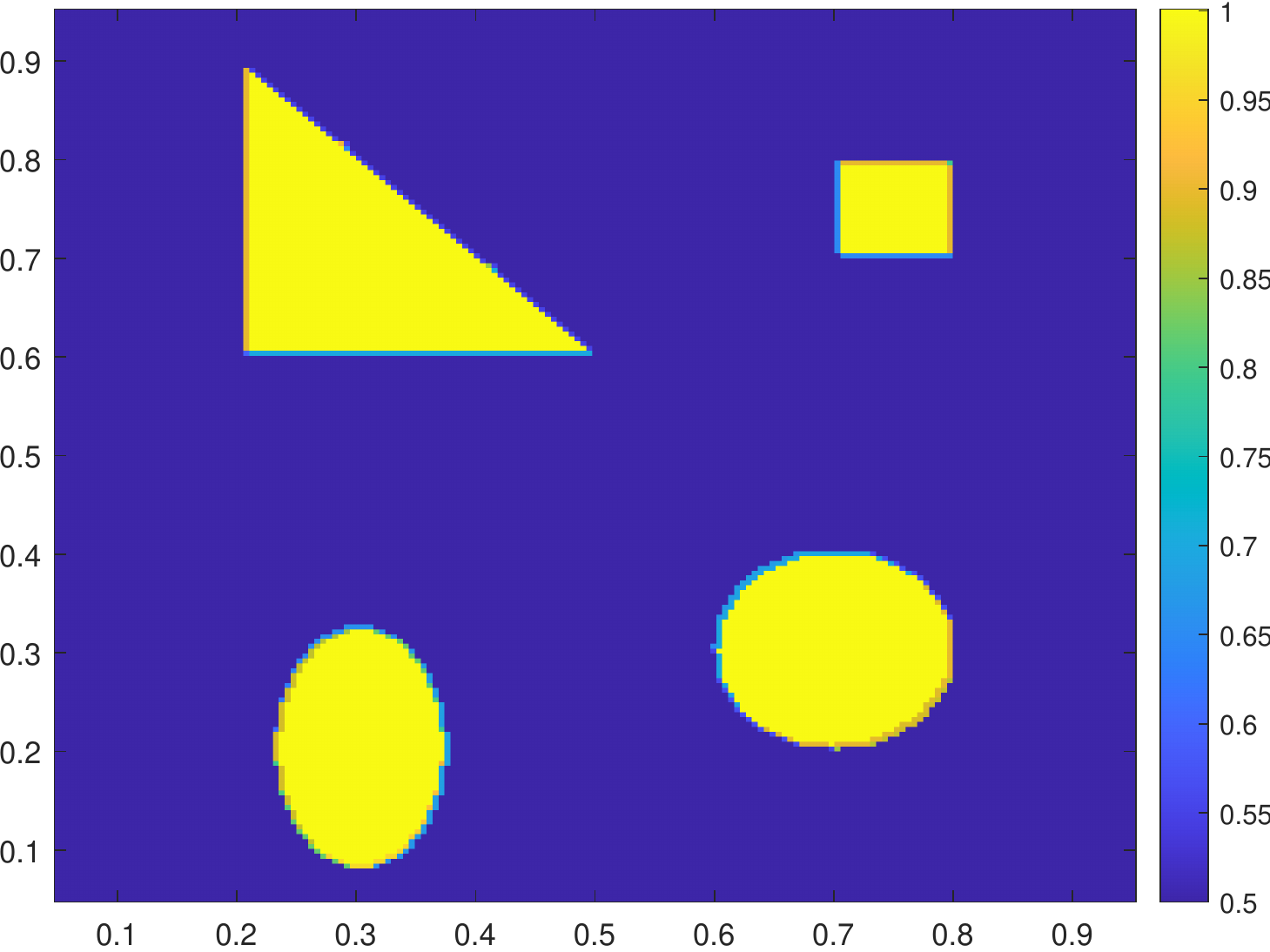}
                 \caption{$f(x)$.}
         \end{subfigure}
    \caption{Example 5. Limited angle tomography, with $10\%$ additive Gaussian noise: reconstructions with $\Phi = \pi/3$ (first row) and $\Phi = 2\pi/9$ (second row).
    }
\label{example5}
\end{figure}

\textbf{Example 6.} In this example, we consider a three-dimensional reconstruction. The reconstruction by DSM (with $\alpha = 4$, $\gamma = 0.9$) and FBP method under $1\%$ Gaussian noise are shown in Fig.\,\ref{example6}, with the mesh size $h = 10^{-2}$. For this example, the measurement data is available for 900 discrete angles $ \Gamma_{\theta} \subset \mathbb{S}^2$, and discrete measurement points $\Gamma_t(\theta)\subset I_{\theta}$ as defined in \eqref{def_Itheta}. We point out that the distribution of measurement angles in this example is relatively sparse considering the difficulty of the three-dimensional reconstruction. The three objects are one rectangular box and two balls located in $\Omega = [-0.5,0.5]^3$ as illustrated in Fig.\,\ref{example6}(a). The target function $f(x) = 0.5$ if $x$ lies in these three objects and $f(x) = 0.3$ otherwise. The corresponding reconstruction errors are given respectively by
\begin{equation*}
\Err^2_{DSM} = 0.061 \,,\quad
\Err^2_{FBP} = 0.140\,,\quad
\Err^{\infty}_{DSM} = 0.361 \,,\quad
\Err^{\infty}_{FBP} = 0.778\,,\quad
\end{equation*}
 To better illustrate reconstruction results, in Fig.\,\ref{example6}, we set $I_{DSM}(z) = 0$ if $|I_{DSM}(z)|<0.4$ and  $I_{DSM}(z) = 1$ if $|I_{DSM}(z)|\ge 0.4$ for $z\in \Omega$ to represent the support of objects reconstructed by the DSM, and we do the same for $I_{FBP}$. From Fig.\,\ref{example6}(b), we see that DSM can recover the basic shape, size, and position of the three objects quite reasonably, with three objects well separated, especially the two balls that are rather close to each other. While the reconstruction by the FBP method in Fig.\,\ref{example6}(c) generates many improper noisy points in the whole sampling domain. This example demonstrates the accuracy of DSM in reconstructing the support of objects in $\mathbb{R}^3$ with noisy measurement data.

\begin{figure}
  \begin{subfigure}[b]{0.33\textwidth}
                 \centering
         \includegraphics[scale = 0.3]{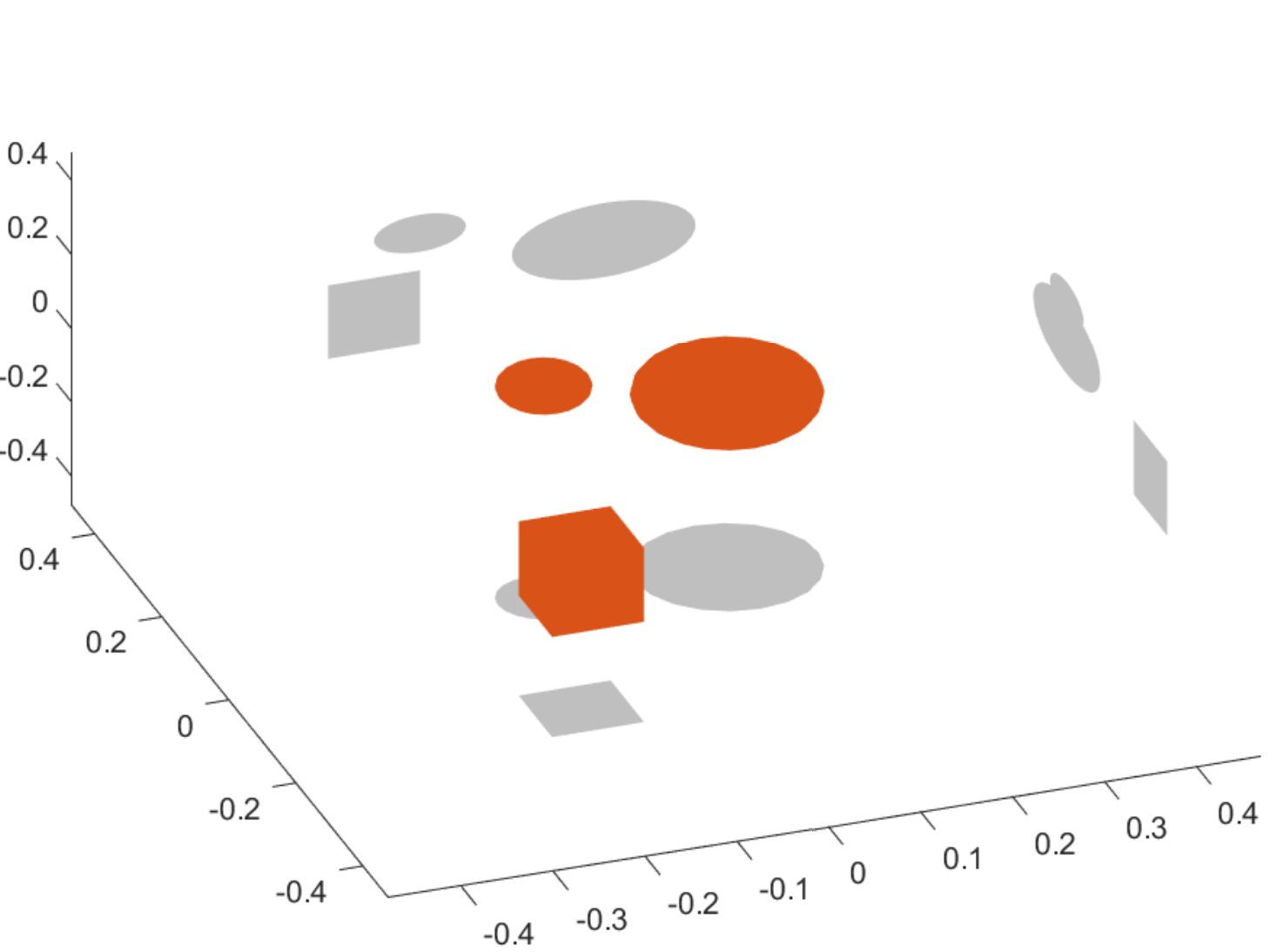}
                 \caption{$f(x)$.}  
         \end{subfigure}
   \begin{subfigure}[b]{0.33\textwidth}
                 \centering
                 \includegraphics[scale = 0.3]{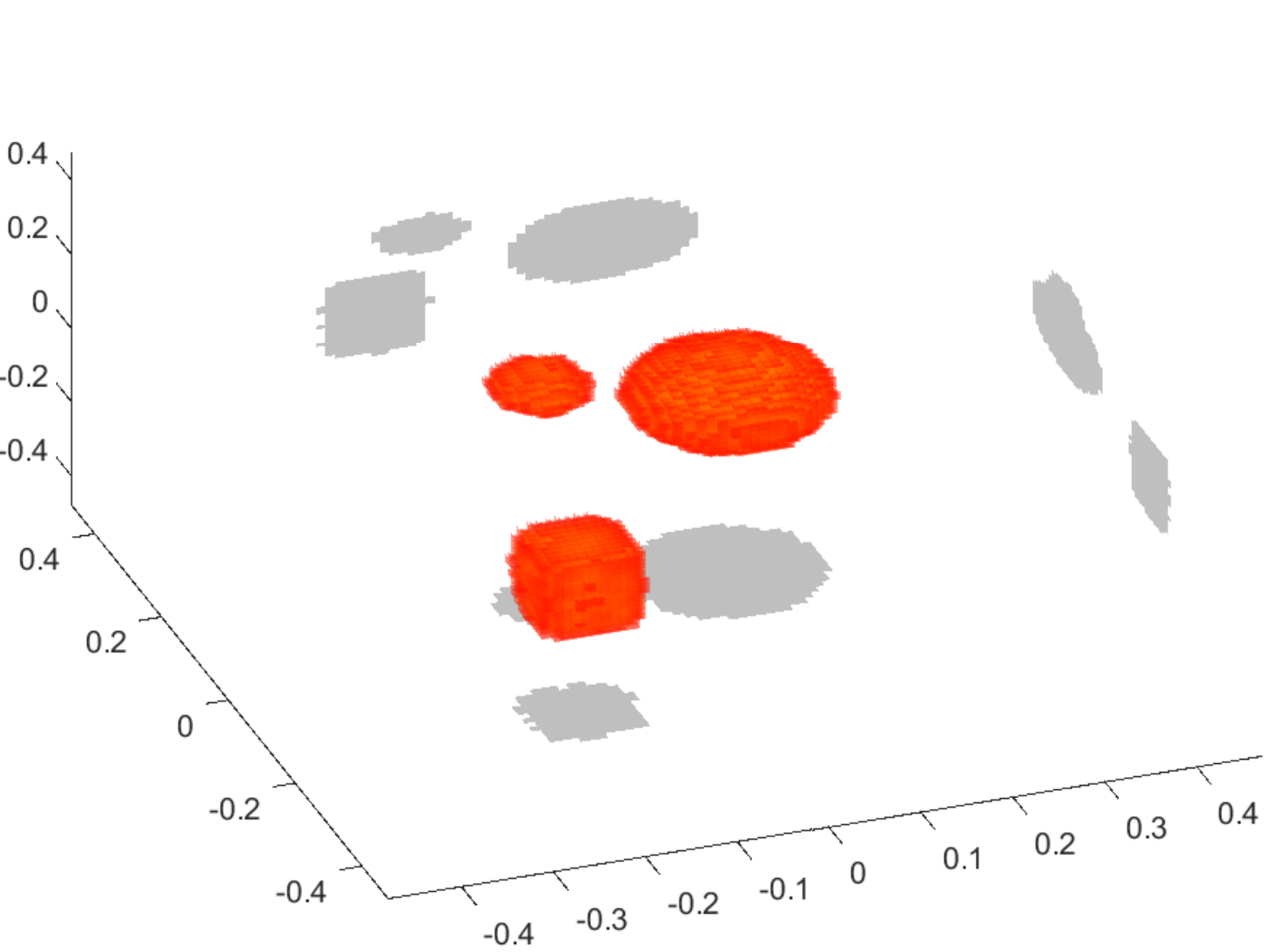}
                 \caption{DSM: $\alpha = 4$, $\gamma = 0.9$.}
    \end{subfigure}
    \begin{subfigure}[b]{0.33\textwidth}
                 \centering
                 \includegraphics[scale = 0.3]{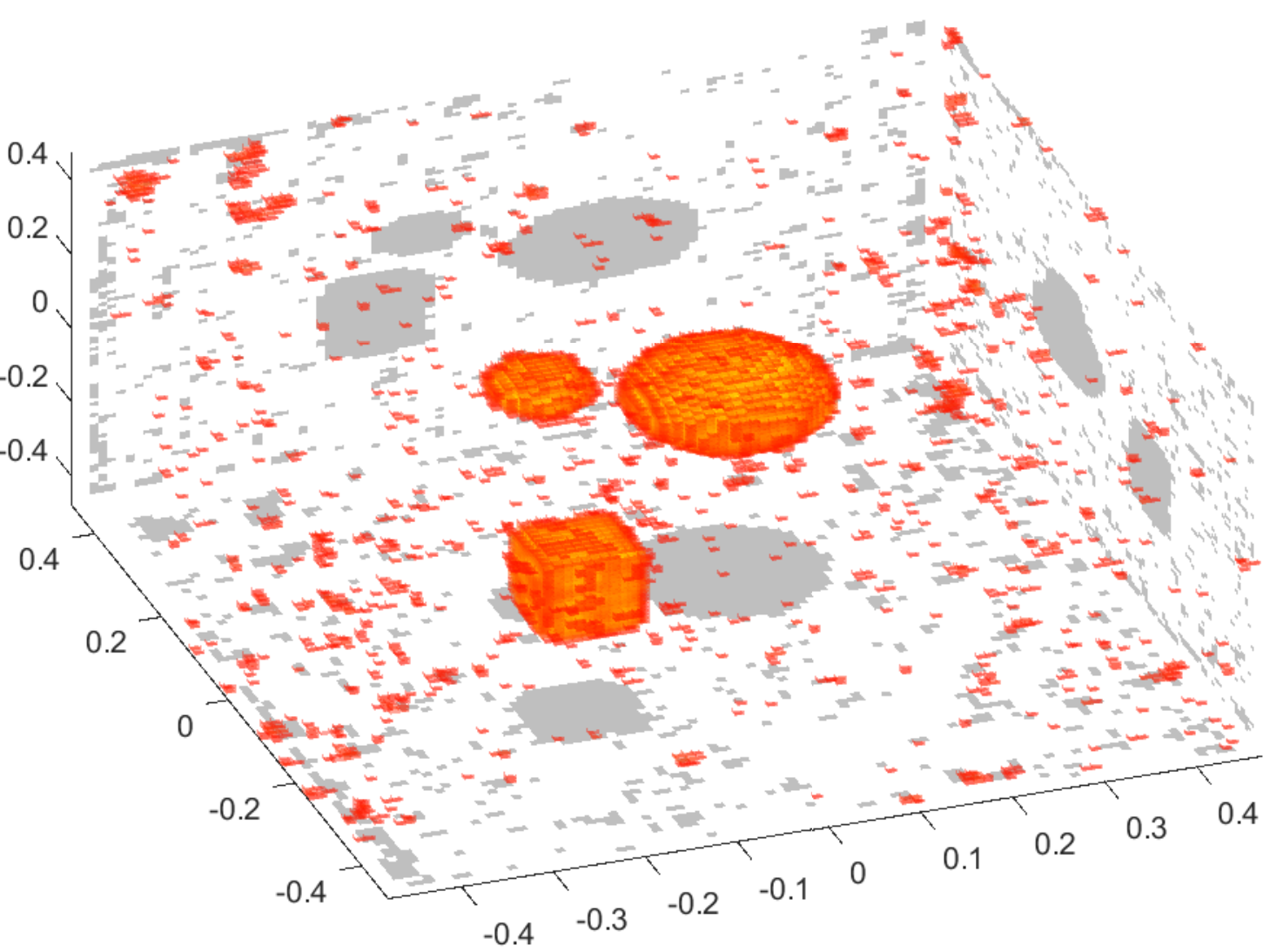}
                 \caption{FBP: Hamming filter.}
    \end{subfigure}
\caption{Example 6. Reconstruction in $\mathbb{R}^3$, with $1\%$ additive Gaussian noise.}
\label{example6}
\end{figure}

\section{Concluding remarks}
We have proposed a novel stable, fast, and parallelable direct sampling method for the inversion of the Radon transform, which is severely ill-posed when the measurement data is noisy and very limited as it appears frequently in real applications.

The DSM leverages on an important almost orthogonality property under a fractional order duality product. A family of probing functions is constructed by modifying the Green's function associated with a related fractional Laplacian. As a result of the choices of the appropriate duality product space and probing functions, the novel DSM can generate fast and satisfactory reconstruction results in challenging cases when the measurement data is highly noisy and limited. So DSM may have good potential applications in many real scenarios, such as security scanning, cancer detection, and portable CT scanner, and so on. 

Along this research direction, there are several important topics that are worth exploring in the future. For instance, a more systematic derivation and optimal choice of other effective probing functions are very interesting, which can provide more concrete guidance in practice when DSM is applied. Moreover, the validation of the DSM for the sparse tomography and the limited angle tomography are also very important due to the wide applications of these imaging techniques. From our analyses in this work, it is feasible to generalize direct sampling type methods to many other tomography problems, for instance, the general exponential Radon transform, the cone-beam computed tomography, the geodesic Radon transform, and so on. In the meantime, the generalization should preserve similar nice features to the ones of DSM in this work. 

\sectionfont{\fontsize{10}{10}\selectfont}
\begin{appendices}

\section{{Choice of the smooth extension function $\psi_{n+1}$ in \eqref{def_probing}.}}
\label{sec_appencix}
In this appendix, we shall present our choice of the smooth extension function $\psi_{n+1}$ in the definition of the auxiliary function $\zeta_{n+1}^h$ \eqref{def_zeta} which is further employed to define the crucial probing function in \eqref{def_probing}. We shall point out that the smooth extension function for other choices $\alpha$ in \eqref{def_zeta} can be constructed similarly.

 We notice that, to allow $\zeta_{n+1}^h$ possess desired properties stated in \eqref{def_zeta}, it is sufficient to require $\psi_{n+1}: [0,h]\rightarrow \mathbb{R}$ to satisfy
\begin{equation}
\label{psi_bc}	
\psi_{n+1}\in C^{2,1}\big([0,h]\big)\,;\quad
	\begin{cases}
		\psi_{n+1}(h) = h^{-n-1}\,,\\
		\psi_{n+1}'(h) = -(n+1)h^{-n-2}\,,\\
		\psi_{n+1}''(h) = (n+1)(n+2)h^{-n-3}\,;
	\end{cases}\quad
	\begin{cases}
		\psi_{n+1}(0) = h^{-n}\,,\\
		\psi_{n+1}'(0) = 0\,,\\
		\psi_{n+1}''(0) = 0\,;
	\end{cases}
\end{equation}
and for $B(0,h)\subset \mathbb{R}^n$ and $h<1$,
\begin{equation}
\label{psi_integral}
\int_{B(0,h)}|\psi_{n+1}(|x|)-h^{-n-1}|dx \leq h\,.
\end{equation}

Our choice of $\psi_{n+1}(t)$ is to construct a polynomial that matches desired boundary conditions when $t=h$ and $t=0$ in \eqref{psi_bc}, and then we restrict the support of the function $\psi_{n+1}(t) - h^{-n-1}$ to meet the requirement \eqref{psi_integral}. For simplicity, we write $k = n+1$ and $b = h-h^2/n$, then $\psi_{k}(t)$ is defined as
\begin{equation}
\label{def_psi}
	\psi_{k}(t) := \frac{1}{h^{k}}\bigg[1+\bigg(\frac{k^2+k}{2h^4}+\frac{4k}{h^5}\bigg)(t-b)^3 -\frac{1}{h^2}\bigg(\frac{k^2+k}{h^4}+\frac{7k}{h^5}\bigg)(t-b)^4 + \frac{1}{h^4}\bigg(\frac{k^2+k}{2h^4} + \frac{3k}{h^5}\bigg)(t-b)^5 \bigg]\,
\end{equation}	 
for $t\in [b,h]$, and $\psi_k(t) := 0$ for $t\in [0,b)\,$.

Therefore, the first and second order derivatives of $\psi_k(t)$ for  $t\in [b,h]$ are
\begin{align*}
	\psi_{k}'(t) = &\frac{1}{h^{k}}\bigg[3\bigg(\frac{k^2+k}{2h^4}+\frac{4k}{h^5}\bigg)(t-b)^2 -\frac{4}{h^2}\bigg(\frac{k^2+k}{h^4}+\frac{7k}{h^5}\bigg)(t-b)^3 + \frac{5}{h^4}\bigg(\frac{k^2+k}{2h^4} + \frac{3k}{h^5}\bigg)(t-b)^4 \bigg]\,,\\
	\psi_{k}''(t) = &\frac{1}{h^{k}}\bigg[6\bigg(\frac{k^2+k}{2h^4}+\frac{4k}{h^5}\bigg)(t-b) -\frac{12}{h^2}\bigg(\frac{k^2+k}{h^4}+\frac{7k}{h^5}\bigg)(t-b)^2 + \frac{20}{h^4}\bigg(\frac{k^2+k}{2h^4} + \frac{3k}{h^5}\bigg)(t-b)^3 \bigg]\,;
\end{align*}
In this case, it is straightforward to verify that $\psi_{k}(t)$ satisfies \eqref{psi_bc}.

To show that the condition \eqref{psi_integral} is satisfied by $\psi_k$, we first notice, for $t\in[b,h]$,  $\psi_{k}(t) - h^{-k}$ equals to
\begin{equation*}
\begin{split}
	\psi_{k}(t) - \frac{1}{h^{k}} = &\frac{(t-b)^3}{h^k}\bigg(\frac{t-b}{h}-1\bigg)\bigg[\bigg(\frac{k^2+k}{2h^4}+\frac{3k}{h^5}\bigg)\bigg(\frac{t-b}{h}\bigg) - \bigg(\frac{k^2+k}{2h^4}+\frac{4k}{h^5}\bigg)\bigg]\,.
\end{split}
\end{equation*}
The above shows $\psi_{k}(t) - h^{-k}>0$ for $t\in[0,h]$. We now integrate $\psi_{k}(|x|)-h^{-k}$ directly by replacing  $t-b$ by $\tau$:
\footnotesize{
\begin{equation*}
\begin{split}
	\int_{B(0,h)}\bigg|\psi_{k}(|x|) - \frac{1}{h^{k}} \bigg| dx =& |S_{n-1}|\int_{h-\frac{h^2}{n}}^h t^{n-1}\frac{(t-b)^3}{h^{k}}\bigg(\frac{t-b}{h}-1\bigg)\bigg[\bigg(\frac{k^2+k}{2h^4}+\frac{3k}{h^5}\bigg)\bigg(\frac{t-b}{h}\bigg) - \bigg(\frac{k^2+k}{2h^4}+\frac{4k}{h^5}\bigg)\bigg] dt\\
	=&\frac{|S_{n-1}|}{h^{k+2}}\int_0^{\frac{h^2}{n}}(\tau+h-\frac{h^2}{n})^{n-1}\tau^3 (\tau-h)\bigg[\bigg(\frac{k^2+k}{2h^4}+\frac{3k}{h^5}\bigg)\tau - \bigg(\frac{k^2+k}{2h^4}+\frac{4k}{h^5}\bigg)h\bigg]d\tau\,.
\end{split}
\end{equation*}}
As $h<1$, for $n=2$, we have
\begin{equation*}
\int_{B(0,h)}\bigg|\psi_{k}(|x|) - \frac{1}{h^{k}} \bigg| dx	=\frac{h\pi}{13440}(-30h^4+473h^3-294h^2-3612h+5040) <  \frac{h\pi}{13340}(473 +5040)<h\,;
\end{equation*}
and for $n=3$, we have
\begin{equation*}
\begin{split}
	\int_{B(0,h)}\bigg|\psi_{k}(|x|) - \frac{1}{h^{k}} \bigg| dx	=&  \frac{h\pi}{688905}(25h^5-810h^4+10914h^3-41076h^2+3402h+136080)\\
	\le& \frac{h\pi}{688905}(25+10914+3402+136080)<h\,.
\end{split}
\end{equation*}
We have verified that our choice of $\psi_{n+1}(t)$ in \eqref{def_psi} satisfies the requirements \eqref{psi_bc} and \eqref{psi_integral} which are a proper candidate to be employed in the numerical computation.
\end{appendices}

\bibliographystyle{siamplain}
\nocite{*}
\bibliography{bio_dsm_radon.bib}
\citation

\end{document}